\documentclass[11pt,reqno]{amsart}
\usepackage{cite}
\usepackage{geometry,graphicx}
\usepackage{amsmath,amssymb,subfig}
 \usepackage{float}
\usepackage{bbm}
\usepackage{color}
\usepackage[table,usenames,dvipsnames]{xcolor}
\usepackage[colorlinks=true,linkcolor=blue]{hyperref}
\usepackage{comment}
\usepackage{dsfont} 
\usepackage{color,soul}

\usepackage{caption}
\usepackage{tikz}
\usetikzlibrary{arrows.meta}
\usetikzlibrary{shapes.geometric, arrows}
\usetikzlibrary{fit}

\tikzstyle{minset} = [rectangle, rounded corners, minimum width=2cm, minimum height=1cm,text centered, draw=black, fill=red!30]
\tikzstyle{arrow} = [thick,->,>=stealth]
\newcommand\pd[1]{\draw (#1) node[minimum size=15pt, white, draw=blue, fill=blue!80] {{\small\bf 1}};}
\newcommand\po[1]{\draw (#1) node[minimum size=15pt, white, draw=blue, fill=purple!80] {{\small\bf 1}};}
\newcommand\p[1]{\draw (#1) node[minimum size=15pt, white, draw=black, fill=black] {{\small\bf 1}};}
\newcommand\m[1]{\draw (#1) node[minimum size=15pt, draw=black] {{\small\bf -1\!}};}

\newcommand\mo[1]{\draw (#1) node[minimum size=15pt, draw=blue, fill=purple!40] {{\small\bf -1\!}};}

\usepackage{tikz-cd}
\tikzcdset{arrow style=tikz, diagrams={>=stealth}}

\usepackage{tikz-3dplot}

\usepackage{pgfplots}
\pgfplotsset{compat=1.9}
\usepackage{pdflscape}

\usepackage{mathtools}  
\usepackage[normalem]{ulem}

\theoremstyle{plain}
\newtheorem{theorem}{Theorem}
\newtheorem{prop}[theorem]{Proposition}
\newtheorem*{prop*}{Proposition}
\newtheorem{lem}[theorem]{Lemma}
\newtheorem{coro}[theorem]{Corollary}
\newtheorem{fact}[theorem]{Fact}

\theoremstyle{definition}
\newtheorem{definition}[theorem]{Definition}
\newtheorem{remark}[theorem]{Remark}

\newtheorem{example}[theorem]{Example}

\newcommand{\exend}{\hfill $\Diamond$}

\title{Monochromatic Arithmetic Progressions in Automatic Sequences with Group Structure}

\author{Ibai Aedo}
\address{
School of Mathematics and Statistics, The Open University,\newline
\hspace*{\parindent}Walton Hall, Milton Keynes MK7 6AA, UK
}
\email{ibai.aedo@open.ac.uk}

\author{Uwe Grimm}

\author{Neil Ma\~nibo}
\address{Fakult\"at f\"ur Mathematik, Universit\"at Bielefeld, \newline
\hspace*{\parindent}Postfach 100131, 33501 Bielefeld, Germany}
\email{cmanibo@math.uni-bielefeld.de }

\author{Yasushi Nagai}
\address{School of General Education, Shinshu University,\newline
\hspace*{\parindent}3-1-1 Asahi, Matsumoto, Nagano, 390-8621, Japan}
\email{ynagai@shinshu-u.ac.jp}

\author{Petra Staynova}
\address{ School of Computing and Engineering, University of Derby,\newline
\hspace*{\parindent}Kedleston Road, Derby DE22 1GB, UK.}
\email{petra.staynova@gmail.com}

\date{}

\begin{document}


\begin{abstract}
We determine asymptotic growth rates for lengths of monochromatic arithmetic progressions in certain automatic sequences. In particular, we look at (one-sided) fixed points of aperiodic, primitive, bijective substitutions and spin substitutions, which are generalisations of the Thue--Morse and Rudin--Shapiro substitutions, respectively. For such infinite words, 
we show that there exists a subsequence $\left\{d_n\right\}$ of differences along which the maximum length $A(d_n)$ of a monochromatic arithmetic progression (with fixed difference $d_n$) grows at least polynomially in $d_n$. Explicit upper and lower bounds for the growth exponent can be derived from a finite group associated to the substitution. As an application, we obtain bounds for a van der Waerden-type number for a class of colourings parametrised by the size of the alphabet and the length of the substitution. 
\end{abstract}

\maketitle

\begin{center}
    \emph{    ``Here the Maestro laid down his pen.''\\
    This paper is dedicated to our late friend and colleague, Uwe Grimm. }
\end{center}

\section{Introduction}

The study of Ramsey-type properties of morphic words has a long history, spanning from the classic theorem of Graham and Rothschild \cite{GR} to more recent advances such as antipowers \cite{FRSZ} and monochromatic factorisations \cite{BR,LPZ} in infinite words.  
A subset of Ramsey-type properties which has also gathered interest is a consideration of the arithmetic subsequences of automatic or morphic words. 
In \cite{AF}, Avgustinovich and Frid show that any binary word occurs as an arithmetic subsequence of the Thue--Morse sequence (or more generally, a fixed point of any primitive bijective binary constant-length substitution), and go on to investigate properties of the arithmetic complexity of certain words over arbitrary finite alphabets.  
If we instead consider monochromatic arithmetic subsequences of a given substitutive word and fix the difference of the arithmetic progressions, we note that the length is bounded in the case of the Thue--Morse and more general Thue--Morse-like sequences, as shown, respectively, in~\cite{MSS} and, by the present authors, in~\cite{AGNS}; see also \cite{KS,KST} for results regarding arithmetic progressions in model sets.

While van der Waerden's theorem ensures the existence of arbitrarily long arithmetic progressions within any finite colouring of $\mathbb{N}$, it does not immediately provide an estimate for the initial segment within which these can be found. 
The van der Waerden numbers were initially introduced to study this, and are defined as the minimal initial segment of the integers such that any colouring with $n$ colours will give an arithmetic progression of length $\ell$. 
Only a handful of van der Waerden numbers are known, and Gowers \cite{Gow} gives hyper-exponential upper bounds for the rest.

Here, we pursue the line of enquiry that expands on results by Frid et al \cite{AFdFF,AF,Frid}, Parshina \cite{Olga0, Olga1,Olga2}, and the present authors \cite{AGNS, NaAkLe2021}. We focus on fixed points of constant-length substitutions over finite alphabets, and their images under codings.

Let $w$ be a fixed point of a substitution $\varrho$ over a finite alphabet $\mathcal{A}$. Fix a difference ${d\geqslant 1}$. 
We denote by $A^{ }_w(d)$ the maximum length of a monochromatic arithmetic progression of difference $d$ which occurs in $w$. 
When the context is clear, we make $w$ implicit and just refer to $A(d)$. 
We restrict to classes of substitutions which possess an explicit group structure, which provides direct access to bounding $A(d)$ for some specific values of $d$ and allows some asymptotic estimates.
This work expands the results in \cite{AGNS, Olga1, Olga2} to a more general setting, giving  upper bounds of $A(d)$ for a larger class of automatic sequences over arbitrary finite alphabets. 
The following result establishes an asymptotic lower bound for $A(d)$ along a subsequence of differences for bijective substitutions.

\begin{theorem}\label{thm:poly-growth}
Let $\varrho$ be an aperiodic primitive bijective constant-length substitution on a finite alphabet $\mathcal{A}$ and let $w\in\mathcal{A}^{\mathbb{N}}$ be a one-sided fixed point of $\varrho$. There exists an increasing sequence $\left\{d_n\right\}\subset \mathbb{N}$ such that $A(d_n)\gtrsim d^{\alpha}_{n}$, for some $0<\alpha\leqslant 1$.
\end{theorem}

We introduce the notion of \emph{$g$-palindromicity}, and show that this implies $\alpha=1$ (i.e., there is a subsequence where $A(d)$ grows linearly); see Proposition~\ref{prop:pal}. In particular, this holds for the family of cylic Thue--Morse substitutions on $L$ letters; see Section~\ref{sec: L-TM} below. We also deal with non-bijective substitutions with a \emph{supersubstitution} structure and show that they satisfy bounds similar to those for bijective substitutions. For Vandermonde substitutions, which can be constructed from Vandermonde matrices, one has the following result.

\begin{theorem}\label{thm:spin-growth}
Let $\varrho$ be a constant-length spin substitution arising from a Vandermonde matrix and let $w=\pi^{ }_{G}(v)\in\mathcal{A}^{\mathbb{N}}$ be the spin coding of a fixed point $v$ of $\varrho$. There exists an increasing sequence $\left\{d_n\right\}\subset \mathbb{N}$ such that $A(d_n)\gtrsim d_n^{\alpha}/L$, for some $0<\alpha\leqslant 1$.
\end{theorem}

Under some mild assumptions on the substitution, one can also obtain upper bounds for $A(d)$. 
In particular, sufficient conditions for $A(d)<\infty$ to hold are given in Propositions~\ref{prop:finiteness} and \ref{prop:RS,A(d)finite}. For a subclass of bijective substitutions, a $d$-dependent computable upper bounds are given in Corollary~\ref{coro:upper bound 1} and Proposotion~\ref{prop2_upper_bound}.

The bounds used to prove Theorem~\ref{thm:poly-growth} only depends on the size $c$ of the alphabet and the length $L$ of the substitution. This allows one to associate a \emph{van der Waerden-type constant} $W(\mathcal{B}(c,L),M)$ to a family of bijective substitutions sharing these same attributes; see Proposition~\ref{prop: vdW bijective upper bounds} for an general upper bound and Corollary~\ref{coro:lower bound VdW} for lower bounds. 

The paper is organised as follows.
In Section \ref{SecPrelim}, we provide some basic notions on combinatorics on words and substitutions. 
In Section \ref{SecBijective}, we focus on bijective substitutions, extending results on the lower bounds in \cite{AGNS, Olga1, Olga2} to this family. Note that this includes the class of symmetric morphisms in \cite{Frid2} and group substitutions in \cite{Goldstein}. We prove Theorem~\ref{thm:poly-growth} in Section~\ref{sec:LB-bij}. 
In Section \ref{SecvdW}, we develop the notion of a van der Waerden type-constant for bijective substitution and provide explicit upper bounds. We use recurrence properties for substitutive words and the results obtained in Section~\ref{sec:LB-bij} to compute these bounds. Section~\ref{sec:UB-bij-abel} deals with upper bounds for $A(d)$ for bijective substitutions with additional properties. 
Poignant examples, including the family of cyclic Thue--Morse substitutions, are given in Section~\ref{sec:Ex-bij}. We show that some of the results can be partially extended to the non-bijective case in Section~\ref{SecNonBij} for substitutions admitting a supersubstitution structure.  
In Section~\ref{sec:spin}, we deal with spin substitutions. We begin with the case of the Rudin--Shapiro substitution in Sections~\ref{sec:Rudin-Shapiro-1} and \ref{sec:Rudin-Shapiro-2} where we derive lower bounds using two combinatorial approaches, namely via the spin matrix and via the stagerred substitution approach. At this point, we would like to mention that the same bounds have been found for the Rudin--Shapiro sequence in \cite{So2022} using a different method. The novelty of the approach we use in this work is that it extends to other automatic sequences derived from other spin matrices. We carry this out in Section~\ref{sec:Vandermonde} in the case where the relevant matrix is an $L\times L$-Vandermonde or discrete Fourier transform (DFT) matrix, where we prove Theorem~\ref{thm:spin-growth}. Finally, in Section~\ref{SecOutlook}, we end with some open questions and illustrate potential ways of extending our results to the non-bijective constant length case through a concrete example.

\section{Preliminaries}\label{SecPrelim}

\subsection{Combinatorics on words and substitutions}

Throughout this work, an \emph{alphabet}  $\mathcal{A}$ will be a finite collection of symbols called \emph{letters}. We denote by $\mathcal{A}^*$ and $\mathcal{A}^+$ the sets of all \emph{finite words} and \emph{non-empty finite words} over $\mathcal{A}$, respectively. The sets of all \emph{one-sided} and \emph{two-sided infinite words} over $\mathcal{A}$ are denoted by $\mathcal{A}^{\mathbb{N}}$ and $\mathcal{A}^{\mathbb{Z}}$, respectively. These are also called \emph{sequences} over $\mathcal{A}$.
Here $\mathbb{Z}$ is the set of integers and $\mathbb{N}=\mathbb{Z}_{\geqslant 0}$.

The \emph{length} of a word $w\in \mathcal{A}^{+}$ is denoted by $|w|$.
For each $0\leqslant i<|w|$, we write $w_i$ to denote the $i$th letter of $w$.
A \emph{subword} of $w$ is a word of the form $w_{i}w_{i+1}\cdots w_{j}$, for some $0\leqslant i\leqslant j < |w|$.
A  \emph{substitution} $\varrho$ on $\mathcal{A}$ is a map $\varrho\colon \mathcal{A}\to\mathcal{A}^+$, which extends to a map on $\mathcal{A}^{+}$ by concatenation. 
This allows one to define $\varrho^{n}$ inductively via $\varrho^{n}(a):= \varrho^{n-1}(\varrho(a))$ for $n\geqslant 2$. We call $\varrho^{n}(a)$ a \emph{level-$n$ superword} of type $a$.  
We say that $\varrho$ is \emph{primitive} if there exists an $n\in\mathbb{N}$ such that, for every $a\in\mathcal{A}$, the word $\varrho^n(a)$ contains all the letters in $\mathcal{A}$. 
A word $w\in \mathcal{A}^{+}$ is \emph{legal} with respect to $\varrho$ if there exists $n\in\mathbb{N}$ and a letter $a$ such that $w$ is a subword of $\varrho^{n}(a)$. 
We denote by $\mathcal{L}_n(\varrho)$ the set of all legal words of length $n$. 
The set 
$\mathcal{L}(\varrho):= \bigcup_{n\geqslant 1} \mathcal{L}_n(\varrho)$ of all legal words is called the \emph{language} of $\varrho$. 

If there exists an $L\in\mathbb{N}$ such that, for all $a\in\mathcal{A}$, the word $\varrho(a)$ has length $L$, we say that $\varrho$ is a substitution of \emph{constant length} $L$. Such maps are also called \emph{uniform morphisms}.
For each $0 \leqslant i \leqslant L-1$ and $a\in\mathcal{A}$, we define $\varrho^{ }_i(a)$ to be the $i$th letter of $\varrho(a)$. We call the map $\varrho^{ }_i\colon \mathcal{A}\to\mathcal{A}$ the $i$th \emph{column} of $\varrho$. If $\varrho^{ }_i$ is a bijection of $\mathcal{A}$, we call it a \emph{bijective} column. If $\varrho^{ }_i(a)=b$ for some fixed $b\in\mathcal{A}$ for all $a\in\mathcal{A}$, we call $\varrho^{ }_i$ a \emph{coincidence}. If $1<|\varrho^{ }_i(\mathcal{A})|<|\mathcal{A}|$, we call $\varrho^{ }_i$ a \emph{partial coincidence}. 

One can find the columns of a power $\varrho^n$ of $\varrho$ via the following well-known result.

\begin{fact}\label{fact:i-th-column-power}
Let $n\in\mathbb{N}$ and let $\varrho$ be a substitution of constant length $L$.  Let $0\leqslant k\leqslant L^{n}-1$. Then the $k$th column of the substitution $\varrho^n$ is given by the functional composition
\[
\left(\varrho^n\right)_k = \varrho^{ }_{k_0}\circ\varrho^{ }_{k_1}\circ\dotsb\circ\varrho^{ }_{k_{n-1}},
\]
where $[k_{n-1},\dotsc,k_1,k_0]$ is the base-$L$ expansion of $k$. \qed
\end{fact}

To continue, let $\mathcal{C}$ be another finite alphabet. A letter-to-letter map $\tau\colon \mathcal{A}\to \mathcal{C}$ is called a \emph{coding}, which extends to a map $\tau\colon \mathcal{A}^{\mathbb{N}}\to  \mathcal{C}^{\mathbb{N}}$.
We say that $w\in\mathcal{A}^{\mathbb{N}}$ is a \emph{fixed point} of $\varrho$ if $\varrho(w)=w$. 
A fixed point $w$ of $\varrho$ is called \emph{aperiodic} if there does not exist a finite word $v$ such that $w=v^{\infty}$, i.e., $w$ is not a concatenation of infinite copies of $v$. We call a substitution $\varrho$ aperiodic if it does not admit any periodic fixed point.

\begin{remark}
Cobham's little theorem then implies that these infinite words are actually $L$-automatic (where $L$ is the length of $\varrho$). An automatic sequence is one which can be retrieved as an output of a deterministic finite state automaton with output (DFAO). We do not define what a DFAO here and refer the reader to \cite{AS} instead.
\exend
\end{remark}

\subsection{Monochromatic arithmetic progressions in substitution fixed points}

Consider $w\in\mathcal{A}^{\mathbb{N}}$. One can view $w$ as an $|\mathcal{A}|$-colouring of $\mathbb{N}$, where the colours are in one-to-one correspondence with the elements of $\mathcal{A}$. Fix $d,M\geqslant 1$.
We say that $w$ contains a \emph{monochromatic arithmetic progression} of difference $d$ and length $M$ if there exists a starting position $k\in\mathbb{N}$ such that $w_{k}=w_{k+dn}$, for $0\leqslant n\leqslant M-1$. The monochromatic arithmetic progression is infinite if $w_{k}=w_{k+dn}$, for all $n\in\mathbb{N}$. 
Monochromatic arithmetic progressions in two-sided infinite words are similarly defined in the obvious way. 

\begin{definition}
Let $w\in\mathcal{A}^{\mathbb{N}}$ and $d\in \mathbb{N}$. We denote by $A^{ }_w(d)$ the maximum length of a monochromatic arithmetic progression of difference $d$ that can be found in $w$.    
\end{definition}

Next, we define the \emph{height} of the substitution $\varrho$. Note that this definition does not depend on the fixed point $w$; see \cite{Dekking}. Let $\varrho$ be an aperiodic, primitive, constant-length substitution. Let $w$ be a one-sided fixed point of (possibly some power of) $\varrho$. The height $h(\varrho)$ is given by 
\[
h(\varrho):=\max\left\{n\geqslant\colon \gcd(n,L)=1, n\text{ divides }\gcd\left\{w_0=w_a\right\}\right\}
\]

We have the following sufficient condition for the finiteness of $A(d)$, for all $d\in \mathbb{N}$, for a fixed point $w$ in terms of the columns of $\varrho$; compare \cite[Prop.~8]{AGNS}. 

\begin{prop}\label{prop:finiteness}
Let $\varrho$ be an aperiodic, primitive, constant-length substitution with height $1$. Let $w$ be a fixed point of any power of $\varrho$. Then $A(d)<\infty$ for all $d \geqslant 1$ if and only if $\varrho$ does not have a coincidence column.  \qed
\end{prop}

This finiteness result carries over to codings of certain substitution fixed points; see Section~\ref{sec:spin} below for the treatment of the Rudin--Shapiro substitution and its generalisations.  

\subsection{Notation}
Here, we recall some standard notation concerning asymptotics of non-negative functions; compare \cite[Ch.~1]{LaRo2004}. For  functions $f,g\colon \mathbb{N}\to \mathbb{R}_{>0}$, we write
\begin{itemize} 
    \item $f(n)\sim g(n)$, if $\lim_{n\to\infty}{|f(n)/g(n)|}=1$,
    \item $f(n)\gtrsim g(n)$, if there exists $h(n)\colon\mathbb{N}\to\mathbb{R}_{>0}$ with $f(n)\geqslant h(n)$ and $h(n)\sim g(n)$.
\end{itemize}

\section{Bijective automata} \label{SecBijective}

A substitution $\varrho$ of constant length $L$ is called \emph{bijective} if every column $\varrho^{ }_i$ of $\varrho$ is a bijection. We denote by $G=G^{(1)}=\left\langle\varrho^{ }_{i}\right\rangle_{0\leqslant i\leqslant L-1}$ the group generated by the columns of $\varrho$, seen as a subgroup of the symmetric group $S_{|\mathcal{A}|}$. Throughout this section, $\varrho$ will be a length-$L$ substitution satisfying the following assumptions, which we denote by $(\boldsymbol{\ast})$ for brevity,
\begin{equation}\tag{$\boldsymbol{\ast}$}
    \text{aperiodic, primitive, bijective, } \varrho^{ }_0=\text{id}. 
\end{equation}
Note that the condition of the zeroth column being the identity is natural for bijective substitutions and can be achieved by taking a suitable power.  
We refer the reader to \cite[Ch.~9]{Queffelec} for a comprehensive treatment of bijective substitutions; see also \cite{Frank,KY}.

Let $w$ be a fixed point of (a power of) a constant-length substitution $\varrho$ which satisfies $(\boldsymbol{\ast})$. The following proposition is a version of 
Proposition~\ref{prop:finiteness} for bijective substitutions that does not need the height-$1$ condition.

\begin{prop}\label{prop2:finiteness}
    Let $\varrho$ be an aperiodic, primitive and bijective substitution.
    Any fixed point $w$ of a power of $\varrho$ satisfies $A(d)<\infty$, for all $d\geqslant 1$.
\end{prop}
\begin{proof}
If $w$ is not periodic, then
by Lee-Moody-Solomyak's overlap algorithm\cite[Theorem 4.7, Lemma A.9]{Lee-Moody-Solomyak},
the corresponding self-similar tiling
is not pure point and
does not admit infinite arithmetic 
progressions by \cite[Theorem 5.1]{NaAkLe2021}.
\end{proof}

In the next section, we provide lower bounds for $A(d)$ for specific values of $d$ and prove Theorem~\ref{thm:poly-growth}.

\subsection{Lower bounds and polynomial growth of \texorpdfstring{$A(d)$}{text}}\label{sec:LB-bij}

The following result shows that, for the fixed points of a primitive bijective substitution, we can find a sequence of differences $d$ for which $A(d)$ grows polynomially in $d$.

\begin{prop}\label{prop:genbijective}
Let $\varrho$ be a length-$L$ substitution satisfying $(\boldsymbol{\ast})$, and let $G$ be the group generated by the columns of $\varrho$. Then any fixed point of $\varrho$ satisfies, for every $k\geqslant 1$,
\[
L^{k}\leqslant A\left(\frac{L^{k|G|}-1}{L^k-1}\right)< \infty.
\]
\end{prop}

\begin{proof}
Let $k=1$. Consider the substitution $\varrho^{|G|}$, which is of length $L^{|G|}$. Let $i^{ }_0,i^{ }_1,\dotsc,i^{ }_{L-1}$ be the arithmetic progression of difference $d^{(1)}=\sum^{|G|-1}_{j=0}L^{j}$ with $i^{ }_0=0$. For every $0\leqslant m\leqslant L-1$, $i^{ }_m=\sum^{|G|-1}_{j=0}mL^{j}$, which has base-$L$ expansion $[m,m,\ldots,m]$. Then, for every $0\leqslant m\leqslant L-1$, the $i^{ }_m$th column of $\varrho^{|G|}$ is equal to the identity. Indeed, $\left(\varrho^{|G|}\right)_{i^{ }_m}=\varrho^{ }_m\circ\dotsb\circ\varrho^{ }_m=\left(\varrho^{ }_m\right)^{|G|}=\text{id}$, where the first equality holds by Fact~\ref{fact:i-th-column-power}, and the last equality holds because $g^{|G|}=\text{id}$, for every group element $g\in G$. Since $\varrho^{|G|}$ has $L$ columns equal to the identity substitution distributed in arithmetic progression of difference $d$, any fixed point of $\varrho$ has a monochromatic arithmetic progression of difference $d$ and length at least $L$. This completes the proof for $k=1$. This proof extends to every positive integer $k$ because, since $\varrho$ has a column which is equal to the identity substitution (the leftmost column), the group generated by the columns of $\varrho$ is equal to the group generated by the columns of $\varrho^k$~\cite{BuLuMa2021}. This means for a fixed $k\geqslant 1$, one can take $\varrho^{k|G|}$ and construct $d^{(k)}$, this time with $0\leqslant m\leqslant L^{k}-1$. The finiteness of $A(d)$ follows from Proposition~\ref{prop2:finiteness}.
\end{proof}

\begin{remark}\label{remark:genbijective}
Note that one can replace $|G|$ with $\text{lcm} \left\{ \text{ord}(g)\colon g\in G \right\}$ in Proposition~\ref{prop:genbijective} and get the same lower bound. \exend
\end{remark}

The following is immediate from Proposition~\ref{prop:genbijective}.

\begin{coro}\label{coro:genbijective1}
For all $d=\frac{L^{k|G|}-1}{L^k-1}$ with $k\geqslant 1$, $A(d)\gtrsim d^{\alpha}$, where $\alpha=(|G|-1)^{-1}$.
\end{coro}

Theorem~\ref{thm:poly-growth} follows directly from Corollary~\ref{coro:genbijective1}.
Note that when $|G|=2$, then there exists a subsequence of distances for which $A(d)$ grows linearly in $d$. This is exactly the subfamily treated in \cite{AGNS}. Below, we provide another sufficient condition for a bijective substitution (now on a possibly larger alphabet) to admit an infinite subsequence of differences along which $A(d)$ grows linearly in $d$.
We begin with the following definition.

\begin{definition}
Let $\varrho$ be a length-$L$ substitution satisfying $(\boldsymbol{\ast})$, and let $G$ be the group generated by the columns of $\varrho$. If there exists $g\in G$ such that $\varrho^{ }_i\cdot \varrho^{ }_{L-1-i}=g$, for all $0\leqslant i\leqslant L-1$, we say that $\varrho$ is  \emph{$g$-palindromic}. If $g=\text{id}$, we say that $\varrho$ is \emph{inverse palindromic}. 
\end{definition}

\begin{prop}\label{prop:pal}
Let $\varrho$ be a length-$L$ substitution satisfying $(\boldsymbol{\ast})$. Suppose further that
\begin{enumerate}
\item the group $G$ generated by the columns of $\varrho$ is Abelian,
\item $\varrho$ is $g$-palindromic, for $g\in G$. 
\end{enumerate}
Then any fixed point of $\varrho$ satisfies, for every $n\geqslant 1$ and even $\ell\geqslant 2$,
\[
L^{n}\leqslant A\left(\frac{L^{n\ell}-1}{L^n+1}\right)<\infty.
\]
\end{prop}

\begin{proof}
To prove the lower bound for $n=1$ and an even $\ell\geqslant 2$, we consider the substitution $\varrho^{\ell}$, which has length $L^{\ell}$. Let $i^{ }_1,i^{ }_2,\dotsc,i^{ }_{L}$ be an arithmetic progression of difference $d={(L^{\ell}-1)/(L+1)}={\sum^{\ell-1}_{j=0}(-1)^{j+1}L^{j}}$, where $i^{ }_k=k\,d$, for each $1\leqslant k \leqslant L$.
Using the identity
\[
k\sum_{j=0}^{\ell-1}(-1)^{j+1}L^{j}\,=\,(k-1)\sum_{j=0}^{\frac{\ell-2}{2}}L^{2j+1}\,+\,(L-k)\sum_{j=0}^{\frac{\ell-2}{2}}L^{2j},
\]
we see that the base-$L$ representation of $i^{ }_k$ is $[k-1,L-k,\dotsc,k-1,L-k]$, with all the even digits equal to $L-k$, and all the odd digits equal to $k-1$. Then, the $i^{ }_k$th column of $\varrho^{\ell}$ is given by
\[
\begin{matrix*}
(\varrho^{\ell})^{ }_{i^{ }_k}&=&
\hphantom{_{-(L-L)}}
\varrho^{ }_{L-k}\circ\varrho^{ }_{k-1}
\circ\cdots\circ\,
\varrho^{ }_{L-k}\circ\varrho^{ }_{k-1}
\hphantom{_{-(L-L)}}
&=& \hphantom{g^{\ell/2},}\\
&=&\underbracket[0.5pt]{\varrho^{ }_{L-k}\circ\varrho^{ }_{L-1-(L-k)}}_{=g}
\circ\cdots\circ
\underbracket[0.5pt]{\varrho^{ }_{L-k}\circ\varrho^{ }_{L-1-(L-k)}}_{=g}
&=& g^{\frac{\ell}{2}},
\end{matrix*}
\]
where the first equality holds by Fact~\ref{fact:i-th-column-power}, and the last equality holds because $\varrho$ is $g$-palindromic. This implies that $A(d)\geqslant L$, as required.

To prove the claim for an integer $n \geqslant 2$, it suffices to show that $\varrho^{n}$ is $g^{n}$-palindromic when $\varrho$ is ${g\text{-palindromic}}$. Notice that $\varrho^{n}$ has length $L^{n}$. Let $[i^{ }_{n-1},\ldots,i^{ }_1,i^{ }_0]$ be the base-$L$ representation of an integer $0 \leqslant i \leqslant L^n-1$. It is easy to check that the base-$L$ representation of $L^n-1-i$ is ${[L-1-i^{ }_{n-1},\dotsc,L-1-i^{ }_1,L-1-i^{ }_0]}$. Then,
\[
\begin{matrix*}[l]
(\varrho^n)^{ }_i\circ(\varrho^n)^{ }_{L^n-1-i}
&=&
\big(\varrho^{ }_{i^{ }_0}\circ\dotsb\circ\varrho^{ }_{i^{ }_{n-1}}\big)
\circ
\big(\varrho^{ }_{L-1-i^{ }_0}\circ\dotsb\circ\varrho^{ }_{L-1-i^{ }_{n-1}}\big)
&=& \hphantom{g^{n},}\\
&=&
\,\,\,\,\,\,
\underbracket[0.5pt]{\varrho^{ }_{i^{ }_0}\circ\varrho^{ }_{L-1-i^{ }_0}}_{=g}
\circ\dotsb\circ
\underbracket[0.5pt]{\varrho^{ }_{i^{ }_{n-1}}\circ\varrho^{ }_{L-1-i^{ }_{n-1}}}_{=g}
&=&g^{n},
\end{matrix*}
\]
where the first equality holds by Fact~\ref{fact:i-th-column-power}, the second equality holds because $G$ is Abelian, and the last equality holds because $\varrho$ is $g$-palindromic. So $\varrho^n$ is $g^{n}$-palindromic. Similar to the $n=1$ case, this implies that, for all integers $n\geqslant 2$, $A(d)\geqslant L^n$, as required.
Finally, the finiteness of $A(d)$ follows by Proposition~\ref{prop2:finiteness}, for every positive integer $n$.
\end{proof}

Notice that if we pick $\ell=2|G|$ in Proposition~\ref{prop:pal}, we get an analogue of Proposition~\ref{prop:genbijective} for another family of differences.

\begin{remark}\label{rem:pal}
We observe also that, for inverse palindromic substitutions, the monochromatic arithmetic progression found in Proposition~\ref{prop:pal} can be extended by two. Indeed, from the base-$L$ representations of $0 d$ and $(L+1) d$, 
it is easy to see that $(\varrho^{\ell})^{ }_{0 d}=\text{id}$ and $(\varrho^{\ell})^{ }_{(L+1)d}=\varrho_{L-1}^{\ell}$. Since, for inverse palindromic substitutions $g=\varrho^{ }_{L-1}=\text{id}$, this implies that $A(d)\geqslant L^n+2$. 
\exend
\end{remark}

The following is immediate from Proposition~\ref{prop:pal}.

\begin{coro}\label{coro:pal1}
For all $d=\frac{L^{n\ell}-1}{L^n+1}$ with $n\geqslant 1$ and even $\ell\geqslant 2$, $A(d)\gtrsim d^{\alpha}$, where $\alpha=(\ell-1)^{-1}$. In particular, $A(d)\gtrsim d$ for differences $d=L^{n}-1$.
\end{coro}

\subsection{Van der Waerden-type numbers}\label{SecvdW}

Van der Waerden's theorem~\cite{vdW} states that, for every $c,M\geqslant 1$, there exists an $n\geqslant 1$ such that any colouring of $\{0,1,\dotsc,n-1\}$ with $c$ many colours contains a monochromatic arithmetic progression of length $M$. 
The smallest threshold of $n$, for given values of $c$ and $M$, is the \emph{van der Waerden number} $W(c,M)$. 
In this subsection, we define van der Waerden-type numbers for automatic sequences arising from substitutions $\varrho$ satisfying condition $(\boldsymbol{\ast})$, i.e., `aperiodic, primitive, bijective, with $\varrho^{ }_0=\text{id}$', and provide explicit upper bounds.

A word $x\in\mathcal{A}^{\mathbb{N}}$ is called \emph{linearly recurrent} if 
there exists a positive constant $R_x$, such that the distance between any two consecutive occurrences of a finite subword $u$ of $x$ is at most $R_x|u|$. 
We say that $R_x$ is a \emph{linear recurrence constant} for $x$. Since the fixed points of a primitive substitution $\varrho$ are linearly recurrent (see~\cite{AS,Du1998,DuHoSk1999}), and moreover, one can find an $R_x$ that is independent of $x$ and depends only on $\varrho$ (see~\cite[Thm.~18]{Du1998}), we can associate a \emph{linear recurrence constant} to $\varrho$, and denote it by $R=R_\varrho$.
Let $u\in \mathcal{L}$. A \emph{return word} $v$ to $u$ 
is a legal word such that (i) $vu\in\mathcal{L}$, (ii) $u$ is a prefix of $v$, and (iii) $u$ occurs exactly once in $v$. 
Below, we mention some well-known results on the linear recurrence constant for primitive substitutions; compare \cite{Du1998,DuHoSk1999,DuLe2022}.

\begin{prop}\label{prop:LR-bound}
Let $\varrho$ be a primitive constant-length substitution on a finite alphabet.
\begin{enumerate}
    \item The substitution $\varrho$ is linearly recurrent for the constant $R=L\zeta^{ }_2$, where $\zeta^{ }_2$ is the maximum length of a return word for a legal word of length $2$ in $\mathcal{L}$.
    \item All legal words of length $N$ appear in any legal word of length $(R+1)N$.
\end{enumerate}
\end{prop}

\begin{definition}\label{def:B(c,L)}
Consider the class of length-$L$ substitutions on $c$ letters that satisfy the $(\boldsymbol{\ast})$ condition. 
We denote by  $\mathcal{B}(c,L)$  the set of all fixed points of substitutions from this class.
\end{definition}

One can compute $R$ that works for all $x\in\mathcal{B}(c,L)$ by computing an upper bound for $\zeta^{ }_2$ that depends only on $c$ and $L$. We obtain an upper bound for this gap using the level-$2$ induced substitution on $\mathcal{L}_2$, where $\mathcal{L}_2$ is the set of all length-$2$ legal words.

First, we identify $\mathcal{L}_2$ with the set of \emph{right-collared words} of the form $a^{ }_b$, where $ab\in \mathcal{L}$. 
The level-$2$ induced substitution $\varrho^{(2)}\colon \mathcal{L}_2\rightarrow\left(\mathcal{L}_2\right)^+$ is then the substitution arising from the original $\varrho$ which respect the collaring. 
As an example, the level-$2$ induced substitution for Thue--Morse is given by 
\[
\varrho^{(2)}\colon
 0^{ }_0\mapsto 0^{ }_{1}1^{ }_0 \quad \quad  1^{ }_0\mapsto 1^{ }_{0}0^{ }_0 \quad \quad
 0^{ }_1\mapsto 0^{ }_{1}1^{ }_1\quad \quad  1^{ }_1\mapsto 1^{ }_{0}0^{ }_1. 
\]
It is well known that $\varrho^{(2)}$ is also primitive whenever $\varrho$ is primitive; see \cite[Sec.~4.8.3]{TAO}. 
Combining (1) in Proposition~\ref{prop:LR-bound} with Wielandt's bound \cite{Wie} for the index of primitivity yields the following result.
\begin{lem}\label{lem:recurrence-const}
For given $c,L\geqslant 2$ and $x\in\mathcal{B}(c,L)$, a linear recurrence constant for $x$ is
\[
R= 2L^{c^4-2c^2+3}-L.
\] 
\end{lem}

\begin{proof}
The index of primitivity of a $c\times c$ primitive matrix $M$ is bounded from above by $c^2-2c+2$, i.e., $M^{c^2-2c+2}>0$ (seen entry-wise). 
Let $\varrho$ be a substitution satisfying the conditions in Definition~\ref{def:B(c,L)}, $\varrho^{(2)}$ the level-2 induced substitution, and $M^{(2)}$ be the substitution matrix of $\varrho^{(2)}$. 
Note that there are at most $|\mathcal{A}|^2=c^2$ length-$2$ legal words for $\varrho$. 
This means the size of $M^{(2)}$ is at most $c^2\times c^2$. 
Applying Wielandt's bound \cite{Wie}, we get that $\left(M^{(2)}\right)^{c^4-2c^2+2}$ is a strictly positive matrix. 
It follows that $\left(\varrho^{(2)}\right)^{c^4-2c^2+2}(a^{ }_b)$ contains all collared words in $\mathcal{L}_2$. 

Fix $a^{ }_b\in \mathcal{L}_2$. We know that $x$ can be written as a concatenation of level-$(c^4-2c^2+2)$ superwords of $\varrho$, all of which admitting at least one occurrence of $ab$ (possibly at the border) by the argument above. 
It follows that, for any $a^{ }_b\in \mathcal{L}_2$, the longest return word to $a^{ }_b$ has length at most $2L^{c^4-2c^2+2}-1$. A direct application of (1) in Proposition~\ref{prop:LR-bound} proves the claim. 
\end{proof}

\begin{remark}
We comment on the generality of the proof of the previous lemma. First, note that it only depends on the size $c$ of the alphabet and the length $L$ of the alphabet, and hence it gives a linear recurrence constant for all substitutions in the class considered in this section, parametrised by $c$ and $L$. 
Second, since bijectivity is invoked nowhere in the proof, such a bound can be used for extensions to more general classes. \exend
\end{remark}

Note that the set $\mathcal{B}(c,L)$ is a non-empty proper subset of $\mathcal{A}^{\mathbb{N}}$.
As a direct consequence of van der Waerden's theorem, we have the following. 

\begin{prop}\label{prop: vdW bijective exists}
Given $c$, $L\geqslant 2$ and $M\geqslant 1$, there exists a positive integer $n$ such that every length-$n$ subword of every element of $\mathcal{B}(c,L)$ contains a length-$M$ monochromatic arithmetic progression.
\end{prop}

\begin{definition}
Given $c$, $L\geqslant 2$ and $M$, we call the smallest threshold of the number $n$ predicted by Proposition~\ref{prop: vdW bijective exists} a \emph{van der Waerden-type number for $\mathcal{B}(c,L)$}, and we denote it by $W(\mathcal{B}(c,L), M)$.
\end{definition}

It is clear that $W(\mathcal{B}(c,L), M)\leqslant W(c,M)$.

\begin{prop}\label{prop: vdW bijective upper bounds}
For $c,L\geqslant2$ and $M\geqslant1$ , one has
\[
W(\mathcal{B}(c,L), M)\leqslant (R+1) L^{kc!},
\]
where $k=\lceil \log_{L}{M} \rceil$ and $R=2L^{c^4-2c^2+3}-L$.
\end{prop}

\begin{proof}
 By Proposition~\ref{prop:genbijective}, we know that, for any $n\geqslant1$, the maximum length of monochromatic arithmetic progressions in any element in $\mathcal{B}(c,L)$ satisfies,
\[
A(d_n)\geqslant L^n\;,\quad\text{where}\;\;d_n = \frac{L^{n|G|}-1}{L^n-1}\;.
\]
Let $k$ be the least non-negative integer such that $L^k\geqslant M$, which one can write as $k=\lceil \log_{L}{M} \rceil$. 
Then, $A(d_k)\geqslant L^k\geqslant M$. 
Since the arithmetic progression from Proposition~\ref{prop:genbijective} starts at $0$, every fixed point $x$ of $\varrho$ has a prefix $y$ of length 
$1 + (L^k-1) \cdot d_k=L^{k|G|}$ 
containing a monochromatic arithmetic progression of difference $d_k$ and length $M$.

From the discussion above, $x$ is linearly recurrent for some constant $R_x >0$. By Property~2 in Proposition~\ref{prop:LR-bound}, all subwords of $x$ of length $L^{k|G|}$ (in particular the subword $y$) appear in every subword of length $(R_x+1)L^{k|G|}$. 
So, every subword of $x$ of length $(R_x+1)L^{k|G|}$ contains a monochromatic arithmetic progression of difference $d_k$ and length $M$. From Lemma~\ref{lem:recurrence-const}, one can choose $R_x$ to be $2L^{c^4-2c^2+3}-L$. To complete the proof, notice that $|G|\leqslant |S_c|=c!$.
\end{proof} 

Proposition~\ref{prop: vdW bijective upper bounds} can be reformulated as follows: if $\varrho$ is a length-$L$ substitution on $c$ letters which satisfies $(\ast)$ and $M\geqslant1$, every legal word of $\varrho$ of length at least $(R+1)L^{kc!}$ contains a monochromatic arithmetic progression of length $M$.

\begin{example}\label{ex:vdW-type number for c=l=2}
Consider the case $c=L=2$, which is generated by the Thue--Morse substitution
\[
\varrho\colon\;
\begin{matrix*}[l]
a &\mapsto & ab\\
b &\mapsto & ba
\end{matrix*}.
\]
The substitution $\varrho$ has two fixed points. 
The methods in~\cite{DuHoSk1999} yield $R_{\varrho}\leqslant16$, which can be further improved to $R_{\varrho}=9$ using the software Walnut; see \cite[Ex.~3.13]{DuLe2022}. 
Using this result we obtain $W(\mathcal{B}(2,2), M)\leqslant 10 \cdot 4^k$, where $k=\lceil \log_2{M} \rceil$.
Thus the bounds for the first few van der Waerden-type numbers are
\begin{align*}
&W(\mathcal{B}(2,2), M)\leqslant 640 && \text{for $M=6,7,8$},\\
&W(\mathcal{B}(2,2), M)\leqslant 2560 && \text{for $8<M\leqslant 16$},\\
&W(\mathcal{B}(2,2), M)\leqslant 10240&& \text{for $16<M\leqslant 32$},\\
&W(\mathcal{B}(2,2), M)\leqslant 40960&& \text{for $32<M\leqslant 64$},
\end{align*}
which are significantly lower than the respective bounds for the general van der Waerden numbers. 
\exend
\end{example}

The bound for $R$ established in Proposition \ref{prop: vdW bijective upper bounds} is far from optimal. For example, for the family $\mathcal{B}(2,2)$ studied in Example~\ref{ex:vdW-type number for c=l=2}, we obtain $R\leqslant 2^{12}-2=4094$ and consequently,
\[
\qquad\text{where}\qquad
W(\mathcal{B}(2,2), M)\leqslant 4095 \cdot 4^k,
k=\lceil \log_2{M} \rceil,
\]
which is a weaker bound than the bound obtained in Example~\ref{ex:vdW-type number for c=l=2} using the optimal value of the recurrence constant ($R=9$). 
It would be interesting to obtain a better method to compute the constant $R$ in Proposition~\ref{prop: vdW bijective upper bounds} and hence improve the upper bound of $W(\mathcal{B}(c,L), M)$.

\subsection{Upper bounds of \texorpdfstring{$A(d)$}{text} for Abelian bijective substitutions}\label{sec:UB-bij-abel}

Throughout the whole Section~\ref{SecBijective}, we consider substitutions $\varrho$ satisfying condition $(\ast)$.  In this subsection, we add the additional assumption that the group $G$ generated by the columns of $\varrho$ is Abelian. 
From Lemma~\ref{lem:recurrence-const}, there exists a positive integer $N$ such that for any $a\in\mathcal{A}$, $\varrho^N(a)$ contains all legal words of length $2$, i.e., $N\coloneqq \min\{n:\mathcal{L}_2\subseteq \mathcal{L}(\varrho^n(a)),\,a\in\mathcal{A}\}$, where $\mathcal{L}(\varrho^n(a))$ is the set of all words appearing in $\varrho^{n}(a)$.
Let $v\in\mathcal{A}^{\mathbb
N}$ be a fixed point of $\varrho$.
The goal of this section is to provide an upper bound on $A(d)$ for $v$.

We begin with the following results regarding certain columns of $\varrho^{N+M}$ under the existence of certain progressions in $v$, where $M\geqslant 1$. 

\begin{lem}\label{lem:column-eq-Abelian}
 Let $d,M\geqslant 1$, with $d<L^{N+M}$.
 Let $\ell\coloneqq\gcd(d,L^{N+M})$ 
and assume that there exists a non-negative integer $n$ such that $v^{ }_n=v^{ }_{n+jd}$, for $j=0,1,\dotsc,L^{N+M}/\ell$.
Then for all $k\in\mathbb{Z}$ such that $0\leqslant n+k\ell<n+k\ell+d<L^{N+M}$, one has 
\[
\big(\varrho^{N+M}\big)^{ }_{n+k\ell}=\big(\varrho^{N+M}\big)^{ }_{n+k\ell+d}.
\]
\end{lem}

\begin{proof}
We have a trivial inclusion
\begin{align*}
    \{[n+id]_{L^{N+M}}\mid i=0,1,\ldots, L^{N+M}/\ell-1\}\subset\{[m]_{L^{N+M}}\mid 
    m\in\mathbb{Z}, m\equiv n \mod \ell\},
\end{align*}
where $[\cdot]_{L^{N+M}}$ denotes
the equivalence class of
 natural numbers mod $L^{N+M}$.
For $0\leqslant i<j<\frac{L^{N+M}}{\ell}$, we have 
$n+id\not\equiv n+jd \mod L^{N+M}$
by the definition of $\ell$.
This means the two sets above have the same cardinality, and hence are the same set.

We see for each $k$ satisfying the condition above, there exists $i\in \left\{0,1,\ldots ,(L^{N+M}/\ell)-1\right\}$ such that
$n+k\ell\equiv n+id \pmod{L^{N+M}}$.
This implies there is a positive integer $s$ for which
\[     n+k\ell+sL^{N+M}=n+id\quad\text{ and }\quad n+k\ell+d+sL^{N+M}=n+(i+1)d.
\]
By assumption, $v^{ }_{n+id}=v^{ }_{n+(i+1)d}$. 
Note that $v^{ }_{n+id}$ is the $(n+k\ell)$th letter in $\varrho^{N+M}(v^{ }_s)$ and $v^{ }_{n+(i+1)d}$ is the $(n+k\ell+d)$th letter in
$\varrho^{N+M}(v^{ }_s)$. This means
in the $(n+k\ell)$th  and $(n+k\ell+d)$th columns of $\varrho^{N+M}$, the images of one letter $v^{ }_s$ are the same. Since
the column group $G$ is Abelian and acts on $\mathcal{A}$ transitively, we see that the columns 
$\big(\varrho^{N+M}\big)^{ }_{n+k\ell}$ and $\big(\varrho^{N+M}\big)^{ }_{n+k\ell+d}$
(seen as permutations of $\mathcal{A}$) must coincide, thus proving the claim. 
\end{proof}

We now relate the column equality result in Lemma~\ref{lem:column-eq-Abelian} to existence of infinitely long progressions in $v$. 

\begin{lem}\label{lem:infinite-prog}
Let $M\geqslant 1, d\leqslant L^{M}$, and
$\ell=\gcd(L^M,d)$.
Let $n$ be a non-negative integer such that
$\big(\varrho^{N+M}\big)^{ }_{n+k\ell}=\big(\varrho^{N+M}\big)^{ }_{n+k\ell+d}$
for each integer $k$ satisfying $0\leqslant n+k\ell<n+k\ell+d<L^{N+M}$. 
Then, we have
$v^{ }_n=v^{ }_{n+jd}$
for all $j\in\mathbb{N}$. That is, there exists an arithmetic progression of infinite length and difference $d$
starting at $n$.
\end{lem}

\begin{proof}
For each $j\in \mathbb{N}$, $v^{ }_{n+jd}$ and $v^{ }_{n+(j+1)d}$ are included either in a supertile of length $L^M$ or two
consecutive such supertiles. 
This means there exists a $t\geqslant 0$ such that $v^{ }_{n+jd}$ is
the $(n+jd-L^Mt)$th letter in $\varrho^M(v^{ }_tv^{ }_{t+1})$ and
$v^{ }_{n+(j+1)d}$ 
is the $(n+(j+1)d-L^{M}t)$th letter in $\varrho^M(v^{ }_tv^{ }_{t+1})$. 
By assumption, 
the word $v^{ }_{t}v^{ }_{t+1}$ appears in $\varrho^N(a)$, for each $a\in\mathcal{A}$.
There is an $s$ such that the $s$th letter in $\varrho^N(a)$ is $v^{ }_t$ and
the $(s+1)$th letter is $v^{ }_{t+1}$.
This means that $v^{ }_{n+jd}$ is
the $(n+jd-L^Ms-L^Mt)$th letter in $ \varrho^{N+M}(a)$ and
 $v^{ }_{n+(j+1)d}$ is the
 $(n+(j+1)d-L^Ms-L^Mt)$th letter in $ \varrho^{N+M}(a)$. 
By the definition of $\ell$ and the assumption on the columns for $\varrho^{N+M}$, the $(n+jd-L^Ms-L^Mt)$th column and the $(n+(j+1)d-L^Ms-L^Mt)$th column are the same, and we have
  $v^{ }_{n+jd}=v^{ }_{n+(j+1)d}$. Since $j$ is arbitrary, the claim follows. 
\end{proof}

\begin{prop}\label{prop1_upper_bound}
Let $\varrho$ be a length-L substitution which satisfies $(\ast)$ and whose column group $G$ is Abelian. 
Let $v$ be a fixed point of $\varrho$. 
Fix a difference $d$ and let $M$ be a positive integer such
that $d\leqslant L^M$. 
Suppose $\gcd(d,L^M)=\gcd(d,L^{N+M})\eqqcolon\ell$.
We then have
$A(d)\leqslant\frac{L^{N+M}}{\ell}$.
\end{prop}

\begin{proof}
Suppose there exists $n\geqslant 1$ such that $v^{ }_n=v^{ }_{n+jd}$, for $j=0,1,\dotsc,L^{N+M}/\ell$.
It follows from Lemmas~\ref{lem:column-eq-Abelian} and \ref{lem:infinite-prog} that
$v^{ }_n=v^{ }_{n+jd}$, for all $j\in\mathbb{N}$. 
This contradicts Proposition~\ref{prop2:finiteness} stating that $v$ does not admit infinitely long monochromatic progressions, which immediately implies the claim on $A(d)$. 
\end{proof}

Combining Proposition~\ref{prop1_upper_bound} with Proposition~\ref{prop:genbijective}, we get the following.

\begin{coro}\label{coro:upper bound 1}
Let $v$ be a fixed point of an Abelian, length-$L$ substitution $\varrho$ which satisfies $(\boldsymbol{\ast})$. Then, for all $k\geqslant 1$,
\[
L^{k} \leqslant A\left(\frac{L^{k|G|}-1}{L^k-1}\right)\leqslant L^{k|G|+N},
\]
where $G$ is the group generated by the columns of $\varrho$ and $N=\min\{n:\mathcal{L}_2\subseteq\mathcal{L}(\varrho^n(a))\}$.
\end{coro}

Note that from Lemma~\ref{lem:recurrence-const}, $N$ is bounded from above by $c^4-2c^2+3$, where 
$c$ is the size of the alphabet (this is one more than the bound for the index of primitivity for $M^{(2)}$ to include the case when a length-$2$ legal word appears at the boundary). 

One of the restrictions in Proposition~\ref{prop1_upper_bound} is that, for an arbitratry $L$, one is only able to give upper bounds for $A(d)$ for differences which satisfy the $\gcd$-condition. 
In what follows, we give a subclass of lengths for which it is possible to give an upper bound for $A(d)$ for all $d$.

\begin{prop}\label{prop2_upper_bound}
Let $\varrho,v,$ and $L$ be as in Proposition~\textnormal{\ref{prop1_upper_bound}}.
Let $L=p_1^{n_1}p_2^{n_2}\cdots p_t^{n_t}$ be the prime factorisation of
$L$ with $p_1<p_2<\cdots <p_t$
and assume $n_1\leqslant n_2\leqslant\cdots \leqslant n_t$.
Then, for each $d\geqslant 1$ there exist $M\geqslant 1$ with $d\leqslant L^M$ such that $\gcd(d,L^{N+M})=\gcd(d,L^M)$. 
Moreover, we have
$A(d)\leqslant L^{N+1}d^B$,
where $B=\frac{\log L}{n_1\log p_1}$.
In particular, if $t=1$, 
(that is, $L$ is a power of a prime),
$A(d)\lesssim L^{N+1} d$. 
\end{prop}
\begin{proof}
Choose $M\in\mathbb{N}$ such 
that $p_1^{n_1(M-1)}\leqslant d<p_1^{n_1M}$. 
The equality of the greatest common divisors follows from $p_i^{n_iM}\nmid d$, for all $1\leqslant i\leqslant t$. Set $B=\frac{\log L}{n_1\log p_1}$.  
By Proposition \ref{prop1_upper_bound},
we have
\begin{align*}
    A(d)\leqslant \frac{L^{N+M}}{\ell}
    =\frac{L^{N+1}}{\ell}(p_1^{n_1B})^{M-1}=
    \frac{L^{N+1}}{\ell}(p_1^{n_1(M-1)})^{B}\leqslant\frac{L^{N+1}}{\ell}d^B\leqslant L^{N+1}d^{B}.
\end{align*}
The last claim follows since $L=p_1^{n_1}$ implies $B=1$. 
\end{proof}

One can leverage the previous proposition to obtain lower bounds for van der Waerden-type numbers $W(\mathcal{B}(c,L),M)$, for certain values of $L$ and $M$. 

\begin{coro}\label{coro:lower bound VdW}
    Let $c,m>1$ and assume $L$ admits the same form as in  Proposition \ref{prop2_upper_bound}, i.e., it has prime factorisation
      $L=p_1^{n_1}p_2^{n_2}\cdots p_t^{n_t}$,
      where $p_1<p_2<\cdots <p_t$ are such that
       $n_1\leqq n_2\leqq \cdots\leqq n_t$.
      Then we have
     \begin{align*}
         W(\mathcal{B}(c,L),L^{N_0+1}m^{\lceil B\rceil}+1)>L^{N_0+1}m^{\lceil B\rceil+1}+1,
     \end{align*}
     where $N_0=c^4-2c^2+3$ and $B=\frac{\log L}{n_1\log p_1}$. 
\end{coro}

\begin{proof}
    We will prove a stronger statement that there exists
    an $x\in\mathcal{B}(c,L)$ such that any of its subwords
    of length $L^{N_0+1}m^{\lceil B\rceil+1}+1$ does not 
    contain monochromatic arithmetic progressions of
    length $L^{N_0+1}m^{\lceil B \rceil}+1$.

    Take a $\varrho$, which is a primitive, aperiodic, bijective substitution of length $L$ such that its column group is
    abelian and $\varrho^{ }_0=\text{id}$. Such a substitution always exists. Fix a length $L$ and the size of the alphabet $c$. Without loss of generality, one can force the column group to be the cyclic group $G=C_c$ of order $c$, which is Abelian and acts transitively
    on $\mathcal{A}=\left\{0,\ldots,c-1 \right\}$. Transitivity is already sufficient to ensure primitivity; see \cite[Prop.~2.3]{BuLuMa2021}. 

    It remains to construct an aperiodic substitution with that group profile, for any given length. Here, we use a criterion for aperiodicity provided in \cite[Prop.~4.1]{KY}, which states that a sufficient condition for aperiodicity for primitive and bijective substitutions is the existence of two length-two legal words which share either the same starting letter or the same ending letter; see also \cite[Prop.~2.5]{BuLuMa2021}. 
    
    We first handle the case when $L\geqslant 3$. For such lengths, we choose $\varrho^{ }_0=\varrho^{ }_1=\text{id}$ and $\varrho^{ }_2=(12\cdots 0)$, where $\varrho^{ }_2$ generates $G$, and we fill the other positions with permutations from $C_c$. From construction, we immediatelty see that $00$ and $01$ are both legal, and hence implies aperiodicity. 

    For the case $L=2$, we pick $\varrho^{ }_0=\text{id}$ and $\varrho^{ }_1=(12\cdots 0)$ and show that this substitution is aperiodic. Note that, under $\varrho$, $0\mapsto 01$ and $(c-1)\mapsto (c-1)0$. 
    Applying $\varrho^2$ to $(c-1)$ yields 
    $\varrho^{2}(c-1)=(c-1)001$, which means $00$ and $01$ are both legal with respect to $\varrho$. By the same argument for the previous cases, we obtain aperiodicity for all such substitutions.

    Let $N$ be a natural number such that
    for any alphabet $a$, $\varrho^N(a)$ contains
    all of the two-letter legal words for $\varrho.$
    By the argument of Lemma \ref{lem:recurrence-const},
    we have $N\leqslant N_0$. Let $x$ be a fixed point
    for $\varrho$.

    If $d'\leqslant m$, by Proposition \ref{prop2_upper_bound}, the maximal length of 
    monochromatic arithmetic progression is less than or
    equal to $L^{N_0+1}(d')^{\lceil B\rceil}$, which is
    less than or equal to $L^{N_0+1}m^{\lceil B\rceil}$.
    There are no monochromatic arithmetic progressions 
    of difference $d'$ and length
    $L^{N_0+1}m^{\lceil B\rceil}+1$ anywhere in $x$, and
    so anywhere in its subwords.

    If $d'>m$, the 
    progressions of difference $d'$ and length
    $L^{N_0+1}m^{\lceil B\rceil}+1$ span
    as long as $L^{N_0+1}m^{\lceil B\rceil}d'+1$ and
    cannot be contained in a subword of length
    $L^{N_0+1}m^{\lceil B\rceil+1}+1$.
    In either case, the subwords of $x$  with length   $L^{N_0+1}m^{\lceil B\rceil+1}+1$ do not contain
    monochromatic arithmetic progressions of length
    $L^{N_0+1}m^{\lceil B\rceil}+1$.
\end{proof}

\subsection{Examples}\label{sec:Ex-bij}

\subsubsection{Thue--Morse sequence over \texorpdfstring{$L$}{text} letters}\label{sec: L-TM}

The Thue--Morse sequence over the alphabet $\mathcal{A}_L=\{0,1,\dots,L-1\}$ is the infinite word $v=v^{ }_0v^{ }_1v^{ }_2\cdots$, where $v^{ }_i$ is given by the sum (modulo $L$) of the digits in the base-$L$ representation of $i$; see \cite{CKM}. This sequence can also be defined to be the fixed point, with first letter $0$, of the primitive, length-$L$, bijective substitution $\varphi$ defined as~\cite{BaRoYa2018}
\begin{equation}\label{eq: Thue--Morse rule}
\varphi(a)=\varphi^{ }_0(a)\;\varphi^{ }_1(a)\;\dotsb\;\varphi^{ }_{L-1}(a),
\qquad\text{where}\qquad
\varphi^{ }_i(a)=a+i\pmod L.
\end{equation}

The exact values of $A(L^n-1)$ for $L=2$ and $L=3$ where obtained (in \cite{Olga0} and \cite{Olga1}, respectively), and it was shown that the same arguments can be used for any prime number $L$ (see \cite{Olga2}). The result for $L=2$ was reestablished in \cite{AGNS} using a different approach. The key argument of this approach can be easily generalised for all $L$, giving Proposition~\ref{prop:lower bound T-M} below as a result.

The group $G$ generated by the columns of $\varphi$ is the cyclic group $C_{L}$ of order $L$. We write $C_{L}=\left\langle g\right\rangle$ multiplicatively,
where $g$ corresponds to adding $1\pmod L$. We can easily see that $\varphi^{ }_{\ell}$, the $\ell$-th column of $\varphi$, is given by $g^{\ell}$. 
We next show that $\varphi$ is actually $g^{L-1}$-palindromic and so, Proposition~\ref{prop:pal} can also be directly applied to get a subsequence along which $A(d)$ grows faster.

\begin{prop}\label{prop:L-TM-g-pal}
Let\/ $\varphi$ be the generalised Thue--Morse substitution over\/ $L$ letters from Eq.~\eqref{eq: Thue--Morse rule}. Then,\/ $\varphi$ is $g^{L-1}$-palindromic. Consequently, for any fixed point of $\varphi$ one has
\begin{equation}\label{eq:LB-TM-pal)}
 A(L^{n}-1)\geqslant L^{n},
\end{equation}
for all $n\geqslant 1$.
\end{prop}

\begin{proof}
It suffices to show that, for each $0\leqslant i\leqslant L-1$, one has $\varphi^{ }_{i}\cdot \varphi^{ }_{L-(i+1)}=g^{L-1}$. This follows immediately from $\varphi^{ }_i=g^{i}$. The lower bound for $A(L^{n}-1)$ follows from Proposition~\ref{prop:pal} by choosing $\ell=2$.
\end{proof}

Note that we can improve the lower bounds given in Eq.~\eqref{eq:LB-TM-pal)} when $n\equiv 0\bmod L$ by looking at the concatenation of three level-$n$ superwords, which we carry out below. As mentioned earlier, this result generalises that in \cite{AGNS} for $L=2$ to any arbitrary $L$. 

\begin{prop}\label{prop:lower bound T-M}
Let\/ $\varphi$ be the generalised Thue--Morse substitution over\/ $L$ letters from Eq.~\eqref{eq: Thue--Morse rule}. For any fixed point of $\varphi$, one has
\[
A(L^{n}-1)\geqslant \begin{cases}
L^n +2L, & \text{if $n\equiv 0 \bmod L$},\\
L^n, & \text{otherwise}.
\end{cases}
\]
\end{prop}

\begin{proof}
The case when $n\not\equiv 0 \bmod L$ is already covered in Proposition~\ref{prop:L-TM-g-pal} so we assume from hereon that $n\equiv 0 \bmod L$. From the proof of Proposition~\ref{prop:pal}, we have that, for $\varphi^n$, one has $\left(\varphi^{n}\right)_{i}\cdot\left(\varphi^{n }\right)_{L^{n}-(i+1)}=\left(g^{L-1}\right)^n=\text{id}$, for all $0\leqslant i\leqslant L^n-1$. This means, if we now look at $\varphi^{2n}$, we get $\big(\varphi^{2n}\big)_{i^{ }_m}=\text{id}$ with $i^{ }_m=m(L^{n}-1)$ and $1\leqslant m\leqslant L^{n}$. Note further that $\big(\varphi^{2n}\big)_{i^{ }_0}=\big(\varphi^{2n}\big)_{i^{ }_{L^n+1}}=\text{id}$. Altogether, this yields a monochromatic arithmetic progression of $a$s of length $L^n+2$ within the superword $\varphi^{2n}(a)$. The goal is now to look at progressions of $a$s in $\varphi^{2n}(a-1)$ and $\varphi^{2n}(a+1)$ of the same difference. We then extend the progression from $\varphi^{2n}(a)$ to a longer progression in $\varphi^{2n}((a-1)(a)(a+1))$. Note that the word $(a-1)(a)(a+1)\in \mathcal{L}_3(\varphi)$, for any $a\in\mathcal{A}$.

We first look at the supertile $\varphi^{2n}(a-1)$. 
We show that for $0\leqslant m\leqslant L-2$, at positions $i^{ }_m=L^{2n}-(m+1)L^n+(m+1)$, one has $\big(\varphi^{2n}\big)_{i^{ }_m}=g$. These are the positions which correspond to the continuation of the progression from $\varphi^{2n}(a)$ with difference $d=L^{n}-1$; see Figure~\ref{fig:L-Thue-Morse}. One can check that the $L$-ary expansion of $i^{ }_m$ reads
\[
[\,\overbracket[0.5pt]{L-1,L-1,\ldots,L-(m+1)}^{n \text{ digits}},\overbracket[0.5pt]{0,0,\ldots,m+1 \vphantom{L()}}^{n \text{ digits}}\,]
\]
From Fact~\ref{fact:i-th-column-power}, we get that 
\begin{align*}
\big(\varphi^{2n}\big)_{i^{ }_m}&=(\varphi^{ }_{L-1}\circ\varphi^{ }_0)^{n-1}\circ (\varphi^{ }_{L-(m+1)}\circ\varphi^{ }_{m+1})=(g^{L-1})^{n-1}(g^{L-1}\cdot g)\\
&=(g^{L-1})^n\cdot g=g,
\end{align*}
where the second equality holds since $\varphi$ is $g$-palindromic and the last equality holds since $n\equiv  0 \bmod L$. Note that this is only true for $0\leqslant m\leqslant L-2$, since for $m=L-1$, one gets $i^{ }_m=[L-1,L-1,\ldots,L-1,0,0,\ldots,1,0]$. Carrying out the same calculation, we get $\big(\varphi^{2n}\big)_{i^{ }_m}=(g^{L-1})^{n-2}\cdot(g\cdot g^{L-1})=g\cdot (g^{L-1})^{n-1}\neq g$, since $n-1$ is coprime with $L$. This means the extension of the arithmetic progression in $\varphi^{2n}(a-1)$ has length at most $L-1$. 

One can do an analogous analysis for the supertile on the right, which is $\varphi^{2n}(a+1)$. Here the relevant positions are of the form $j_m=(m+1)L^n-(m+2)$, and one needs to show that $\big(\varphi^{2n}\big)_{j_m}=g^{L-1}$. Since the proof uses the same arguments above, we leave it to the reader. Note that here, one can also show that, for $m=L-1$, $\big(\varphi^{2n}\big)_{j_m}\neq g^{L-1}$, which implies that the extension to the right also has length at most $L-1$. Considering the progression of $a$s which straddles across these three supertiles verifies the claim.  
\end{proof}

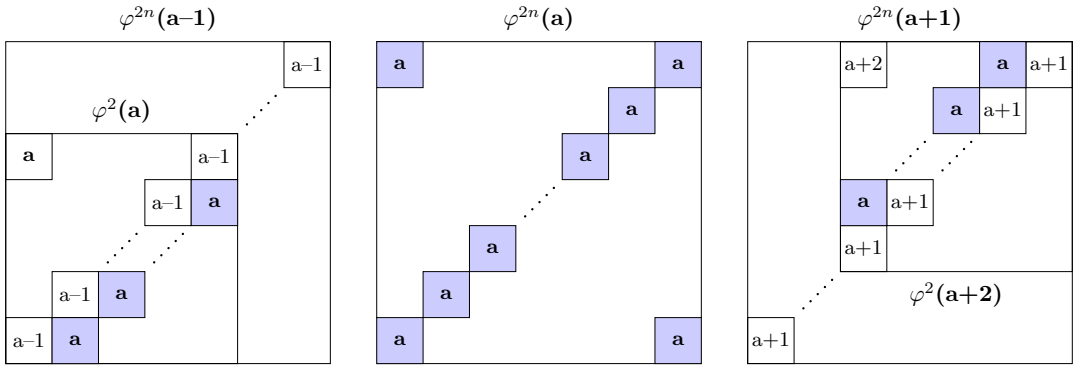
\begin{figure}[!h]
\begin{centering}
\resizebox{0.9\textwidth}{!}{
    \begin{tikzpicture}[x=7mm, y=7mm]
    \draw (0,0) rectangle +(7,7);
    \node at (3.5,7.5) {\small\bf $\mathbf\varphi^{2n}$(a--1)};
    \draw (8,0) rectangle +(7,7);
    \node at (11.5,7.5) {\small\bf $\mathbf\varphi^{2n}$(a)};
    \draw (16,0) rectangle +(7,7);
    \node at (19.5,7.5) {\small\bf $\mathbf\varphi^{2n}$(a+1)};
    \draw (0,0) rectangle +(5,5);
    \node at (2.5,5.5) {\small\bf $\mathbf\varphi^{2}$(a)};
    \draw (18,2) rectangle +(5,5);
    \node at (20.5,1.5) {\small\bf $\mathbf\varphi^{2}$(a+2)};
    \foreach \position in {
    	(0,0), (0,4), (1,1), (3,3), (4,4), (6,6),
    	(16,0), (18,2), (18,6), (19,3), (21,5), (22,6)
    }
    \draw \position rectangle +(1,1);
    \foreach \position in {
    	(1,0), (2,1), (4,3),
    	(8,0), (8,6), (9,1), (10,2), (12,4), (13,5), (14,0), (14,6),
    	(18,3), (20,5), (21,6)
    }
    \draw[fill=blue!20] \position rectangle +(1,1);
    \foreach \position in {
    	(0.5,4.5), (1.5,0.5), (2.5,1.5), (4.5,3.5),
    	(8.5,0.5), (8.5,6.5), (9.5,1.5), (10.5,2.5), (12.5,4.5), (13.5,5.5), (14.5,0.5), (14.5,6.5),
    	(18.5,3.5), (20.5,5.5), (21.5,6.5)
    }
    \node at \position {\footnotesize\bf a};
    \foreach \position in {(0.5,0.5), (1.5,1.5), (3.5,3.5), (4.5,4.5), (6.5,6.5)}
    \node at \position {\footnotesize a--1};
    \foreach \position in {(16.5,0.5), (18.5,2.5), (19.5,3.5), (21.5,5.5), (22.5,6.5)}
    \node at \position {\footnotesize a+1};
    \node at (18.5,6.5) {\footnotesize a+2};
    \draw[line width=1pt, dash pattern=on .03mm off 1.5mm, line cap=round] (2.2,2.2) -- (2.9,2.9);
    \draw[line width=1pt, dash pattern=on .03mm off 1.5mm, line cap=round] (3.2,2.2) -- (3.9,2.9);
    \draw[line width=1pt, dash pattern=on .03mm off 1.5mm, line cap=round] (5.2,5.2) -- (5.9,5.9);
    \draw[line width=1pt, dash pattern=on .03mm off 1.5mm, line cap=round] (11.2,3.2) -- (11.9,3.9);
    \draw[line width=1pt, dash pattern=on .03mm off 1.5mm, line cap=round] (17.2,1.2) -- (17.9,1.9);
    \draw[line width=1pt, dash pattern=on .03mm off 1.5mm, line cap=round] (19.2,4.2) -- (19.9,4.9);
    \draw[line width=1pt, dash pattern=on .03mm off 1.5mm, line cap=round] (20.2,4.2) -- (20.9,4.9);
    \end{tikzpicture}
}
\caption{A monochromatic arithmetic progression of $a$s of length $L^n+2L$ accommodated within the supertiles $\varphi^{2n}(a-1)$, $\varphi^{2n}(a)$ and $\varphi^{2n}(a+1)$. Here level-$2n$ supertiles are written as $L^n\times L^n$ blocks, which are read from left to right, and then top to bottom. The bottom-most shaded square in $\varphi^{2n}(a-1)$ corresponds to $i^{ }_0=L^{2n}-L^n+1$ while the top-most shaded square in $\varphi^{2n}(a+1)$ is at $j_0=L^n-2$.
}\label{fig:L-Thue-Morse}
\end{centering}
\end{figure}

We conjecture that the lower bounds given in Proposition~\ref{prop:lower bound T-M} are actually exact values. This has been settled when $n$ is prime in \cite{Olga2}.  We now look at other differences $d$. 
From Proposition~\ref{prop:genbijective}, we directly obtain the lower bound 
\[
A\left(\frac{L^{Ln}-1}{L^n-1}\right)\geqslant L^{n},
\]
for all $n \geqslant 1$. This result can also be geometrically visualised as in the previous proposition. 

\begin{example}[Ternary Thue--Morse]
We fix $L=3$ and consider the ternary Thue--Morse sequence, which is the fixed point $v=012\dotsb$ of the substitution
\[
\varphi\colon \begin{matrix}
0\mapsto 012\\ 
1\mapsto 120\\
2\mapsto 201
\end{matrix}\;.
\]

For differences of the form $d=3^{n}-1$,   Proposition~\ref{prop:lower bound T-M} and Corollary~\ref{coro:upper bound 1} imply that, for all $n \geqslant 1$, we have
$3^n \leqslant A(3^n-1) \leqslant 3^{n+4}$.
For differences of the form $d=3^{2n}+3^n+1$, it follows directly from Proposition~\ref{prop:genbijective} that, for all $n \geqslant 1$, we have
\[
A\bigl(3^{2n}+3^n+1\bigr)\geqslant 3^n;
\]
see Figure~\ref{fig:plot-A(d)-ternary} for a plot of $A(d)$ for differences up to $2200$, for the ternary Thue--Morse sequence.

We can give an alternative visual approach by identifying a long monochromatic arithmetic progression across a diagonal of a block substitution, as in the proof of Proposition~\ref{prop:lower bound T-M}, but now in three dimensions. As in Figure~\ref{fig:L-Thue-Morse}, we can consider the word $\varphi^{3n}(0)$ and arrange it inside a block. The only difference is now we arrange it in a three-dimensional cube of side-length $3^n$. There is no fixed choice of fitting the word inside a cube, and one must only be consistent when going up and through a layer. 

In our choice depicted in Figure~\ref{fig: cubes} below, we start from the lower left corner of the cube, traverse along the $x$-direction, then go up the next row. Once all rows in the bottom-most layer are filled, one moves one layer up and starts directly above the point where the origin is. The red shaded squares precisely correspond to the monochromatic arithmetic progression that starts at the origin, with difference $d=3^{2n}+3^n+1$ and has length $3^n$. \exend

\begin{figure}[!h]
\begin{centering}
\resizebox{0.4\textwidth}{!}{
    \def\r{0.5}
    \tdplotsetmaincoords{120}{120}
    \begin{tikzpicture}
    	[
    		tdplot_main_coords,
    		zero/.style = {opacity = 0.25, fill = purple},
    		zeromarked/.style = {opacity = 1, thick, fill = purple},
    		one/.style = {opacity = 0.25, fill = orange},
    		two/.style = {opacity = 0.25, fill = green}
    	]
    	\draw [dashed] (0, 0, 0) -- (0, 0, 8*\r);
    	\draw [dashed] (0, 0, 0) -- (0, 9*\r, 0);
    	\draw [dashed] (0, 0, 0) -- (9*\r, 0, 0);
    	\draw [dashed] (0, 0, 8*\r) -- (0, 9*\r, 8*\r);
    	\draw [dashed] (0, 0, 8*\r) -- (9*\r, 0, 8*\r);
    	\draw [dashed] (9*\r, 9*\r, 0) -- (9*\r, 0, 0);
    	\draw [dashed] (9*\r, 9*\r, 0) -- (0, 9*\r, 0);
    	\draw [dashed] (0, 9*\r, 0) -- (0, 9*\r, 8*\r);
    	\draw [dashed] (9*\r, 0, 0) -- (9*\r, 0, 8*\r);
    	\draw [dashed] (9*\r, 0, 8*\r) -- (9*\r, 9*\r, 8*\r);
    	\draw [dashed] (9*\r, 9*\r, 0) -- (9*\r, 9*\r, 8*\r);
    	\draw [dashed] (0, 9*\r, 8*\r) -- (9*\r, 9*\r, 8*\r);
    	\draw [zeromarked] (0, 0, 0) -- (\r, 0, 0) -- (\r, \r, 0) -- (0, \r, 0) -- cycle;
    	\draw [zero] (5*\r, 0, 0) -- (6*\r, 0, 0) -- (6*\r, \r, 0) -- (5*\r, \r, 0) -- cycle;
    	\draw [zero] (7*\r, 0, 0) -- (8*\r, 0, 0) -- (8*\r, \r, 0) -- (7*\r, \r, 0) -- cycle;
    	\draw [zero] (0, 5*\r, 0) -- (\r, 5*\r, 0) -- (\r, 6*\r, 0) -- (0, 6*\r, 0) -- cycle;
    	\draw [zero] (5*\r, 5*\r, 0) -- (6*\r, 5*\r, 0) -- (6*\r, 6*\r, 0) -- (5*\r, 6*\r, 0) -- cycle;
    	\draw [zero] (7*\r, 5*\r, 0) -- (8*\r, 5*\r, 0) -- (8*\r, 6*\r, 0) -- (7*\r, 6*\r, 0) -- cycle;
    	\draw [zero] (0, 7*\r, 0) -- (\r, 7*\r, 0) -- (\r, 8*\r, 0) -- (0, 8*\r, 0) -- cycle;
    	\draw [zero] (5*\r, 7*\r, 0) -- (6*\r, 7*\r, 0) -- (6*\r, 8*\r, 0) -- (5*\r, 8*\r, 0) -- cycle;
    	\draw [zero] (7*\r, 7*\r, 0) -- (8*\r, 7*\r, 0) -- (8*\r, 8*\r, 0) -- (7*\r, 8*\r, 0) -- cycle;
    	\draw [zero] (2*\r, \r, 0) -- (3*\r, \r, 0) -- (3*\r, 2*\r, 0) -- (2*\r, 2*\r, 0) -- cycle;
    	\draw [zero] (4*\r, \r, 0) -- (5*\r, \r, 0) -- (5*\r, 2*\r, 0) -- (4*\r, 2*\r, 0) -- cycle;
    	\draw [zero] (6*\r, \r, 0) -- (7*\r, \r, 0) -- (7*\r, 2*\r, 0) -- (6*\r, 2*\r, 0) -- cycle;
    	\draw [zero] (2*\r, 3*\r, 0) -- (3*\r, 3*\r, 0) -- (3*\r, 4*\r, 0) -- (2*\r, 4*\r, 0) -- cycle;
    	\draw [zero] (4*\r, 3*\r, 0) -- (5*\r, 3*\r, 0) -- (5*\r, 4*\r, 0) -- (4*\r, 4*\r, 0) -- cycle;
    	\draw [zero] (6*\r, 3*\r, 0) -- (7*\r, 3*\r, 0) -- (7*\r, 4*\r, 0) -- (6*\r, 4*\r, 0) -- cycle;
    	\draw [zero] (2*\r, 8*\r, 0) -- (3*\r, 8*\r, 0) -- (3*\r, 9*\r, 0) -- (2*\r, 9*\r, 0) -- cycle;
    	\draw [zero] (4*\r, 8*\r, 0) -- (5*\r, 8*\r, 0) -- (5*\r, 9*\r, 0) -- (4*\r, 9*\r, 0) -- cycle;
    	\draw [zero] (6*\r, 8*\r, 0) -- (7*\r, 8*\r, 0) -- (7*\r, 9*\r, 0) -- (6*\r, 9*\r, 0) -- cycle;
    	\draw [zero] (\r, 2*\r, 0) -- (2*\r, 2*\r, 0) -- (2*\r, 3*\r, 0) -- (\r, 3*\r, 0) -- cycle;
    	\draw [zero] (3*\r, 2*\r, 0) -- (4*\r, 2*\r, 0) -- (4*\r, 3*\r, 0) -- (3*\r, 3*\r, 0) -- cycle;
    	\draw [zero] (8*\r, 2*\r, 0) -- (9*\r, 2*\r, 0) -- (9*\r, 3*\r, 0) -- (8*\r, 3*\r, 0) -- cycle;
    	\draw [zero] (\r, 4*\r, 0) -- (2*\r, 4*\r, 0) -- (2*\r, 5*\r, 0) -- (\r, 5*\r, 0) -- cycle;
    	\draw [zero] (3*\r, 4*\r, 0) -- (4*\r, 4*\r, 0) -- (4*\r, 5*\r, 0) -- (3*\r, 5*\r, 0) -- cycle;
    	\draw [zero] (8*\r, 4*\r, 0) -- (9*\r, 4*\r, 0) -- (9*\r, 5*\r, 0) -- (8*\r, 5*\r, 0) -- cycle;
    	\draw [zero] (\r, 6*\r, 0) -- (2*\r, 6*\r, 0) -- (2*\r, 7*\r, 0) -- (\r, 7*\r, 0) -- cycle;
    	\draw [zero] (3*\r, 6*\r, 0) -- (4*\r, 6*\r, 0) -- (4*\r, 7*\r, 0) -- (3*\r, 7*\r, 0) -- cycle;
    	\draw [zero] (8*\r, 6*\r, 0) -- (9*\r, 6*\r, 0) -- (9*\r, 7*\r, 0) -- (8*\r, 7*\r, 0) -- cycle;
    	\draw [one] (\r, 0, 0) -- (2*\r, 0, 0) -- (2*\r, \r, 0) -- (\r, \r, 0) -- cycle;
    	\draw [one] (3*\r, 0, 0) -- (4*\r, 0, 0) -- (4*\r, \r, 0) -- (3*\r, \r, 0) -- cycle;
    	\draw [one] (8*\r, 0, 0) -- (9*\r, 0, 0) -- (9*\r, \r, 0) -- (8*\r, \r, 0) -- cycle;
    	\draw [one] (\r, 5*\r, 0) -- (2*\r, 5*\r, 0) -- (2*\r, 6*\r, 0) -- (\r, 6*\r, 0) -- cycle;
    	\draw [one] (3*\r, 5*\r, 0) -- (4*\r, 5*\r, 0) -- (4*\r, 6*\r, 0) -- (3*\r, 6*\r, 0) -- cycle;
    	\draw [one] (8*\r, 5*\r, 0) -- (9*\r, 5*\r, 0) -- (9*\r, 6*\r, 0) -- (8*\r, 6*\r, 0) -- cycle;
    	\draw [one] (\r, 7*\r, 0) -- (2*\r, 7*\r, 0) -- (2*\r, 8*\r, 0) -- (\r, 8*\r, 0) -- cycle;
    	\draw [one] (3*\r, 7*\r, 0) -- (4*\r, 7*\r, 0) -- (4*\r, 8*\r, 0) -- (3*\r, 8*\r, 0) -- cycle;
    	\draw [one] (8*\r, 7*\r, 0) -- (9*\r, 7*\r, 0) -- (9*\r, 8*\r, 0) -- (8*\r, 8*\r, 0) -- cycle;
    	\draw [one] (0, \r, 0) -- (\r, \r, 0) -- (\r, 2*\r, 0) -- (0, 2*\r, 0) -- cycle;
    	\draw [one] (5*\r, \r, 0) -- (6*\r, \r, 0) -- (6*\r, 2*\r, 0) -- (5*\r, 2*\r, 0) -- cycle;
    	\draw [one] (7*\r, \r, 0) -- (8*\r, \r, 0) -- (8*\r, 2*\r, 0) -- (7*\r, 2*\r, 0) -- cycle;	
    	\draw [one] (0, 3*\r, 0) -- (\r, 3*\r, 0) -- (\r, 4*\r, 0) -- (0, 4*\r, 0) -- cycle;
    	\draw [one] (5*\r, 3*\r, 0) -- (6*\r, 3*\r, 0) -- (6*\r, 4*\r, 0) -- (5*\r, 4*\r, 0) -- cycle;
    	\draw [one] (7*\r, 3*\r, 0) -- (8*\r, 3*\r, 0) -- (8*\r, 4*\r, 0) -- (7*\r, 4*\r, 0) -- cycle;	
    	\draw [one] (0, 8*\r, 0) -- (\r, 8*\r, 0) -- (\r, 9*\r, 0) -- (0, 9*\r, 0) -- cycle;
    	\draw [one] (5*\r, 8*\r, 0) -- (6*\r, 8*\r, 0) -- (6*\r, 9*\r, 0) -- (5*\r, 9*\r, 0) -- cycle;
    	\draw [one] (7*\r, 8*\r, 0) -- (8*\r, 8*\r, 0) -- (8*\r, 9*\r, 0) -- (7*\r, 9*\r, 0) -- cycle;
    	\draw [one] (2*\r, 2*\r, 0) -- (3*\r, 2*\r, 0) -- (3*\r, 3*\r, 0) -- (2*\r, 3*\r, 0) -- cycle;
    	\draw [one] (4*\r, 2*\r, 0) -- (5*\r, 2*\r, 0) -- (5*\r, 3*\r, 0) -- (4*\r, 3*\r, 0) -- cycle;
    	\draw [one] (6*\r, 2*\r, 0) -- (7*\r, 2*\r, 0) -- (7*\r, 3*\r, 0) -- (6*\r, 3*\r, 0) -- cycle;
    	\draw [one] (2*\r, 4*\r, 0) -- (3*\r, 4*\r, 0) -- (3*\r, 5*\r, 0) -- (2*\r, 5*\r, 0) -- cycle;
    	\draw [one] (4*\r, 4*\r, 0) -- (5*\r, 4*\r, 0) -- (5*\r, 5*\r, 0) -- (4*\r, 5*\r, 0) -- cycle;
    	\draw [one] (6*\r, 4*\r, 0) -- (7*\r, 4*\r, 0) -- (7*\r, 5*\r, 0) -- (6*\r, 5*\r, 0) -- cycle;
    	\draw [one] (2*\r, 6*\r, 0) -- (3*\r, 6*\r, 0) -- (3*\r, 7*\r, 0) -- (2*\r, 7*\r, 0) -- cycle;
    	\draw [one] (4*\r, 6*\r, 0) -- (5*\r, 6*\r, 0) -- (5*\r, 7*\r, 0) -- (4*\r, 7*\r, 0) -- cycle;
    	\draw [one] (6*\r, 6*\r, 0) -- (7*\r, 6*\r, 0) -- (7*\r, 7*\r, 0) -- (6*\r, 7*\r, 0) -- cycle;
    	\draw [two] (2*\r, 0, 0) -- (3*\r, 0, 0) -- (3*\r, \r, 0) -- (2*\r, \r, 0) -- cycle;
    	\draw [two] (4*\r, 0, 0) -- (5*\r, 0, 0) -- (5*\r, \r, 0) -- (4*\r, \r, 0) -- cycle;
    	\draw [two] (6*\r, 0, 0) -- (7*\r, 0, 0) -- (7*\r, \r, 0) -- (6*\r, \r, 0) -- cycle;
    	\draw [two] (2*\r, 5*\r, 0) -- (3*\r, 5*\r, 0) -- (3*\r, 6*\r, 0) -- (2*\r, 6*\r, 0) -- cycle;
    	\draw [two] (4*\r, 5*\r, 0) -- (5*\r, 5*\r, 0) -- (5*\r, 6*\r, 0) -- (4*\r, 6*\r, 0) -- cycle;
    	\draw [two] (6*\r, 5*\r, 0) -- (7*\r, 5*\r, 0) -- (7*\r, 6*\r, 0) -- (6*\r, 6*\r, 0) -- cycle;
    	\draw [two] (2*\r, 7*\r, 0) -- (3*\r, 7*\r, 0) -- (3*\r, 8*\r, 0) -- (2*\r, 8*\r, 0) -- cycle;
    	\draw [two] (4*\r, 7*\r, 0) -- (5*\r, 7*\r, 0) -- (5*\r, 8*\r, 0) -- (4*\r, 8*\r, 0) -- cycle;
    	\draw [two] (6*\r, 7*\r, 0) -- (7*\r, 7*\r, 0) -- (7*\r, 8*\r, 0) -- (6*\r, 8*\r, 0) -- cycle;
    	\draw [two] (\r, \r, 0) -- (2*\r, \r, 0) -- (2*\r, 2*\r, 0) -- (\r, 2*\r, 0) -- cycle;
    	\draw [two] (3*\r, \r, 0) -- (4*\r, \r, 0) -- (4*\r, 2*\r, 0) -- (3*\r, 2*\r, 0) -- cycle;
    	\draw [two] (8*\r, \r, 0) -- (9*\r, \r, 0) -- (9*\r, 2*\r, 0) -- (8*\r, 2*\r, 0) -- cycle;	
    	\draw [two] (\r, 3*\r, 0) -- (2*\r, 3*\r, 0) -- (2*\r, 4*\r, 0) -- (\r, 4*\r, 0) -- cycle;
    	\draw [two] (3*\r, 3*\r, 0) -- (4*\r, 3*\r, 0) -- (4*\r, 4*\r, 0) -- (3*\r, 4*\r, 0) -- cycle;
    	\draw [two] (8*\r, 3*\r, 0) -- (9*\r, 3*\r, 0) -- (9*\r, 4*\r, 0) -- (8*\r, 4*\r, 0) -- cycle;
    	\draw [two] (\r, 8*\r, 0) -- (2*\r, 8*\r, 0) -- (2*\r, 9*\r, 0) -- (\r, 9*\r, 0) -- cycle;
    	\draw [two] (3*\r, 8*\r, 0) -- (4*\r, 8*\r, 0) -- (4*\r, 9*\r, 0) -- (3*\r, 9*\r, 0) -- cycle;
    	\draw [two] (8*\r, 8*\r, 0) -- (9*\r, 8*\r, 0) -- (9*\r, 9*\r, 0) -- (8*\r, 9*\r, 0) -- cycle;
    	\draw [two] (0, 2*\r, 0) -- (\r, 2*\r, 0) -- (\r, 3*\r, 0) -- (0, 3*\r, 0) -- cycle;
    	\draw [two] (5*\r, 2*\r, 0) -- (6*\r, 2*\r, 0) -- (6*\r, 3*\r, 0) -- (5*\r, 3*\r, 0) -- cycle;
    	\draw [two] (7*\r, 2*\r, 0) -- (8*\r, 2*\r, 0) -- (8*\r, 3*\r, 0) -- (7*\r, 3*\r, 0) -- cycle;
    	\draw [two] (0, 4*\r, 0) -- (\r, 4*\r, 0) -- (\r, 5*\r, 0) -- (0, 5*\r, 0) -- cycle;
    	\draw [two] (5*\r, 4*\r, 0) -- (6*\r, 4*\r, 0) -- (6*\r, 5*\r, 0) -- (5*\r, 5*\r, 0) -- cycle;
    	\draw [two] (7*\r, 4*\r, 0) -- (8*\r, 4*\r, 0) -- (8*\r, 5*\r, 0) -- (7*\r, 5*\r, 0) -- cycle;
    	\draw [two] (0, 6*\r, 0) -- (\r, 6*\r, 0) -- (\r, 7*\r, 0) -- (0, 7*\r, 0) -- cycle;
    	\draw [two] (5*\r, 6*\r, 0) -- (6*\r, 6*\r, 0) -- (6*\r, 7*\r, 0) -- (5*\r, 7*\r, 0) -- cycle;
    	\draw [two] (7*\r, 6*\r, 0) -- (8*\r, 6*\r, 0) -- (8*\r, 7*\r, 0) -- (7*\r, 7*\r, 0) -- cycle;
    	\draw [zeromarked] (\r, \r, \r) -- (2*\r, \r, \r) -- (2*\r, 2*\r, \r) -- (\r, 2*\r, \r) -- cycle;
    	\draw [zeromarked] (2*\r, 2*\r, 2*\r) -- (3*\r, 2*\r, 2*\r) -- (3*\r, 3*\r, 2*\r) -- (2*\r, 3*\r, 2*\r) -- cycle;
    	\draw [zeromarked] (3*\r, 3*\r, 3*\r) -- (4*\r, 3*\r, 3*\r) -- (4*\r, 4*\r, 3*\r) -- (3*\r, 4*\r, 3*\r) -- cycle;
    	\draw [zeromarked] (4*\r, 4*\r, 4*\r) -- (5*\r, 4*\r, 4*\r) -- (5*\r, 5*\r, 4*\r) -- (4*\r, 5*\r, 4*\r) -- cycle;
    	\draw [zeromarked] (5*\r, 5*\r, 5*\r) -- (6*\r, 5*\r, 5*\r) -- (6*\r, 6*\r, 5*\r) -- (5*\r, 6*\r, 5*\r) -- cycle;
    	\draw [zeromarked] (6*\r, 6*\r, 6*\r) -- (7*\r, 6*\r, 6*\r) -- (7*\r, 7*\r, 6*\r) -- (6*\r, 7*\r, 6*\r) -- cycle;
    	\draw [zeromarked] (7*\r, 7*\r, 7*\r) -- (8*\r, 7*\r, 7*\r) -- (8*\r, 8*\r, 7*\r) -- (7*\r, 8*\r, 7*\r) -- cycle;
    	\draw [zeromarked] (8*\r, 8*\r, 8*\r) -- (9*\r, 8*\r, 8*\r) -- (9*\r, 9*\r, 8*\r) -- (8*\r, 9*\r, 8*\r) -- cycle;
    \end{tikzpicture}
}
\caption{A monochromatic arithmetic progression of difference $d=3^{2n}+3^n+1$ in the word $\varphi^{3n}(0)$, with $n=2$. }
\label{fig: cubes}
\end{centering}
\end{figure}
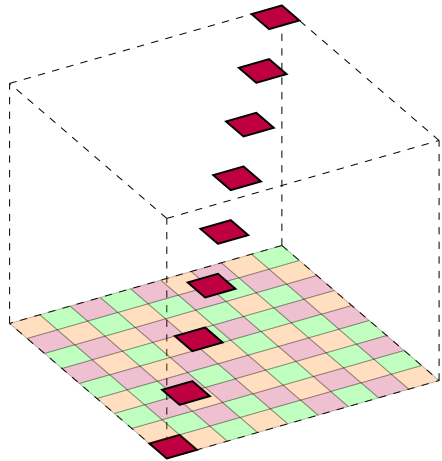
\end{example}

\begin{figure}[h!]
\centering
\begin{tikzpicture}
	\begin{axis}
	[
		width=\linewidth,
		height=0.3\textheight,
		line width=.3pt,
		axis lines*=left,
		xlabel style={at=(current axis.right of origin),anchor=west},
		xlabel=$d$,
		xmin=0,
		xmax=2210,
		xticklabel style={font=\footnotesize},
		xtick align=outside,
		xtick={81, 243, 486, 729, 1458, 2187},
		xticklabels={$3^4$, $3^5$, $2 \cdot 3^5$, $3^6$, $2 \cdot 3^6$, $3^7$},
		ylabel style={at=(current axis.above origin), anchor=south, rotate=-90},
		ylabel=$A$,
		ymin=0,
		ymax=2210,
		yticklabel style={font=\footnotesize},
		ytick align=outside,
		ytick={81, 243, 729, 2187},
		yticklabels={$3^4$, $3^5$, $3^6$, $3^7$}
	]
	\addplot
	[
		blue!50!white,
		line width=.2pt,
		mark options={scale=.1,fill=red,draw=red},
		mark=*
	]
	coordinates{(1,2)(2,3)(3,2)(4,4)(5,3)(6,3)(7,3)(8,9)(9,2)(10,3)(11,4)(12,4)(13,9)(14,5)(15,3)(16,7)(17,3)(18,3)(19,4)(20,3)(21,3)(22,5)(23,4)(24,9)(25,3)(26,33)(27,2)(28,3)(29,4)(30,3)(31,9)(32,6)(33,4)(34,5)(35,4)(36,4)(37,9)(38,4)(39,9)(40,5)(41,6)(42,5)(43,5)(44,5)(45,3)(46,4)(47,6)(48,7)(49,5)(50,6)(51,3)(52,19)(53,3)(54,3)(55,5)(56,5)(57,4)(58,4)(59,6)(60,3)(61,9)(62,6)(63,3)(64,10)(65,8)(66,5)(67,6)(68,6)(69,4)(70,6)(71,4)(72,9)(73,4)(74,6)(75,3)(76,6)(77,4)(78,33)(79,3)(80,81)(81,2)(82,4)(83,5)(84,3)(85,8)(86,5)(87,4)(88,6)(89,5)(90,3)(91,13)(92,5)(93,9)(94,5)(95,5)(96,6)(97,5)(98,6)(99,4)(100,4)(101,9)(102,5)(103,6)(104,12)(105,4)(106,8)(107,4)(108,4)(109,9)(110,5)(111,9)(112,6)(113,7)(114,4)(115,6)(116,10)(117,9)(118,6)(119,6)(120,5)(121,8)(122,6)(123,6)(124,6)(125,7)(126,5)(127,5)(128,6)(129,5)(130,10)(131,5)(132,5)(133,6)(134,8)(135,3)(136,5)(137,5)(138,4)(139,6)(140,5)(141,6)(142,6)(143,8)(144,7)(145,6)(146,7)(147,5)(148,8)(149,5)(150,6)(151,6)(152,11)(153,3)(154,7)(155,5)(156,19)(157,7)(158,7)(159,3)(160,43)(161,3)(162,3)(163,5)(164,5)(165,5)(166,5)(167,5)(168,5)(169,10)(170,6)(171,4)(172,6)(173,9)(174,4)(175,7)(176,5)(177,6)(178,6)(179,6)(180,3)(181,8)(182,9)(183,9)(184,6)(185,9)(186,6)(187,8)(188,5)(189,3)(190,6)(191,7)(192,10)(193,7)(194,9)(195,8)(196,7)(197,8)(198,5)(199,5)(200,7)(201,6)(202,6)(203,7)(204,6)(205,6)(206,8)(207,4)(208,7)(209,6)(210,6)(211,6)(212,6)(213,4)(214,7)(215,4)(216,9)(217,5)(218,7)(219,4)(220,6)(221,9)(222,6)(223,6)(224,8)(225,3)(226,10)(227,7)(228,6)(229,8)(230,10)(231,4)(232,7)(233,5)(234,33)(235,4)(236,5)(237,3)(238,6)(239,4)(240,81)(241,3)(242,243)(243,2)(244,3)(245,5)(246,4)(247,9)(248,6)(249,5)(250,6)(251,5)(252,3)(253,13)(254,5)(255,8)(256,6)(257,6)(258,5)(259,7)(260,8)(261,4)(262,5)(263,8)(264,6)(265,7)(266,8)(267,5)(268,11)(269,5)(270,3)(271,13)(272,6)(273,13)(274,6)(275,8)(276,5)(277,7)(278,7)(279,9)(280,6)(281,8)(282,5)(283,7)(284,8)(285,5)(286,10)(287,8)(288,6)(289,7)(290,7)(291,5)(292,8)(293,6)(294,6)(295,7)(296,5)(297,4)(298,5)(299,10)(300,4)(301,8)(302,7)(303,9)(304,9)(305,7)(306,5)(307,7)(308,6)(309,6)(310,6)(311,9)(312,12)(313,6)(314,8)(315,4)(316,6)(317,7)(318,8)(319,6)(320,24)(321,4)(322,7)(323,4)(324,4)(325,9)(326,6)(327,9)(328,9)(329,6)(330,5)(331,7)(332,7)(333,9)(334,5)(335,7)(336,6)(337,8)(338,9)(339,7)(340,9)(341,7)(342,4)(343,12)(344,8)(345,6)(346,10)(347,7)(348,10)(349,7)(350,10)(351,9)(352,6)(353,7)(354,6)(355,9)(356,10)(357,6)(358,6)(359,10)(360,5)(361,8)(362,6)(363,8)(364,9)(365,6)(366,6)(367,7)(368,8)(369,6)(370,6)(371,6)(372,6)(373,6)(374,6)(375,7)(376,10)(377,11)(378,5)(379,6)(380,5)(381,5)(382,8)(383,7)(384,6)(385,7)(386,9)(387,5)(388,9)(389,8)(390,10)(391,7)(392,6)(393,5)(394,8)(395,8)(396,5)(397,8)(398,8)(399,6)(400,19)(401,7)(402,8)(403,7)(404,6)(405,3)(406,5)(407,5)(408,5)(409,7)(410,7)(411,5)(412,7)(413,9)(414,4)(415,8)(416,10)(417,6)(418,8)(419,6)(420,5)(421,9)(422,7)(423,6)(424,7)(425,7)(426,6)(427,13)(428,8)(429,8)(430,7)(431,8)(432,7)(433,6)(434,7)(435,6)(436,9)(437,7)(438,7)(439,8)(440,9)(441,5)(442,10)(443,10)(444,8)(445,9)(446,8)(447,5)(448,8)(449,10)(450,6)(451,6)(452,10)(453,6)(454,7)(455,7)(456,11)(457,7)(458,11)(459,3)(460,9)(461,6)(462,7)(463,8)(464,7)(465,5)(466,8)(467,6)(468,19)(469,7)(470,7)(471,7)(472,11)(473,8)(474,7)(475,7)(476,11)(477,3)(478,7)(479,8)(480,43)(481,7)(482,7)(483,3)(484,124)(485,3)(486,3)(487,5)(488,5)(489,5)(490,5)(491,5)(492,5)(493,9)(494,9)(495,5)(496,6)(497,9)(498,5)(499,7)(500,6)(501,5)(502,8)(503,6)(504,5)(505,11)(506,9)(507,10)(508,10)(509,8)(510,6)(511,7)(512,8)(513,4)(514,5)(515,11)(516,6)(517,9)(518,8)(519,9)(520,9)(521,9)(522,4)(523,11)(524,8)(525,7)(526,6)(527,7)(528,5)(529,7)(530,7)(531,6)(532,8)(533,8)(534,6)(535,8)(536,10)(537,6)(538,7)(539,7)(540,3)(541,7)(542,9)(543,8)(544,8)(545,8)(546,9)(547,7)(548,7)(549,9)(550,7)(551,10)(552,6)(553,8)(554,7)(555,9)(556,7)(557,8)(558,6)(559,8)(560,15)(561,8)(562,7)(563,8)(564,5)(565,8)(566,5)(567,3)(568,6)(569,9)(570,6)(571,7)(572,8)(573,7)(574,10)(575,7)(576,10)(577,8)(578,8)(579,7)(580,13)(581,10)(582,9)(583,9)(584,9)(585,8)(586,7)(587,10)(588,7)(589,13)(590,7)(591,8)(592,7)(593,7)(594,5)(595,7)(596,7)(597,5)(598,13)(599,10)(600,7)(601,10)(602,9)(603,6)(604,12)(605,9)(606,6)(607,7)(608,7)(609,7)(610,6)(611,13)(612,6)(613,8)(614,7)(615,6)(616,7)(617,8)(618,8)(619,9)(620,6)(621,4)(622,8)(623,7)(624,7)(625,7)(626,7)(627,6)(628,7)(629,9)(630,6)(631,7)(632,9)(633,6)(634,9)(635,9)(636,6)(637,10)(638,9)(639,4)(640,13)(641,7)(642,7)(643,9)(644,6)(645,4)(646,8)(647,4)(648,9)(649,5)(650,8)(651,5)(652,6)(653,8)(654,7)(655,10)(656,9)(657,4)(658,7)(659,9)(660,6)(661,19)(662,8)(663,9)(664,12)(665,9)(666,6)(667,8)(668,8)(669,6)(670,6)(671,8)(672,8)(673,7)(674,8)(675,3)(676,34)(677,7)(678,10)(679,9)(680,11)(681,7)(682,6)(683,7)(684,6)(685,8)(686,10)(687,8)(688,6)(689,14)(690,10)(691,9)(692,7)(693,4)(694,10)(695,7)(696,7)(697,10)(698,6)(699,5)(700,7)(701,5)(702,33)(703,5)(704,7)(705,4)(706,6)(707,7)(708,5)(709,7)(710,9)(711,3)(712,10)(713,8)(714,6)(715,10)(716,10)(717,4)(718,7)(719,5)(720,81)(721,4)(722,5)(723,3)(724,6)(725,4)(726,243)(727,3)(728,735)(729,2)(730,3)(731,5)(732,3)(733,9)(734,6)(735,5)(736,6)(737,5)(738,4)(739,13)(740,5)(741,9)(742,6)(743,6)(744,6)(745,6)(746,9)(747,5)(748,5)(749,9)(750,6)(751,7)(752,8)(753,5)(754,18)(755,5)(756,3)(757,31)(758,5)(759,13)(760,6)(761,6)(762,5)(763,8)(764,9)(765,8)(766,6)(767,13)(768,6)(769,7)(770,7)(771,6)(772,9)(773,8)(774,5)(775,8)(776,7)(777,7)(778,8)(779,7)(780,8)(781,8)(782,7)(783,4)(784,5)(785,9)(786,5)(787,10)(788,7)(789,8)(790,10)(791,7)(792,6)(793,9)(794,7)(795,7)(796,8)(797,7)(798,8)(799,6)(800,15)(801,5)(802,7)(803,6)(804,11)(805,8)(806,12)(807,5)(808,11)(809,5)(810,3)(811,13)(812,6)(813,13)(814,6)(815,8)(816,6)(817,10)(818,8)(819,13)(820,7)(821,10)(822,6)(823,8)(824,8)(825,8)(826,9)(827,10)(828,5)(829,13)(830,9)(831,7)(832,13)(833,8)(834,7)(835,9)(836,6)(837,9)(838,6)(839,8)(840,6)(841,8)(842,8)(843,8)(844,10)(845,13)(846,5)(847,9)(848,9)(849,7)(850,9)(851,8)(852,8)(853,9)(854,9)(855,5)(856,9)(857,7)(858,10)(859,11)(860,8)(861,8)(862,11)(863,8)(864,6)(865,6)(866,10)(867,7)(868,10)(869,11)(870,7)(871,10)(872,10)(873,5)(874,9)(875,10)(876,8)(877,9)(878,10)(879,6)(880,13)(881,8)(882,6)(883,8)(884,10)(885,7)(886,9)(887,6)(888,5)(889,7)(890,7)(891,4)(892,5)(893,10)(894,5)(895,9)(896,8)(897,10)(898,10)(899,8)(900,4)(901,13)(902,7)(903,8)(904,7)(905,9)(906,7)(907,9)(908,12)(909,9)(910,10)(911,7)(912,9)(913,7)(914,10)(915,7)(916,8)(917,7)(918,5)(919,8)(920,8)(921,7)(922,8)(923,11)(924,6)(925,11)(926,7)(927,6)(928,11)(929,9)(930,6)(931,7)(932,12)(933,9)(934,9)(935,7)(936,12)(937,7)(938,10)(939,6)(940,8)(941,9)(942,8)(943,7)(944,8)(945,4)(946,7)(947,9)(948,6)(949,9)(950,9)(951,7)(952,9)(953,7)(954,8)(955,7)(956,7)(957,6)(958,8)(959,8)(960,24)(961,6)(962,9)(963,4)(964,5)(965,8)(966,7)(967,6)(968,64)(969,4)(970,7)(971,4)(972,4)(973,9)(974,5)(975,9)(976,9)(977,7)(978,6)(979,8)(980,7)(981,9)(982,5)(983,9)(984,9)(985,8)(986,7)(987,6)(988,8)(989,10)(990,5)(991,10)(992,8)(993,7)(994,12)(995,7)(996,7)(997,8)(998,8)(999,9)(1000,6)(1001,13)(1002,5)(1003,7)(1004,9)(1005,7)(1006,8)(1007,8)(1008,6)(1009,7)(1010,7)(1011,8)(1012,9)(1013,7)(1014,9)(1015,8)(1016,8)(1017,7)(1018,8)(1019,8)(1020,9)(1021,11)(1022,13)(1023,7)(1024,9)(1025,9)(1026,4)(1027,13)(1028,8)(1029,12)(1030,7)(1031,10)(1032,8)(1033,8)(1034,7)(1035,6)(1036,8)(1037,10)(1038,10)(1039,8)(1040,12)(1041,7)(1042,8)(1043,7)(1044,10)(1045,15)(1046,8)(1047,7)(1048,8)(1049,9)(1050,10)(1051,7)(1052,10)(1053,9)(1054,7)(1055,7)(1056,6)(1057,9)(1058,7)(1059,7)(1060,8)(1061,13)(1062,6)(1063,8)(1064,8)(1065,9)(1066,9)(1067,7)(1068,10)(1069,10)(1070,6)(1071,6)(1072,7)(1073,10)(1074,6)(1075,9)(1076,9)(1077,10)(1078,8)(1079,15)(1080,5)(1081,8)(1082,8)(1083,8)(1084,9)(1085,9)(1086,6)(1087,6)(1088,10)(1089,8)(1090,9)(1091,6)(1092,9)(1093,8)(1094,6)(1095,6)(1096,6)(1097,8)(1098,6)(1099,5)(1100,6)(1101,7)(1102,10)(1103,6)(1104,8)(1105,15)(1106,9)(1107,6)(1108,7)(1109,6)(1110,6)(1111,8)(1112,8)(1113,6)(1114,9)(1115,10)(1116,6)(1117,10)(1118,10)(1119,6)(1120,11)(1121,7)(1122,6)(1123,8)(1124,10)(1125,7)(1126,7)(1127,7)(1128,10)(1129,9)(1130,8)(1131,11)(1132,12)(1133,8)(1134,5)(1135,6)(1136,6)(1137,6)(1138,7)(1139,6)(1140,5)(1141,9)(1142,8)(1143,5)(1144,9)(1145,13)(1146,8)(1147,7)(1148,7)(1149,7)(1150,11)(1151,10)(1152,6)(1153,9)(1154,10)(1155,7)(1156,8)(1157,11)(1158,9)(1159,11)(1160,9)(1161,5)(1162,9)(1163,9)(1164,9)(1165,13)(1166,11)(1167,8)(1168,8)(1169,8)(1170,10)(1171,9)(1172,10)(1173,7)(1174,9)(1175,7)(1176,6)(1177,9)(1178,9)(1179,5)(1180,8)(1181,7)(1182,8)(1183,8)(1184,10)(1185,8)(1186,10)(1187,8)(1188,5)(1189,10)(1190,7)(1191,8)(1192,9)(1193,10)(1194,8)(1195,10)(1196,9)(1197,6)(1198,8)(1199,8)(1200,19)(1201,10)(1202,9)(1203,7)(1204,8)(1205,8)(1206,8)(1207,9)(1208,7)(1209,7)(1210,52)(1211,7)(1212,6)(1213,7)(1214,6)(1215,3)(1216,5)(1217,6)(1218,5)(1219,7)(1220,7)(1221,5)(1222,10)(1223,8)(1224,5)(1225,7)(1226,8)(1227,7)(1228,7)(1229,6)(1230,7)(1231,8)(1232,8)(1233,5)(1234,7)(1235,11)(1236,7)(1237,11)(1238,9)(1239,9)(1240,9)(1241,10)(1242,4)(1243,7)(1244,10)(1245,8)(1246,9)(1247,13)(1248,10)(1249,10)(1250,10)(1251,6)(1252,8)(1253,10)(1254,8)(1255,8)(1256,6)(1257,6)(1258,10)(1259,10)(1260,5)(1261,9)(1262,10)(1263,9)(1264,8)(1265,10)(1266,7)(1267,9)(1268,7)(1269,6)(1270,10)(1271,9)(1272,7)(1273,9)(1274,11)(1275,7)(1276,7)(1277,9)(1278,6)(1279,9)(1280,10)(1281,13)(1282,8)(1283,11)(1284,8)(1285,10)(1286,8)(1287,8)(1288,10)(1289,8)(1290,7)(1291,11)(1292,7)(1293,8)(1294,6)(1295,8)(1296,7)(1297,6)(1298,7)(1299,6)(1300,9)(1301,8)(1302,7)(1303,10)(1304,10)(1305,6)(1306,9)(1307,8)(1308,9)(1309,11)(1310,11)(1311,7)(1312,10)(1313,10)(1314,7)(1315,9)(1316,10)(1317,8)(1318,9)(1319,12)(1320,9)(1321,11)(1322,10)(1323,5)(1324,8)(1325,8)(1326,10)(1327,10)(1328,9)(1329,10)(1330,10)(1331,21)(1332,8)(1333,7)(1334,7)(1335,9)(1336,8)(1337,8)(1338,8)(1339,11)(1340,9)(1341,5)(1342,10)(1343,9)(1344,8)(1345,12)(1346,7)(1347,10)(1348,9)(1349,7)(1350,6)(1351,7)(1352,18)(1353,6)(1354,9)(1355,8)(1356,10)(1357,9)(1358,8)(1359,6)(1360,13)(1361,9)(1362,7)(1363,8)(1364,9)(1365,7)(1366,8)(1367,10)(1368,11)(1369,9)(1370,7)(1371,7)(1372,9)(1373,10)(1374,11)(1375,8)(1376,10)(1377,3)(1378,12)(1379,6)(1380,9)(1381,9)(1382,10)(1383,6)(1384,10)(1385,9)(1386,7)(1387,9)(1388,8)(1389,8)(1390,9)(1391,11)(1392,7)(1393,10)(1394,9)(1395,5)(1396,9)(1397,10)(1398,8)(1399,10)(1400,9)(1401,6)(1402,9)(1403,6)(1404,19)(1405,8)(1406,11)(1407,7)(1408,11)(1409,8)(1410,7)(1411,11)(1412,8)(1413,7)(1414,10)(1415,9)(1416,11)(1417,15)(1418,8)(1419,8)(1420,9)(1421,8)(1422,7)(1423,7)(1424,9)(1425,7)(1426,8)(1427,10)(1428,11)(1429,8)(1430,25)(1431,3)(1432,8)(1433,8)(1434,7)(1435,10)(1436,9)(1437,8)(1438,7)(1439,8)(1440,43)(1441,8)(1442,6)(1443,7)(1444,12)(1445,8)(1446,7)(1447,7)(1448,11)(1449,3)(1450,8)(1451,8)(1452,124)(1453,7)(1454,7)(1455,3)(1456,370)(1457,3)(1458,3)(1459,5)(1460,5)(1461,5)(1462,5)(1463,6)(1464,5)(1465,11)(1466,7)(1467,5)(1468,7)(1469,11)(1470,5)(1471,7)(1472,6)(1473,5)(1474,7)(1475,6)(1476,5)(1477,9)(1478,9)(1479,9)(1480,8)(1481,9)(1482,9)(1483,8)(1484,8)(1485,5)(1486,5)(1487,11)(1488,6)(1489,9)(1490,7)(1491,9)(1492,9)(1493,8)(1494,5)(1495,11)(1496,8)(1497,7)(1498,8)(1499,6)(1500,6)(1501,10)(1502,9)(1503,5)(1504,7)(1505,8)(1506,8)(1507,8)(1508,14)(1509,6)(1510,10)(1511,9)(1512,5)(1513,11)(1514,18)(1515,11)(1516,7)(1517,8)(1518,9)(1519,9)(1520,15)(1521,10)(1522,7)(1523,9)(1524,10)(1525,11)(1526,7)(1527,8)(1528,10)(1529,11)(1530,6)(1531,12)(1532,12)(1533,7)(1534,11)(1535,7)(1536,8)(1537,8)(1538,8)(1539,4)(1540,5)(1541,12)(1542,5)(1543,11)(1544,9)(1545,11)(1546,8)(1547,12)(1548,6)(1549,16)(1550,8)(1551,9)(1552,9)(1553,10)(1554,8)(1555,7)(1556,9)(1557,9)(1558,8)(1559,9)(1560,9)(1561,9)(1562,11)(1563,9)(1564,11)(1565,8)(1566,4)(1567,9)(1568,8)(1569,11)(1570,8)(1571,11)(1572,8)(1573,10)(1574,9)(1575,7)(1576,8)(1577,11)(1578,6)(1579,7)(1580,7)(1581,7)(1582,8)(1583,9)(1584,5)(1585,9)(1586,11)(1587,7)(1588,9)(1589,8)(1590,7)(1591,9)(1592,7)(1593,6)(1594,7)(1595,10)(1596,8)(1597,8)(1598,8)(1599,8)(1600,11)(1601,9)(1602,6)(1603,9)(1604,9)(1605,8)(1606,9)(1607,9)(1608,10)(1609,9)(1610,11)(1611,6)(1612,9)(1613,9)(1614,7)(1615,9)(1616,10)(1617,7)(1618,7)(1619,7)(1620,3)(1621,8)(1622,9)(1623,7)(1624,8)(1625,10)(1626,9)(1627,9)(1628,8)(1629,8)(1630,8)(1631,9)(1632,8)(1633,10)(1634,10)(1635,8)(1636,10)(1637,10)(1638,9)(1639,8)(1640,10)(1641,7)(1642,8)(1643,8)(1644,7)(1645,8)(1646,8)(1647,9)(1648,7)(1649,9)(1650,7)(1651,9)(1652,9)(1653,10)(1654,8)(1655,12)(1656,6)(1657,11)(1658,14)(1659,8)(1660,10)(1661,9)(1662,7)(1663,9)(1664,9)(1665,9)(1666,9)(1667,11)(1668,7)(1669,9)(1670,10)(1671,8)(1672,7)(1673,9)(1674,6)(1675,8)(1676,13)(1677,8)(1678,11)(1679,9)(1680,15)(1681,8)(1682,9)(1683,8)(1684,8)(1685,9)(1686,7)(1687,10)(1688,7)(1689,8)(1690,11)(1691,9)(1692,5)(1693,7)(1694,39)(1695,8)(1696,7)(1697,8)(1698,5)(1699,9)(1700,6)(1701,3)(1702,10)(1703,9)(1704,6)(1705,8)(1706,7)(1707,9)(1708,9)(1709,7)(1710,6)(1711,8)(1712,9)(1713,7)(1714,8)(1715,8)(1716,8)(1717,10)(1718,11)(1719,7)(1720,8)(1721,9)(1722,10)(1723,10)(1724,9)(1725,7)(1726,9)(1727,7)(1728,10)(1729,8)(1730,8)(1731,8)(1732,8)(1733,8)(1734,8)(1735,8)(1736,11)(1737,7)(1738,10)(1739,9)(1740,13)(1741,10)(1742,9)(1743,10)(1744,11)(1745,8)(1746,9)(1747,9)(1748,8)(1749,9)(1750,8)(1751,9)(1752,9)(1753,8)(1754,13)(1755,8)(1756,6)(1757,9)(1758,7)(1759,10)(1760,11)(1761,10)(1762,9)(1763,11)(1764,7)(1765,9)(1766,10)(1767,13)(1768,9)(1769,8)(1770,7)(1771,8)(1772,10)(1773,8)(1774,8)(1775,9)(1776,7)(1777,10)(1778,9)(1779,7)(1780,7)(1781,7)(1782,5)(1783,7)(1784,11)(1785,7)(1786,9)(1787,9)(1788,7)(1789,10)(1790,10)(1791,5)(1792,9)(1793,11)(1794,13)(1795,8)(1796,10)(1797,10)(1798,10)(1799,9)(1800,7)(1801,9)(1802,10)(1803,10)(1804,9)(1805,10)(1806,9)(1807,11)(1808,7)(1809,6)(1810,11)(1811,8)(1812,12)(1813,8)(1814,10)(1815,9)(1816,8)(1817,8)(1818,6)(1819,11)(1820,9)(1821,7)(1822,7)(1823,7)(1824,7)(1825,7)(1826,8)(1827,7)(1828,8)(1829,10)(1830,6)(1831,11)(1832,8)(1833,13)(1834,8)(1835,15)(1836,6)(1837,7)(1838,8)(1839,8)(1840,10)(1841,9)(1842,7)(1843,9)(1844,8)(1845,6)(1846,10)(1847,8)(1848,7)(1849,10)(1850,8)(1851,8)(1852,8)(1853,10)(1854,8)(1855,9)(1856,9)(1857,9)(1858,11)(1859,9)(1860,6)(1861,9)(1862,6)(1863,4)(1864,8)(1865,9)(1866,8)(1867,7)(1868,10)(1869,7)(1870,11)(1871,8)(1872,7)(1873,7)(1874,10)(1875,7)(1876,8)(1877,11)(1878,7)(1879,11)(1880,9)(1881,6)(1882,10)(1883,9)(1884,7)(1885,11)(1886,8)(1887,9)(1888,8)(1889,9)(1890,6)(1891,6)(1892,10)(1893,7)(1894,10)(1895,10)(1896,9)(1897,11)(1898,12)(1899,6)(1900,10)(1901,8)(1902,9)(1903,9)(1904,7)(1905,9)(1906,8)(1907,11)(1908,6)(1909,9)(1910,9)(1911,10)(1912,8)(1913,11)(1914,9)(1915,11)(1916,8)(1917,4)(1918,13)(1919,9)(1920,13)(1921,8)(1922,9)(1923,7)(1924,9)(1925,10)(1926,7)(1927,9)(1928,9)(1929,9)(1930,9)(1931,15)(1932,6)(1933,12)(1934,8)(1935,4)(1936,34)(1937,10)(1938,8)(1939,9)(1940,6)(1941,4)(1942,8)(1943,4)(1944,9)(1945,5)(1946,7)(1947,5)(1948,6)(1949,8)(1950,8)(1951,7)(1952,11)(1953,5)(1954,8)(1955,9)(1956,6)(1957,9)(1958,9)(1959,8)(1960,11)(1961,7)(1962,7)(1963,10)(1964,9)(1965,10)(1966,8)(1967,8)(1968,9)(1969,7)(1970,10)(1971,4)(1972,7)(1973,10)(1974,7)(1975,10)(1976,10)(1977,9)(1978,10)(1979,9)(1980,6)(1981,12)(1982,10)(1983,19)(1984,8)(1985,8)(1986,8)(1987,9)(1988,11)(1989,9)(1990,10)(1991,10)(1992,12)(1993,10)(1994,11)(1995,9)(1996,10)(1997,10)(1998,6)(1999,9)(2000,10)(2001,8)(2002,9)(2003,8)(2004,8)(2005,8)(2006,8)(2007,6)(2008,9)(2009,11)(2010,6)(2011,8)(2012,10)(2013,8)(2014,8)(2015,11)(2016,8)(2017,10)(2018,10)(2019,7)(2020,7)(2021,9)(2022,8)(2023,8)(2024,8)(2025,3)(2026,18)(2027,8)(2028,34)(2029,9)(2030,10)(2031,7)(2032,10)(2033,7)(2034,10)(2035,9)(2036,10)(2037,9)(2038,8)(2039,12)(2040,11)(2041,13)(2042,11)(2043,7)(2044,9)(2045,11)(2046,6)(2047,9)(2048,11)(2049,7)(2050,8)(2051,7)(2052,6)(2053,8)(2054,11)(2055,8)(2056,9)(2057,11)(2058,10)(2059,9)(2060,13)(2061,8)(2062,9)(2063,11)(2064,6)(2065,8)(2066,10)(2067,14)(2068,7)(2069,10)(2070,10)(2071,7)(2072,8)(2073,9)(2074,10)(2075,9)(2076,7)(2077,9)(2078,9)(2079,4)(2080,14)(2081,5)(2082,10)(2083,9)(2084,6)(2085,7)(2086,9)(2087,8)(2088,7)(2089,7)(2090,18)(2091,10)(2092,11)(2093,15)(2094,6)(2095,9)(2096,9)(2097,5)(2098,13)(2099,7)(2100,7)(2101,7)(2102,8)(2103,5)(2104,9)(2105,5)(2106,33)(2107,5)(2108,8)(2109,5)(2110,7)(2111,8)(2112,7)(2113,10)(2114,12)(2115,4)(2116,7)(2117,10)(2118,6)(2119,15)(2120,8)(2121,7)(2122,12)(2123,7)(2124,5)(2125,9)(2126,8)(2127,7)(2128,7)(2129,10)(2130,9)(2131,9)(2132,12)(2133,3)(2134,30)(2135,7)(2136,10)(2137,9)(2138,11)(2139,8)(2140,7)(2141,10)(2142,6)(2143,8)(2144,9)(2145,10)(2146,6)(2147,15)(2148,10)(2149,8)(2150,8)(2151,4)(2152,9)(2153,7)(2154,7)(2155,8)(2156,7)(2157,5)(2158,12)(2159,5)(2160,81)(2161,5)(2162,8)(2163,4)(2164,6)(2165,7)(2166,5)(2167,8)(2168,8)(2169,3)(2170,18)(2171,8)(2172,6)(2173,8)(2174,10)(2175,4)(2176,7)(2177,5)(2178,243)(2179,4)(2180,6)(2181,3)(2182,6)(2183,4)(2184,735)(2185,3)(2186,2187)(2187,2)(2188,4)(2189,5)(2190,3)(2191,9)(2192,5)(2193,5)(2194,6)(2195,5)(2196,3)(2197,13)(2198,5)(2199,9)(2200,6)};
	\node[anchor=south, font=\tiny] at (axis cs: 100,81) {$\;\;\left(3^4-1,3^4\right)$};
	\node[anchor=south, font=\tiny] at (axis cs: 242,243) {$\left(3^5-1,3^5\right)$};
	\node[anchor=south, font=\tiny] at (axis cs: 728,735) {$\left(3^6-1,3^6+6\right)$};
	\node[anchor=south, font=\tiny] at (axis cs: 1456,370) {$\left(2(3^6-1),\frac{1}{2}(3^6+11)\right)$};
	\node[anchor=south east, font=\tiny] at (axis cs: 2200,2040) {$\left(3^7-1,3^7\right)$};
	\end{axis}
\end{tikzpicture}
\caption{Plot of the exact values of $A(d)$ for $1 \leqslant d \leqslant 2200$ for the ternary Thue--Morse sequence.}
\label{fig:plot-A(d)-ternary}
\end{figure}
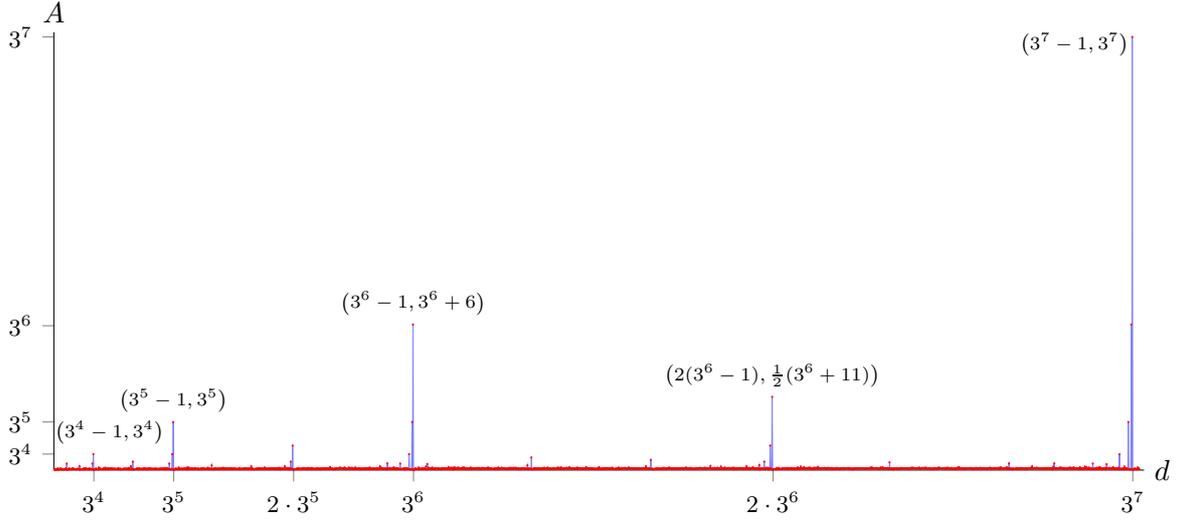

\subsubsection{Other bijective substitutions}

Below, we give more examples of bijective substitutions and the corresponding explicit bounds from the results in the previous sections. 
\begin{example}[$G=A^{ }_4$]
\label{ex:G=A4}
Consider the following substitution of length $L=3$,
\[
\varrho\colon\,\, \begin{matrix}
0\mapsto 011&\\ 
1\mapsto 120&\\
2\mapsto 203&\\
3\mapsto 332&
\end{matrix},
\]
with $\varrho^{ }_0 = \text{id}$, $\varrho^{ }_1 = (012)$ and $\varrho^{ }_2 = (01)(23)$. The group $G$ generated by $\varrho^{ }_1$ and $\varrho^{ }_2$ is the alternating group $A^{ }_4$, which consists of $\vert G\vert=12$ elements.
Proposition~\ref{prop:genbijective} then implies that
$A(d)\geqslant3^k\geqslant d^{1/11}$, for all $k\in\mathbb{N}$ and all $d=(3^{12k}-1)/(3^k-1)$. 
However, Remark~\ref{remark:genbijective} shows that, instead of the group order $\vert G\vert$, we can use the least common multiple of $\vert\varrho^{ }_1\vert = 3$ and $\vert\varrho^{ }_2\vert = 2$, which is $6$.
Consequently, $A(d)\geqslant3^k\geqslant d^{1/5}$, for all $k\in\mathbb{N}$ and all $d=(3^{6k}-1)/(3^k-1)$. \exend
\end{example}

\begin{example}[Inverse-palindromic Abelian]\label{ex:C3-invpal}
Consider the three-letter substitution 
\[
\varrho\colon\,\, \begin{matrix}
0\mapsto 02010&\\ 
1\mapsto 10121&\\
2\mapsto 21202&
\end{matrix},
\]
with $\varrho^{ }_0=\varrho^{ }_2=\varrho^{ }_4=\text{id},\varrho^{ }_1=(021)$ and $\varrho^{ }_3=(012)=\varrho^{-1}_1$, and hence $\varrho$ is inverse palindromic. The group generated by the columns is $G=C_3$, which is Abelian. It follows from Proposition~\ref{prop:pal} and Remark~\ref{rem:pal} that $A(5^k-1)\geqslant 5^k+2$, for $k\in\mathbb{N}$. 
Corollary~\ref{coro:upper bound 1} with $L=5$, $|G|=3$ and $N=2$ yields, for every positive integer $k$,
\[
5^{k}\leqslant A(25^k+5^k+1)\leqslant 25\cdot 5^{3k}.
\]
\exend
\end{example}

\begin{example}[Inverse-palindromic non-Abelian]
Here we demonstrate why the Abelian assumption is necessary in Proposition~\ref{prop:pal}. Consider 
\[
\varrho\colon\,\, \begin{matrix}
0\mapsto 01120&\\ 
1\mapsto 12001&\\
2\mapsto 20212&
\end{matrix}.
\]
This substitution is inverse palindromic with $\varrho^{ }_0=\text{id}=\varrho^{ }_4,\varrho^{ }_2=(01)$ and $\varrho^{ }_1=(012)=\varrho^{-1}_3$ and $G=S_3$, which is non-Abelian. Consider $\varrho^{2}$ and let us compute $(\varrho^{2})_7$ and $(\varrho^{2})_{17}$. Since $7=[2,1]$, we have 
$(\varrho^{2})_7=\varrho^{ }_2\varrho^{ }_1=(12)$. Similarly $17=[2,3]$ so $(\varrho^{2})_{17}=\varrho^{ }_2\varrho^{ }_3=(02)$. 
Since $(\varrho^{2})^{ }_7\neq (\varrho^{2})^{-1}_{17}$, $\varrho^2$ is not inverse palindromic. \exend 
\end{example}

\subsection{Non-bijective substitutions with super-substitution structure}
\label{SecNonBij}
\begin{definition}
Let $\mathcal{A}$ 
be a finite alphabet
and consider a constant-length substitution $\varrho$ of length $L$ over $\mathcal{A}$. 
A partition
\[
   \mathcal{A}=\mathcal{A}_1\cup\mathcal{A}_2\cup\cdots \cup\mathcal{A}_n
\]
of $\mathcal{A}$ is said to induce a \emph{supersubstitution} for $\varrho$ if, for each
$\ell\in \left\{0,1,\ldots L-1\right\}, i\in \left\{1,2,\ldots ,n\right\}$ and $a,b\in\mathcal{A}_i$, there exists $j\in \left\{1,2,\ldots ,n\right\}$ such that
$\varrho^{ }_{\ell}(a)$ and $\varrho^{ }_{\ell}(b)$ are both in $\mathcal{A}_j$.
In such a case, we can define a map $\theta\colon\mathcal{A}\rightarrow\{1,2,\ldots,n\}$ by defining
$\theta(a)=i$, where $a\in\mathcal{A}_i$.

This allows one to define a substitution  $\xi$ of length $L$ on the alphabet $\{1,2,\ldots ,n\}$ 
via $\xi^{ }_{\ell}(i)=j$ for each $\ell\in \left\{0,1,\ldots ,L-1\right\}$ and $i\in\left\{1,2,\ldots,n\right\}$, where $j$ is such that
$\varrho^{ }_{\ell}(a)\in\mathcal{A}_j$ for each $a\in\mathcal{A}_i$.
By definition, we have $\theta \circ \varrho=\xi \circ \theta$.
\end{definition}

In what follows, we let $\varrho$ be a constant-length substitution of length $L$ over
$\mathcal{A}$ which admits a partition $ \mathcal{A}=\mathcal{A}_1\cup\mathcal{A}_2\cup\cdots \cup\mathcal{A}_n$ that induces a supersubstitution $\xi$.
Without loss of generality, let $v$ be the fixed point for $\varrho$ starting with some $a^{ }_1\in\mathcal{A}_1$ and $w$ be the fixed point of $\xi$ starting
with $1$. 
We have $\theta(v)=w$, because we have $\theta\circ\varrho=\xi\circ\theta$ and so
\begin{align*}
       \theta \circ \varrho^n(a^{ }_1)=\xi^n \circ \theta(a^{ }_1)=\xi^n(1)
\end{align*}
for each $n\geqslant 1$.

\begin{prop}\label{prop:supersubs}
Suppose $\mathcal{A}_1$ is singleton $\{a^{ }_1\}$ and that
$\xi$ satisfies property $(\ast)$ in Section~\textnormal{\ref{SecBijective}} with column group $G$. Then we have 
\begin{align*}
       A^{ }_v\left(\frac{L^{k|G|}-1}{L^k-1}\right)\geqslant L^k,
\end{align*}
for all $k\geqslant 1$.
\end{prop}

\begin{proof}
       By Proposition~\ref{prop:genbijective}, there is an arithmetic progression of $\mathcal{A}_1$ of difference
       $\frac{L^{k|G|-1}}{L^k-1}$ of length $L^k$ in $w$. In the preimage $v$, there is an arithmetic progression of $a^{ }_1$ of
       the same difference and length, since $\theta^{-1}(1)=\left\{a^{ }_1\right\}$.
\end{proof}

\begin{example}
Consider the alphabet $\mathcal{A}=\{a,b,c,d,e\}$ with the
partition $\mathcal{A}_1=\{a\}$,
$\mathcal{A}_2=\{b,c\},\mathcal{A}_3=\{d,e\}$.
Let  $G=S_3$. We start with a bijective substitution $\xi$ with
letters $\mathcal{A}_1,\mathcal{A}_2,\mathcal{A}_3$ such that
\[
\xi\colon \,\,
\begin{matrix}
    \mathcal{A}_1\mapsto\mathcal{A}_1\mathcal{A}_2\mathcal{A}_3\mathcal{A}_1\mathcal{A}_3\mathcal{A}_2&\\
    \mathcal{A}_2\mapsto\mathcal{A}_2\mathcal{A}_1\mathcal{A}_2\mathcal{A}_3\mathcal{A}_1\mathcal{A}_3&\\
    \mathcal{A}_3\mapsto\mathcal{A}_3\mathcal{A}_3\mathcal{A}_1\mathcal{A}_2\mathcal{A}_2\mathcal{A}_1&.
\end{matrix}
\]
We then construct a substitution $\varrho$ on $\mathcal{A}$ which is compatible with the supersubstitution  $\xi$. For example,
consider
\[
\varrho\colon\,\,
\begin{matrix}
    a\mapsto acdaec\\
    b\mapsto babead\\
    c\mapsto bacead\\
    d\mapsto ddabca\\
    e\mapsto edabca
\end{matrix}.
\]
By Proposition~\ref{prop:supersubs}, for the fixed point $v$ of $\varrho$ starting with $a$, one has
$A\left(\frac{L^{6k}-1}{L^k-1}\right)\geqslant L^k$. \exend
\end{example}

One can relax the singleton criterion in Proposition~\ref{prop:supersubs} and replace it with some restrictions on the columns of the original substitution.

\begin{prop}\label{prop:non-bijective}
Let $\varrho$ be a length-$L$ substitution with a supersubstitution structure $\xi$. 
Assume there are $0\leqslant c_1<c_2<\cdots<c_k<L$ such that
$\varrho^{ }_{c_j}(a)=a^{ }_1$ for each $j$, for all $a\in\mathcal{A}_1$.
Consider $n\geqslant 1$ with $L$-adic expansion $n=[n_m,n_{m-1},\ldots, n_1, n_0]$, with $n_0\in\{c_1,c_2,\ldots,c_k\}$.
Set $n'$ to be $n'=[n_m,n_{m-1},\ldots, n_1]$.
If $w_{n'}=1$, then $v^{ }_n=a^{ }_1$.
\end{prop}

\begin{proof}
Since  $\theta(v^{ }_{n'})=w_{n'}=1$, $v^{ }_{n'}\in\mathcal{A}_1$ and we have $v^{ }_n=\varrho^{ }_{n_0}(v^{ }_{n'})=a^{ }_1$.
\end{proof}

The previous result allows one to construct differences which correspond to long arithmetic progressions. 

\begin{example}\label{ex:non-bij-super}
Consider 
     $\mathcal{A}=\{a,b,c,d,e,f\}$ with $\mathcal{A}_1=\{a,b\},\mathcal{A}_2=\{c,d\}, \mathcal{A}_3=\{e,f\}$. Fix the length of the substitution to be $L=6$ and consider
     the substitution $\varrho$ given by
     \begin{align*}
     a\mapsto abbabd  &&c\mapsto cddcce  &&e\mapsto effeea \\[-3pt]
       b\mapsto aabaac  &&d\mapsto dccddf  &&f\mapsto fefefb.
     \end{align*}
     We see that, here $k=2$ and  $c_1=0, c_2=3$.
     The supersubstitution is
\[
\xi\colon \,\,
\begin{matrix}
         \mathcal{A}_1&\mapsto \mathcal{A}_1\mathcal{A}_1\mathcal{A}_1\mathcal{A}_1\mathcal{A}_1\mathcal{A}_2\\
         \mathcal{A}_2&\mapsto \mathcal{A}_2\mathcal{A}_2\mathcal{A}_2\mathcal{A}_2\mathcal{A}_2\mathcal{A}_3\\
         \mathcal{A}_3&\mapsto  \mathcal{A}_3 \mathcal{A}_3 \mathcal{A}_3 \mathcal{A}_3 \mathcal{A}_3 \mathcal{A}_1
    \end{matrix}\,.
\]
 For each $n\geqslant 1$, set
    \begin{align*}
        d_n=3+3\cdot 6^n+3\cdot 6^{2n}=[3,0,0,\ldots, 0,0,3,0,0,\ldots,0,0,3],
    \end{align*}
    where $[\cdot]$ is the base-6
    expansion. If $0\leqslant k\leqslant 4\cdot 6^{n-1}-1$, then
    $3k<2\cdot 6^n$ and the base-6
    expansion for $3k$ is
    $[i^{ }_n, i^{ }_{n-1},\cdots, i^{ }_0]$ with
    $i^{ }_n< 2$ and $i^{ }_0\in\left\{0,3\right\}$. For all such $k$, the base-6 expansion of  $kd_n=3k+3k\cdot 6^n+3k\cdot 6^{2n}$ is
    \[
        [i^{ }_n,i^{ }_{n-1},\ldots ,i^{ }_1, i^{ }_0+i^{ }_n,i^{ }_{n-1},\ldots i^{ }_1, i^{ }_0+i^{ }_n,\ldots i^{ }_{n-1},\ldots i^{ }_1, i^{ }_0].
    \]
    We now consider $k^{\prime}_n:=(kd_n)^{\prime}$ to be the number whose base-$6$
    expansion is the same as $kd_n$ with the last digit $i^{ }_0$ 
    omitted.  
    Let $w$ be the fixed point of $\xi$ starting at $1$. We show that
    $w^{ }_{k^{\prime}_n}=1$ for all $0\leqslant k\leqslant 4\cdot 6^{n-1}-1$. From the supersubstitution, we get that
    \[
    w^{ }_{k^{\prime}_n}=\big(\xi^{ }_{i^{ }_n}\xi^{2}_{i^{ }_0+i^{ }_n}\prod_{j=1}^{n-1}\xi^{ 3}_{i^{ }_1}\big)(1)=1. 
    \]
     Here, $\xi^{ }_{i^{ }_n}=\text{id}$ because $i^{ }_n<2$, $\xi^{ }_{i^{ }_0+i^{ }_n}=\text{id}$ because $i^{ }_0+i^{ }_n<5$, and the last factor in the product is also the identity because all $\xi^{ }_{i}$ are cube roots of unity.
    We can now apply the previous result to the original sequence $kd_n$. Since $i^{ }_0\in \left\{0,3\right\}$, by Proposition~\ref{prop:non-bijective}, the
    $(kd_n)$th letter in the fixed
    point $v$ starting with $a$ is $\varrho^{ }_{i^{ }_0}(\mathcal{A}_1)=a$, which yields an arithmetic progression
     of $a$s  with length $4\cdot 6^{n-1}$. \exend
\end{example}

\begin{remark}
Example~\ref{ex:non-bij-super} shows how Proposition~\ref{prop:non-bijective} can be used to establish lower bounds for $A(d)$ for some differences $d$ appropriately chosen. For differences of the form $d=\frac{L^{k|G|}-1}{L^k-1}$, one can use Propositions~\ref{prop:supersubs} and~\ref{prop:non-bijective} to show that $A_v(Ld)\geqslant L^k$, for all $k\geqslant 1$. Furthermore, if we assume that $c_1=0$ in Proposition~\ref{prop:non-bijective}, then $A_v(L^md)\geqslant L^k$, for every $m\geqslant 1$ and all $k\geqslant 1$ (see~\cite{Ae2023} for details).
\exend
\end{remark}

\section{Spin substitutions}\label{sec:spin}

We consider monochromatic arithmetic progressions in infinite words arising from spin substitutions, which are generalisations of the Rudin--Shapiro substitution. A \emph{spin substitution} $\theta$ is a special type of constant-length substitution. A finite set $\mathcal{D}$ of digits is considered, each of which, carrying a spin, can be in a finite number of distinct states. The spin states are represented using a finite Abelian group $G$, called the \emph{spin group}. This results in the alphabet $\mathcal{A}=\mathcal{D} \times G$. The substitution $\theta$ is then completely determined using a $|\mathcal{D}| \times |\mathcal{D}|$ matrix $V$ with entries in $G$, which is called the \emph{spin matrix}. The matrix $V$ encodes, for each digit $d\in\mathcal{D}$, the spin state of the letters of the image of $d$ under the substitution. For background on spin substitutions and generalisations, we refer the reader to~\cite{CGS,Queffelec,AL,FM}.

In Sections~\ref{sec:Rudin-Shapiro-1} and~\ref{sec:Rudin-Shapiro-2}, we study $A(d)$ for the Rudin--Shapiro sequence. We give lower bounds for $A(d)$ for two sequences of differences along which $A(d)$ grows at least linearly in $d$, in analogy to the classical Thue--Morse case studied in~\cite{AGNS}. In Section~\ref{sec:Vandermonde} we extend these results to Vandermonde sequences.

\subsection{The Rudin--Shapiro sequence}\label{sec:Rudin-Shapiro-1}

Consider a spin substitution $\theta$ with digit set $\mathcal{D}=\{0,1\}$, spin group $G=C_2$ and spin matrix $V=\begin{psmallmatrix*}[r] 1 & 1\\ 1& -1 \end{psmallmatrix*}$. The resulting alphabet is $\mathcal{A}=\mathcal{D} \times G=\{0,1,\tilde{0},\tilde{1}\}$, where `tilded' letters have non-trivial spin. The spin matrix determines the positions of the tildes in $\theta(0)$ and $\theta(1)$; the positions of the tildes in $\theta(\tilde{0})$ and $\theta(\tilde{1})$ are determined via the invariance relation $\theta(\tilde{a})=\widetilde{\theta(a)}$ with number of tildes modulo $2$, for $a\in\mathcal{D}$. The resulting substitution is
\[
\theta\colon\quad
\begin{matrix*}
0 \mapsto 01\\
1 \mapsto 0\tilde{1}
\end{matrix*}
\qquad\qquad
\begin{matrix*}
\tilde{0} \mapsto \tilde{0}\tilde{1}\\
\tilde{1} \mapsto \tilde{0}1
\end{matrix*}\,.
\]

The \emph{Rudin--Shapiro sequence} $u$ over the alphabet $\{1,-1\}$ is obtained from the fixed point of $\theta$ starting with $0$ under the projection
\[
\pi^{ }_{G}\colon\quad  0, 1 \mapsto 1,\quad  \tilde{0}, \tilde{1} \mapsto -1;
\]
see \cite[Section~7.7.1]{TAO}.
The first few terms of $u$ (with commas inserted for the sake of clarity) are
\[
u = 1,1,1,-1,1,1,-1,1,1,1,1,-1,-1,-1,1,-1,\dotsb.
\]
The $n$th element of $u$ can be derived from $V$ as
\begin{equation}\label{nth letter of RS}
u_n = \prod_{i=0}^{k-1}{V(n_{i+1},n_i)} = V(n_k,n_{k-1})\ \dotsb\ V(n_2,n_1)\ V(n_1,n_0),
\end{equation}
where $[n_k,\dotsc,n_1,n_0]$ is the binary representation of $n$ with $n_0$ the least significant digit and $n_k$ the most significant digit, and $V(i,j)$ is the $(i,j)$th entry of $V$~\cite{AL,FM}. Alternatively, $u_n$ can be obtained as $u_n=(-1)^{t(n)}$, where $t(n)$ counts the number of (possibly overlapping) occurrences of the word $11$ in the binary representation of the integer $n$; see \cite{AS}. The following recurrence relations can easily be obtained from Eq.~\eqref{nth letter of RS}
\begin{equation}\label{nth letter in RS}
u_{2n} = u_{n}, \qquad u_{2n+1} = (-1)^n \, u_{n}.
\end{equation}

The following simple argument invoking Proposition~\ref{prop:finiteness} shows that $A(d)<\infty$, for all $d\in\mathbb{N}$, for the sequence $u$.

\begin{prop}\label{prop:RS,A(d)finite}
There is no infinite monochromatic arithmetic progression in the sequence $u$.
\end{prop}
\begin{proof}
Let $v$ be the fixed point of $\theta$ starting with $0$. It follows from the definition of $\theta$ that $v^{ }_{2n+a}\in\{a,\tilde{a}\}$, for all $a\in\mathcal{D}$ and $n\in\mathbb{N}$. Then, since $u=\pi^{ }_G(v)$, we know that $u_{2n+a}=1$ (resp. $-1$) implies $v^{ }_{2n+a}=a$ (resp. $\tilde{a}$), for all $a\in\mathcal{D}$ and $n\in\mathbb{N}$. The proof is by contradiction. Assume there exist $s\in\mathbb{N}$ and $d \geqslant 1$ such that $u_{s+nd}=1$ (resp. $-1$), for all $n\in\mathbb{N}$. This implies that $v^{ }_{s+nd}=a$ (resp. $\tilde{a}$) for all $n \in 2 \mathbb{N}$, where $a\equiv s\bmod{2}$. But this is a contradiction because $\theta$ is an aperiodic, primitive, constant-length substitution of height $1$ and so, by Proposition~\ref{prop:finiteness}, $v$ does not contain infinite monochromatic arithmetic progressions.
\end{proof}

It is not difficult to show that $A(2^n d)=A(d)$, for all $d,n\geqslant 1$ (similar to the Thue--Morse case~\cite{AGNS}), and from here that $A(2^n )=4$, for all $n\in\mathbb{N}$. The next two propositions, where we find sequences of long monochromatic arithmetic progressions for differences of the form $2^n\pm 1$, are an analog of Proposition~\ref{prop:genbijective} for bijective substitutions.

\begin{prop}\label{prop:RS,2**n+1}
The sequence $u$ satisfies $A(2^n+1)\geqslant 2^{n-1}+2$, for all $n\geqslant 1$.
\end{prop}
\begin{proof}
It is easy to see, by direct inspection of $u$, that the result holds for $n=1$. For $n>1$, we will show that $u_k=1$ with $k=2^{2n+1}+m(2^n+1)$, for all $-1\leqslant m \leqslant 2^{n-1}$. Fix $n$. For $m=-1$, $k=2^{2n+1}-2^n-1$ with binary representation given by $[1,\dotsc,1,0,1,\dotsc,1]$, consisting of two sequences of $n$ consecutive $1$'s separated by a single $0$. Then, by Eq.~\eqref{nth letter of RS}, $u_k=1$. For $0 \leqslant m \leqslant 2^{n-1}$, let the binary representation of $m$ be $[m_r,\dotsc,m_1,m_0]$, where $0\leqslant r\leqslant n-1$. Then the binary representation of $k=2^{2n+1}+m(2^n+1)$ is
\[
[\,1,\overbracket[0.5pt]{0,\dotsc,0}^{n-r},m_r,\dotsc,m_1,m_0,\overbracket[0.5pt]{0,\dotsc,0}^{n-r-1},m_r,\dotsc,m_1,m_0\,].
\]
For $0\leqslant r<n-1$, we have $n-r-1 \geqslant 1$ and then, by Eq.~\eqref{nth letter of RS}, $u_k=1$. For $r=n-1$, we have $n-r-1=0$
and then, by Eq.~\eqref{nth letter of RS}, $u_k=V(m_0,m_r)$. But, for $r=n-1$, we also have $m=2^{n-1}=[1,0,0,\dotsc,0]$ and so, $u_k=V(m_0,m_r)=V(0,1)=1$.
\end{proof}

\begin{prop}\label{prop:RS,2**n-1}
The sequence $u$ satisfies, for all $n\geqslant 1$,
\[
A(2^n-1) \geqslant
\begin{cases}
2^{n-1}+1,\qquad\text{if $n$ is even,}\\
2^{n-1}+3,\qquad\text{otherwise.}
\end{cases}
\]
\end{prop}
\begin{proof}
The result holds if $n=1$, so we assume that $n \geqslant 2$. We will first show that, for every $n\geqslant 2$ and all $0 \leqslant m \leqslant 2^{n-1}$, there exists $a\in\{1,-1\}$ such that $u_k=a$, where $k=2^{2n}+(m+1)(2^n-1)$. Fixing $n$ and writing the binary representation of $m$ as $[m_{n-1},\dotsc,m_1,m_0]$, where $m_i\in\{0,1\}$ for all $0 \leqslant i \leqslant n-1$, the binary representation of $k$ takes the form $[1,m_{n-1},\dotsc,m_1,m_0,\overline{m_{n-1}},\dotsc,\overline{m_1},\overline{m_0}]$, where $\overline{m_i}=1-m_i$. By Eq.~\eqref{nth letter of RS} and given that, for each $0 \leqslant i \leqslant n-2$, $V(m_{i+1},m_i) V(\overline{m_{i+1}},\overline{m_i})$ is equal to $-1$ if $m_{i+1}=m_i$, and to $1$ otherwise, we see that
\[
u_k = V(1,m_{n-1}) \ V(m_0,\overline{m_{n-1}}) \ (-1)^{n-1+m_{n-1}-m_0}
\]
If $m_{n-1}\neq m_0$, $u_k= (-1)^{n-1}$. If $m_{n-1}=m_0$, $m_{n-1}=m_0=0$ because $m\leqslant 2^{n-1}$, and again $u_k= (-1)^{n-1}$. If $n$ is even, this implies that $u_k= -1$, which completes the proof for the even cases. If $n$ is odd, it implies that $u_k= 1$. In this case, one can easily further check that $u_k=1$ for $m=-1$ and $m=-2$, which completes the proof for the odd cases.
\end{proof}

\begin{coro}\label{coro:RS,growth}
For all $d=2^n\pm 1$ with $n\geqslant 1$, the sequence $u$ satisfies $A(d)\gtrsim d/2$.
\end{coro}
\begin{proof}
The claim follows directly from Propositions~\ref{prop:RS,2**n+1} and~\ref{prop:RS,2**n-1}.
\end{proof}

By computer experiments we have verified the preceding results for $1 \leqslant d \leqslant 4200$. In fact, we have seen that the inequalities in Propositions~\ref{prop:RS,2**n+1} and~\ref{prop:RS,2**n-1} are equalities, if $n\geqslant 4$ and if $n\geqslant 5$, respectively. Moreover,
the differences of the form $2^n \pm 1$ are those for which the Rudin--Shapiro sequence has the longest monochromatic arithmetic progressions, in the sense that $A(d)$ has  local maxima at these differences. A plot of $A(d)$ similar to that in~\cite{AGNS} for the Thue--Morse sequence can be obtained in this case for the Rudin--Shapiro sequence.

\begin{remark}
Sobolewski's paper~\cite{So2022} concerns the computation of upper bounds of $A^{ }_{w}(d)$, for sequences $w\in\mathcal{A}^{\mathbb{N}}$ defined using a block-counting function. More precisely, given a binary block $v\in\mathcal{A}^+$, the digit $w_n$ is given by the sum mod $2$ of (possibly overlapping) occurrences of $v$ in the binary representation of $n$, for all $n\in\mathbb{N}$. The author focuses most of his attention on the $v=11$ case, for which $w$ is the Rudin--Shapiro sequence. In this case, an upper bound of the maximum length of monochromatic arithmetic progressions starting at position $0$ is given, and exact values of $A(d)$ are determined for differences of the form $2^n \pm 1$.
\exend
\end{remark}

\begin{remark}\label{rem:Hadamard}
The arguments for the Rudin--Shapiro sequence can be extended to the case when the spin matrix is the Hadamard matrix \cite{Frank2}
\[
V=\begin{pmatrix*}[r]
1 & 1 & 1 & 1\\
1& -1 & 1 & -1\\
1 & 1 & -1& -1\\
1 & -1 & -1 & 1 
\end{pmatrix*}
\]
and hence, $\theta$ is a substitution of the eight-letter alphabet $\mathcal{A}=\mathcal{D} \times C_2$, where $\mathcal{D}=\{0,1,2,3\}$. The studied sequence $u$ arises as the image of the fixed point of $\theta$ starting with $0$ under the coding $\pi^{ }_G$ mapping untilded letters to $1$ and tilded letters to $-1$. In this case $A(4^n d)=A(d)$; in particular, $A(4^n)=6$ for all $n\in\mathbb{N}$. Results similar to Propositions~\ref{prop:RS,2**n+1} and~\ref{prop:RS,2**n-1} can also be derived. On the one hand, $A(4^n+1)\geqslant 4^{n-1}+2$, for all $n \geqslant 1$, the proof of which is similar to the proof of Proposition~\ref{prop:RS,2**n+1}. More precisely, it can be shown that, for all $-1\leqslant m \leqslant 4^{n-1}$, $u_k=1$ with $k=4^{2n+1}+m(4^n+1)$. On the other hand, $A(4^n-1) \geqslant 4^{n-1}+3$, for all $n \geqslant 1$, the proof of which is analogous to the proof of Proposition~\ref{prop:RS,2**n-1}. More precisely, it can be shown that, for all $-2 \leqslant m \leqslant 4^{n-1}$, $u_k=1$ with $k=3\cdot 4^{2n+2}+4^n-1+m(4^{n+1}-1)$.
\exend
\end{remark}

\subsection{An alternative approach for the Rudin--Shapiro sequence}\label{sec:Rudin-Shapiro-2}

The Rudin--Shapiro sequence can alternatively be obtained from a staggered substitution or by a substitution acting on an alphabet consisting of pairs of letters in $\{1,-1\}$, say $\{v,\widetilde{v},w,\widetilde{w}\}$ with $v=11$, $w=1{-1}$, where $\widetilde{1}=-1$ and $\widetilde{-1}=1$ swaps the two letters. The substitution reads

\[
\varrho\colon\qquad
\begin{matrix*}
v\mapsto vw\\
w\mapsto v\widetilde{w}
\end{matrix*}
\qquad\qquad
\begin{matrix*}
\tilde{w}\mapsto\widetilde{v}w\\
\tilde{v}\mapsto\widetilde{v}\widetilde{w}
\end{matrix*}\,.
\] 
Note that this substitution is exactly the same as the original four-letter substitution, except that we now interpret $v$ and $w$ as two-letter words in $\{1,-1\}$.

Note that the substitution is invariant under the letter exchange in the sense that
$\varrho(\widetilde{a})=\widetilde{\varrho(a)}$ for all $a\in\{v,w,\widetilde{w},\widetilde{v}\}$. Moreover, the first part of $\varrho(a)$ for any $a\in\{v,w,\widetilde{w},\widetilde{v}\}$ is either $v$ or $\widetilde{v}$, and the last is either $w$ or $\widetilde{w}$. By induction, this structure is preserved for larger superwords as follows.

\begin{lem}\label{lem:superwords}
Let\/ $n\geqslant 1$ and set\/ $v^{(0)}=v$, $w^{(0)}=w$, $v^{(n)}:=\varrho^{n}(v)$, and $w^{(n)}:=\varrho^{n}(w)$.  Then, $v^{(n)}=v^{(n-1)}w^{(n-1)}$ and\/ $w^{(n)}=v^{(n-1)}\widetilde{w^{(n-1)}}$. 
\end{lem}

\begin{proof}
Clearly, this is true for $n=1$. Assuming the structure holds for $n$, we find that
\[
v^{(n+1)}=\varrho(v^{(n)}) = \varrho(v^{(n-1)}w^{(n-1)}) = \varrho(v^{(n-1)})\varrho(w^{(n-1)})
=v^{(n)} w^{(n)}
\]
and
\[
w^{(n+1)}=\varrho(w^{(n)}) = \varrho(v^{(n-1)}\widetilde{w^{(n-1)}}) = \varrho(v^{(n-1)})\widetilde{\varrho(w^{(n-1)})}
=v^{(n)} \widetilde{w^{(n)}},
\]
which completes the proof.
\end{proof}

For the rest of the section, we consider the superword $v^{(2n-1)}=\varrho^{2n-1}(v)$ for $n\geqslant 2$, and write the resulting word in the alphabet $\{1,-1\}$, which has $2^{2n}$ letters, as a square array of letters with $2^n$ rows of length $2^n$. For $n\in\{2,3,4\}$, they are shown in Figure~\ref{fig:RSex}.
Let us denote the $(i,j)$th entry
of this matrix by $a^{ }_{i-1,j-1}$. With this notation, one has 
$v^{(2n-1)}=a^{ }_{0,0},a^{ }_{0,1}\ldots a^{ }_{2^n-1,2^n-1}$. In particular,
the sequence $a^{ }_{0,0},a^{ }_{0,1}\ldots a^{ }_{0,2^n-1}$ corresponds to the topmost row of the block (or the first row of the matrix). 
In what follows, $v^{(n)}_j$ (resp. $w^{(n)}_j$)  denotes the $j$th letter of $v^{(n)}$  (resp. $w^{(n)}$)
seen as a word over $\left\{1,-1\right\}$.

\begin{figure}
\centerline{\resizebox{1\textwidth}{!}{
    \begin{tikzpicture}[x=16pt,y=16pt]
    \node at (1.5,4) {$n=2$};
    \pd{0,3}\p{1,3}\p{2,3}\mo{3,3}
    \p{0,2}\pd{1,2}\mo{2,2}\p{3,2}
    \p{0,1}\p{1,1}\p{2,1}\m{3,1}
    \m{0,0}\m{1,0}\p{2,0}\m{3,0}
    \draw[line width=2pt,red] (1.5,-0.5) -- (1.5,3.5);
    \draw[line width=2pt,red] (-0.5,1.5) -- (3.5,1.5);
    \begin{scope}[shift={(5,-2)}]
    \node at (3.5,8) {$n=3$};
    \foreach \y in {0,2,3,5,6,7} {\p{0,\y}\p{1,\y}\p{2,\y}\m{3,\y}};
    \foreach \y in {1,4} {\m{0,\y}\m{1,\y}\m{2,\y}\p{3,\y}};
    \foreach \y in {0,1,2,6} {\m{4,\y}\m{5,\y}\p{6,\y}\m{7,\y}};
    \foreach \y in {3,4,5,7} {\p{4,\y}\p{5,\y}\m{6,\y}\p{7,\y}};
    \foreach \y in {0,1,2,3} {\pd{\y,7-\y}};
    \foreach \y in {0,1,2,3} {\po{7-\y,7-\y}};
    \draw[line width=2pt,red] (3.5,-0.5) -- (3.5,7.5);
    \draw[line width=2pt,red] (-0.5,3.5) -- (7.5,3.5);
    \end{scope}
    \begin{scope}[shift={(14,-4)}]
    \node at (7.5,16) {$n=4$};
    \foreach \y in {0,2,3,4,9,12} {\m{0,\y}\m{1,\y}\m{2,\y}\p{3,\y}\m{4,\y}\m{5,\y}\p{6,\y}\m{7,\y}};
    \foreach \y in {1,5,6,7,8,10,11,13,14,15} {\p{0,\y}\p{1,\y}\p{2,\y}\m{3,\y}\p{4,\y}\p{5,\y}\m{6,\y}\p{7,\y}};
    \foreach \y in {0,1,2,4,5,7,11,12,13,15} {\p{8,\y}\p{9,\y}\p{10,\y}\m{11,\y}\m{12,\y}\m{13,\y}\p{14,\y}\m{15,\y}};
    \foreach \y in {3,6,8,9,10,14} {\m{8,\y}\m{9,\y}\m{10,\y}\p{11,\y}\p{12,\y}\p{13,\y}\m{14,\y}\p{15,\y}};
    \foreach \y in {0,1,...,7} {\pd{\y,15-\y}};
    \foreach \y in {0,1,...,7} {\mo{15-\y,15-\y}};
    \draw[line width=2pt,red] (7.5,-0.5) -- (7.5,15.5);
    \draw[line width=2pt,red] (-0.5,7.5) -- (15.5,7.5);
    \end{scope}
    \end{tikzpicture}
}}
\caption{The square array of letters in $v^{(2n-1)}$, for $n\in\{2,3,4\}$. The highlighted letters on the diagonals form arithmetic progressions of length $2^{n-1}$ for distances $d=2^n+1$ (upper left quadrant) or $d=2^n-1$ (upper right quadrant).\label{fig:RSex}}
\end{figure}
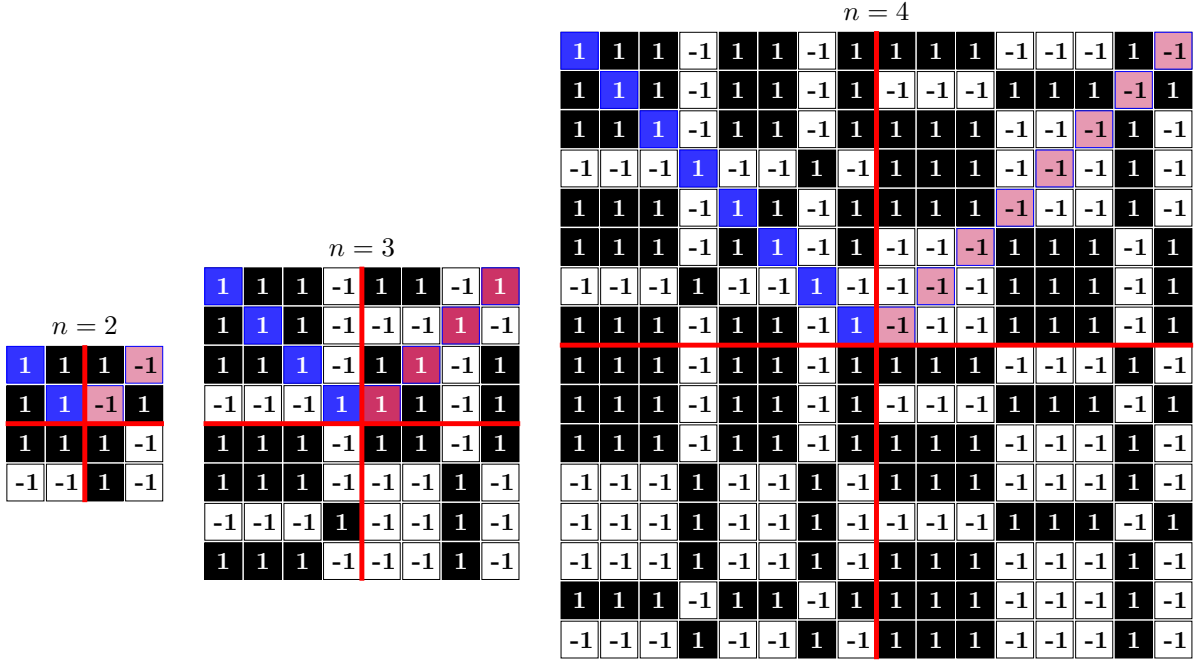

\begin{lem}\label{lem:letters-superwords}
Let $v^{(n)}$ and $w^{(n)}$ be as in Lemma~\textnormal{\ref{lem:superwords}}. 
We then have the following.

\begin{enumerate}
    \item $v^{(n)}_0=v^{(n)}_{2^n}=1$, \quad $v^{(n)}_{2^{n}-1}=\begin{cases}
    -1,& n\in 2\mathbb{Z}\\
    1,& n\in 2\mathbb{Z}+1
    \end{cases}$, \quad $v^{(n)}_{2^{n+1}-1}=\begin{cases}
    1,& n\in 2\mathbb{Z}\\
    -1,&  n\in 2\mathbb{Z}+1
    \end{cases}$.
    \item $w^{(n)}_{0}=1,\,w^{(n)}_{2^n}=-1$,\quad 
    $w^{(n)}_{2^{n}-1}=\begin{cases}
    -1,& n\in 2\mathbb{Z}\\
    1,& n\in 2\mathbb{Z}+1
    \end{cases}$, \quad $w^{(n)}_{2^{n+1}-1}=\begin{cases}
    -1,& n\in 2\mathbb{Z}\\
    1,&  n\in 2\mathbb{Z}+1
    \end{cases}$.
\end{enumerate}

\end{lem}

\begin{proof}
    Note that the word $v^{(n)}$ always
     starts with $1$, and ends with
        $1$ if $n$ is even or
        $-1$ if $n$ is odd.
   Similarly, the word $w^{(n)}$,
        starts with $1$ and
        ends with $-1$ if $n$ is even
        or $1$ is $n$ is odd.
    The properties above then follow from 
    Lemma \ref{lem:superwords} by induction.  
\end{proof}

From the substitution structure, one has 
\[
\varrho^n(a^{ }_{0,2i}a^{ }_{0,2i+1})=a^{ }_{2i,0}a^{ }_{2i,1}\cdots a^{ }_{2i,2^n-1}a^{ }_{2i+1,0}a^{ }_{2i+1,1}\cdots a^{ }_{2i,2^n-1}.\]
Together with Lemma~\ref{lem:letters-superwords}, we get the following. 
\begin{lem}\label{lem:converting-rule-RS}
Let $v^{(2n-1)}=a^{ }_{0,0},a^{ }_{0,1}\ldots a^{ }_{2^n-1,2^n-1}$ as above. 
\begin{enumerate}
    \item \textnormal{(From top to left)} If $a^{ }_{0,2i}a^{ }_{0,2i+1}=v$, then $a^{ }_{2i,0}a^{ }_{2i+1,0}=v$. 
    If $a^{ }_{0,2i}a^{ }_{0,2i+1}=w$, then $a^{ }_{2i,0}a^{ }_{2i+1,0}=w$. 
    
    \item \textnormal{(From top to right)}
     If $a^{ }_{0,2i}a^{ }_{0,2i+1}=v$, then $a^{ }_{2i,2^n-1}a^{ }_{2i+1,2^n-1}$ is
    $w$ if $n$ is odd and
    $\tilde{w}$ if $n$ is even.
    If $a^{ }_{0,2i}a^{ }_{0,2i+1}=w$, then
    $a^{ }_{2i,2^n-1}a^{ }_{2i+1,2^n-1}$ is
    $v$ if $n$ is odd and
    $\tilde{v}$ if $n$ is even. 
\end{enumerate}
with obvious extensions to the case when $a^{ }_{0,2i}a^{ }_{0,2i+1}$ is $\widetilde{v}$ or  $\widetilde{w}$.
\end{lem}

The previous lemma relates words in the topmost row of the matrix to words found along the leftmost and the rightmost columns. 
We can define the following maps
which convert words in the
topmost row in the matrix to
the words along the rightmost column. Note that $v$ read backwards is still $v$ and that  $w$ read backwards is $\widetilde{w}$.

\begin{definition}
Let
\[
I^{ }_{\mathrm{odd}}\colon\quad
\begin{matrix*}
v\mapsto \tilde{w}\\
\tilde{v}\mapsto w
\end{matrix*}
\qquad
\begin{matrix*}
w\mapsto v\\
\tilde{w}\mapsto \tilde{v}
\end{matrix*}\,,
\qquad\qquad\qquad\qquad
I^{ }_{\mathrm{even}}\colon\quad
\begin{matrix*}
v\mapsto w\\
\tilde{v}\mapsto \tilde{w}
\end{matrix*}
\qquad
\begin{matrix*}
w\mapsto \tilde{v}\\
\tilde{w}\mapsto v
\end{matrix*}\,,
\]
and set 
\[
I^{ }_{*}(b_1b_2\cdots b_n)=I^{ }_*(b_n)I^{ }_*(b_{n-1})\cdots I^{ }_*(b_1)
\]
for $b_i\in\{v,w,\tilde{v},\tilde{w}\}$ and $*\,\in\left\{\mathrm{even,odd}\right\}$.
\end{definition}

 From (2) in  Lemma \ref{lem:converting-rule-RS}, the rightmost column, read from bottom to top, is $I^{ }_{\mathrm{even}}(v^{(n-1)})$ if $n$ is even and
$I^{ }_{\mathrm{odd}}(v^{(n-1)})$ if $n$ is odd. We now express $I^{ }_{\ast}(v^{(n-1)})$ as a level-$n$ superword.  

\begin{lem}\label{lem:rightmost-column}
    We have
    \begin{align*}
        I^{ }_{\ast}(v^{(n-1)})=\varrho^{n-1}(\tilde{w})\qquad\text{and}\qquad
       I^{ }_{\ast}(w^{(n-1)})=\varrho^{n-1}(v)   
    \end{align*}
    for  $*\in\left\{\mathrm{even,odd}\right\}$.
\end{lem}
\begin{proof}
Since the proof for the case $\ast=\text{odd}$ is similar, we omit it and only present the one for the even case. We proceed by
    induction. The statement is
    clear for $n=2$. If the statement
    holds for some even $n$, then for
    $n+2$ we have:
    \begin{align*}
        I^{ }_{\mathrm{even}}(\varrho^{n+1}(v))&=I^{ }_{\mathrm{even}}(\varrho^{n-1}\varrho^2(v))\\
        &=I^{ }_{\mathrm{even}}(\varrho^{n-1}(vwv\tilde{w}))\\
        &=I^{ }_{\mathrm{even}}(\varrho^{n-1}(\tilde{w}))
        I^{ }_{\mathrm{even}}(\varrho^{n-1}(v))
        I^{ }_{\mathrm{even}}(\varrho^{n-1}(w))
        I^{ }_{\mathrm{even}}(\varrho^{n-1}(v))\\
        &=\varrho^{n-1}(\tilde{v})\varrho^{n-1}(\tilde{w})\varrho^{n-1}(v)\varrho^{n-1}(\tilde{w})\\
        &=\varrho^{n-1}(\tilde{v}\tilde{w}v\tilde{w})\\
        &=\varrho^{n+1}(\tilde{w}).
    \end{align*}
    By a similar computation, we have
    $I^{ }_{\mathrm{even}}\varrho^{n-1}(w)=\varrho^{n-1}(v)$.
\end{proof}

We now have the following result. 

\begin{prop}\label{prop:edges_square_for_RS}
The word which appears on the leftmost column of the top left quadrant is
always $v^{(n-2)}$. The word on
the rightmost column of the top-right quadrant, read from bottom to top, is $w^{(n-2)}$.
\end{prop}

\begin{proof}
    The first claim follows immediately from (1) of Lemma~\ref{lem:converting-rule-RS}. The second claim follows from (2) of Lemma~\ref{lem:converting-rule-RS} , Lemma~\ref{lem:rightmost-column} and the fact that the 
    word on the upper half of the rightmost quadrant (read from bottom to top) is the second half of $\varrho^{n-1}(\widetilde{w})=\varrho^{n-2}(\widetilde{v}w)=\varrho^{n-2}(\widetilde{v})\varrho^{n-2}(w)$, which is $w^{(n-2)}$; see Figure~\ref{fig:RSblock} for an illustration. 
\end{proof}

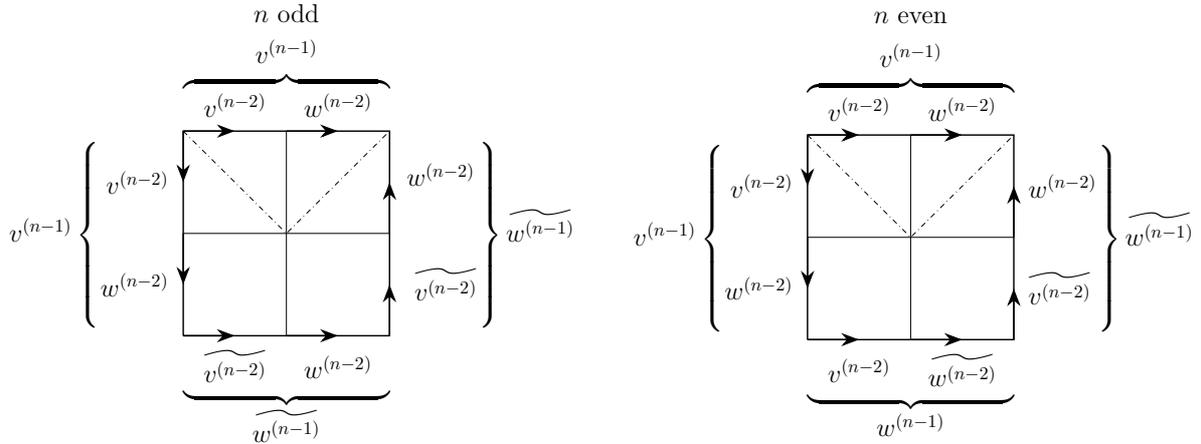
\begin{figure}[h!]
\centerline{\resizebox{1\textwidth}{!}{
    \begin{tikzpicture}[font=\Large]
    \node at (2,6.3) {$n$ odd};
    \draw[thick] (0,0) -- (4,0) -- (4,4) -- (0,4) -- (0,0);
    \draw (0,2) -- (4,2);
    \draw (2,0) -- (2,4);
    \draw[dashdotted] (0,4) -- (2,2) -- (4,4);
    \draw[thick,-{Stealth[scale=1.5]}] (0,0) -- (1,0);
    \draw[thick,-{Stealth[scale=1.5]}] (2,0) -- (3,0);
    \draw[thick,-{Stealth[scale=1.5]}] (4,0) -- (4,1);
    \draw[thick,-{Stealth[scale=1.5]}] (4,2) -- (4,3);
    \draw[thick,-{Stealth[scale=1.5]}] (0,4) -- (1,4);
    \draw[thick,-{Stealth[scale=1.5]}] (2,4) -- (3,4);
    \draw[thick,-{Stealth[scale=1.5]}] (0,4) -- (0,3);
    \draw[thick,-{Stealth[scale=1.5]}] (0,2) -- (0,1);
    \node[below] at (1,-0.15) {$\widetilde{v^{(n-2)}}$};
    \node[below] at (3,-0.25) {$w^{(n-2)}$};
    \node[below] at (2,-0.9) {$\underbrace{\hspace{4cm}}_{\mbox{$\widetilde{w^{(n-1)}}\rule[-1ex]{0ex}{3ex}$}}$};
    \node[above] at (1,4.15) {$v^{(n-2)}$};
    \node[above] at (3,4.15) {$w^{(n-2)}$};
    \node[above] at (2,5.20) {$v^{(n-1)}$};
    \node[above] at (2,4.65) {$\overbrace{\hspace{4cm}}^{}$};
    \node[left] at (-0.15,3) {$v^{(n-2)}$};
    \node[left] at (-0.15,1) {$w^{(n-2)}$};
    \node[right] at (4,3.15) {\hspace{2mm}$w^{(n-2)}$};
    \node[right] at (4.15,1) {\hspace{2mm}$\widetilde{v^{(n-2)}}$};
    \node[left] at (-1.45,2)  {$v^{(n-1)}\left\{\rule{0cm}{2cm}\right.$};
    \node[right] at (5.55,2) {$\left.\rule{0cm}{2cm}\right\} \widetilde{w^{(n-1)}}$};
    \end{tikzpicture}
    \qquad
    \begin{tikzpicture}[font=\Large]
    \node at (2,6.3) {$n$ even};
    \draw[thick] (0,0) -- (4,0) -- (4,4) -- (0,4) -- (0,0);
    \draw (0,2) -- (4,2);
    \draw (2,0) -- (2,4);
    \draw[dashdotted] (0,4) -- (2,2) -- (4,4);
    \draw[thick,-{Stealth[scale=1.5]}] (0,0) -- (1,0);
    \draw[thick,-{Stealth[scale=1.5]}] (2,0) -- (3,0);
    \draw[thick,-{Stealth[scale=1.5]}] (4,0) -- (4,1);
    \draw[thick,-{Stealth[scale=1.5]}] (4,2) -- (4,3);
    \draw[thick,-{Stealth[scale=1.5]}] (0,4) -- (1,4);
    \draw[thick,-{Stealth[scale=1.5]}] (2,4) -- (3,4);
    \draw[thick,-{Stealth[scale=1.5]}] (0,4) -- (0,3);
    \draw[thick,-{Stealth[scale=1.5]}] (0,2) -- (0,1);
    \node[below] at (1,-0.25) {$v^{(n-2)}$};
    \node[below] at (3,-0.15) {$\widetilde{w^{(n-2)}}$};
    \node[below] at (2,-0.9) {$\underbrace{\hspace{4cm}}_{\mbox{$w^{(n-1)}\rule[-1.5ex]{0ex}{3ex}$}}$};
    \node[above] at (1,4.15) {$v^{(n-2)}$};
    \node[above] at (3,4.15) {$w^{(n-2)}$};
    \node[above] at (2,5.20) {$v^{(n-1)}$};
    \node[above] at (2,4.65) {$\overbrace{\hspace{4cm}}^{}$};
    \node[left] at (-0.15,3) {$v^{(n-2)}$};
    \node[left] at (-0.15,1) {$w^{(n-2)}$};
    \node[right] at (4.15,3) {$w^{(n-2)}$};
    \node[right] at (4.15,1) {$\widetilde{v^{(n-2)}}$};
    \node[left] at (-1.45,2) {$v^{(n-1)}\left\{\rule{0cm}{2cm}\right.$};
    \node[right] at (5.45,2) {$\left.\rule{0cm}{2cm}\right\}\widetilde{w^{(n-1)}}$};
    \end{tikzpicture}
}}
\caption{Schematic representation of block arrangement of the superword $v^{(2n-1)}$. The dashed lines indicate the reflection symmetries of the upper left and upper right quadrants.}\label{fig:RSblock}
\end{figure}

Now we can prove the existence of
monochromatic diagonal and anti-diagonals.

\begin{prop}\label{prop:diagonal}
    We have $a^{ }_{j,j}=1$ for
    $0\leqslant j\leqslant 2^{n-1}-1$.
\end{prop}
\begin{proof}
    Each left half-row $a^{ }_{i,0}a^{ }_{i,1}\cdots a^{ }_{i,2^{n-1}-1}$ is either
    $v^{(n-2)}$ or $\widetilde{v^{(n-2)}}$
    and the topmost half-row (with $i=0$)  is
    always $v^{(n-2)}$.
    If $a^{ }_{0,j}=1$, then by Proposition~  \ref{prop:edges_square_for_RS},
    $a^{ }_{j,0}=1$, which
    implies that the $j$th row
    is $v^{(n-2)}$.
   This implies $a^{ }_{j,j}=a^{ }_{0,j}=1$.
        If $a^{ }_{0,j}=-1$, then by Proposition~\ref{prop:edges_square_for_RS}, $a^{ }_{j,0}=-1$, which
    implies that the $j$th row
    is $\widetilde{v^{(n-2)}}$.
    We see $a^{ }_{j,j}=\widetilde{a^{ }_{0,j}}=1$.
\end{proof}

\begin{prop}\label{prop:anti-diagonal}
If $n$ is odd, then
    we have $a^{ }_{j,2^n-j-1}=1$ for
    $0\leqslant j\leqslant 2^{n-1}-1$.
    If $n$ is even, then
    we have $a^{ }_{j,2^n-j-1}=-1$ for
    $0\leqslant j\leqslant 2^{n-1}-1$.
\end{prop}

\begin{proof}
    The $i$th right half-row $a^{ }_{i,2^{n-1}}a^{ }_{i,2^{n-1}+1}\cdots a^{ }_{i,2^{n}-1}$ is either
    $w^{(n-2)}$ or $\widetilde{w^{(n-2)}}$
    and the topmost right half-row is always $w^{(n-2)}$.
Assume that $n$ is odd. We have $a^{ }_{0,2^n-1}=1$. If $a^{ }_{0,2^n-j-1}=1$,
then by Proposition~\ref{prop:edges_square_for_RS}, $a^{ }_{j,2^n-1}=1$,
which implies that the $j$th right half-row is $w^{(n-2)}$ and is the same as the topmost right half-row.
It follows that $a^{ }_{j,2^n-j-1}=a^{ }_{0,2^n-j-1}=1$.
If $a^{ }_{0,2^n-j-1}=-1$,
then by Proposition~\ref{prop:edges_square_for_RS}, $a^{ }_{j,2^n-1}=-1$,
which implies that the $j$th right half-row is $\widetilde{w^{(n-2)}}$. We then have $a^{ }_{j,2^n-j-1}=\widetilde{a^{ }_{0,2^n-j-1}}=1$.
The case when $n$ is even admits a completely analogous proof, which we leave to the reader.
\end{proof}

Note that Proposition~\ref{prop:diagonal} proves the existence of a monochromatic arithmetic progression of difference $d=2^n+1$ while 
Proposition~\ref{prop:anti-diagonal} yields one with difference 
$d=2^n-1$. We conclude with the following comparable version of Propositions~\ref{prop:RS,2**n+1} and \ref{prop:RS,2**n-1} in the previous section. 

\begin{coro}
    For the binary Rudin--Shapiro sequence, one has 
    \[
A(2^n-1)\geqslant 2^{n-1}\quad \text{and} \quad A(2^n+1)\geqslant 2^{n-1}. 
    \]
\end{coro}

\subsection{Vandermonde sequences}\label{sec:Vandermonde}

In this section, we consider general Vandermonde substitutions, the simplest of which is the Rudin--Shapiro substitution. Let be a spin substitution with digit set $\mathcal{D}=\{0,1,\dotsc,L-1\}$ and spin group $G =\{1,\omega,\omega^2,\dotsc,\omega^{L-1}\}$, where $\omega={\mathrm{e}}^{-2\pi{\mathrm{i}}/L}$, resulting in the alphabet $\mathcal{A}=\mathcal{D}\times G$. We will consider the \emph{digit projection} $\pi^{ }_{\mathcal{D}}\colon\mathcal{A}\to\mathcal{D}$, defined by $\pi^{ }_{\mathcal{D}}(a)$ to be the digit of $a$, and the \emph{spin projection} $\pi^{ }_G\colon\mathcal{A}\to G$, defined by $\pi^{ }_G(a)$ to be spin of $a$. For each letter $a\in\mathcal{A}$, let $s_a\in\{0,1,\dotsc,L-1\}\subseteq\mathbb{N}$ be the \emph{spin number} of $a$, given by the exponent of $\omega$ in $\pi^{ }_G(a)$. Let the spin matrix $V$ of the spin substitution be a \emph{Vandermonde matrix}, given by $V(i,j)=\omega^{ij\bmod L}$, for $0\leqslant i,j \leqslant L-1$. In matrix form
\begin{equation*}\label{general Vandermonde matrix}
V=\begin{pmatrix*}[c]
1 & 1 & 1 & \dotsb & 1 & 1\\
1 & \omega & \omega^2 & \dotsb & \omega^{L-2} & \omega^{L-1}\\
1 & \omega^2 & \omega^4 & \dotsb & \omega^{2(L-2)} & \omega^{2(L-1)}\\
\vdots & \vdots & \vdots & \ddots & \vdots & \vdots\\
1 & \omega^{L-1} & \omega^{(L-1)2} & \dotsb & \omega^{(L-1)(L-2)} & \omega^{(L-1)(L-1)}
\end{pmatrix*}.
\end{equation*}
Now, let the spin substitution $\theta\colon\mathcal{A}\to\mathcal{A}^{L}$ be defined, for each $a\in\mathcal{A}$, by $\theta(a)=a^{ }_0a^{ }_1\dotsb a^{ }_{L-1}$, where, for each $0\leqslant i\leqslant L-1$, $a^{ }_i\in\mathcal{A}$ is such  that $\pi^{ }_{\mathcal{D}}(a^{ }_i)=i$ and $\pi^{ }_G(a^{ }_i) =  V(i, \pi^{ }_{\mathcal{D}}(a))\,\pi^{ }_G(a) = \omega^{i \cdot \pi^{ }_{\mathcal{D}}(a)+s_a\bmod L}$. We call $\theta$ a \emph{Vandermonde substitution}, and the infinite word $u$ obtained from the fixed point of $\theta$ starting with $0$ under the projection $\pi^{ }_G$ a \emph{Vandermonde sequence}. The first few terms of $u$ (with commas inserted for the sake of clarity) are
\[
u = \overbracket[0.5pt]{1,1,\dotsc,1}^{L},1,\omega,\omega^2,\dotsc,\omega^{L-2},\omega^{L-1},1,\omega^2,\omega^3,\dotsc,\omega^{L-1},\omega,1,\omega^3,\omega^4,\dotsb.
\]
The $n$th entry of $u$ can be obtained from $V$ using again Eq.~\eqref{nth letter of RS}, namely, $u_n = \prod_{i=0}^{k-1}{V(n_{i+1},n_i)}$, where $[n_k,\dotsc,n_1,n_0]$ is now the base-$L$ representation of $n$. Using this, one can easily prove the following lemma
(the proof of which we omit), which gives analogous recurrence relations to those in Eq.~\eqref{nth letter in RS}.

\begin{lem}\label{lem:general Vandermonde recurrences}
The Vandermonde sequence $u$ satisfies, for all $n\in\mathbb{N}$ and each $0\leqslant a\leqslant L-1$, the recurrence relation $u_{Ln+a}=V(a,b)\,u_n$,  where $0\leqslant b\leqslant L-1$ is such that $n\equiv b\bmod L$.
\end{lem}

A simple argument, similar to that used in Proposition~\ref{prop:RS,A(d)finite}, can be used to show that the Vandermonde sequence $u$ satisfies $A(d)<\infty$, for all $d\geqslant 1$.

\begin{prop}\label{prop:general Vandermonde,A(d)finite}
There is no infinite monochromatic arithmetic progression in the sequence $u$.
\end{prop}
\begin{proof}
Assume there exist $r\in\mathbb{N}$ and $d \geqslant 1$ such that $u_{r+nd}=g\in G$, for all $n\in\mathbb{N}$. Then $u_{r+nLd}=g$, for all $n\in\mathbb{N}$, and writing $r$ as $mL+k$, where $m\in\mathbb{N}$ and $0\leqslant k \leqslant L-1$, we have $u_{sL+k}=g$, where $s=m+nd$, for all $n\in\mathbb{N}$. Let $v\in\mathcal{A}^{\mathbb{N}}$ be the fixed point of $\theta$ starting with $0$, hence $u=\pi^{ }_G(v)$. Then $\pi^{ }_G(v^{ }_{sL+k})=g$ and, by the definition of $\theta$, $\pi^{ }_{\mathcal{D}}(v^{ }_{sL+k})=k$, for all $s$. Therefore, $v$ contains an infinite monochromatic arithmetic progression. But, since $\theta$ is an aperiodic, primitive, constant-length substitution of height $1$, this is a contradiction, by Proposition~\ref{prop:finiteness}.
\end{proof}

It can be shown that $A(L^n d)=A(d)$ for all $n\in\mathbb{N}$ and, in particular, $A(L^n)=A(1)=L+2$. As an analogy with the Rudin--Shapiro sequence, in the following proposition we look at the differences of the form $\frac{L^{nL}-1}{L^n-1}=L^{(L-1)n}+L^{(L-2)n}+\dotsc + L^n+1$.

\begin{prop}\label{prop:general Vandermonde, lower bound}
The sequence $u$ satisfies $A\left(\frac{L^{nL}-1}{L^n-1}\right)\geqslant L^{n-1}+1$, for all $n\geqslant 1$. 
\end{prop}
\begin{proof}
We omit the details for the case $n=1$. Let be $n>1$.  To prove the claim, we will show that $u_k=1$ with $k=L^{Ln+1}+m(L^{(L-1)n}+\dotsc + L^n+1)$, for all $0\leqslant m \leqslant L^{n-1}$. Let the base-$L$ representation of $m$ be $[m_r,\dotsc,m_1,m_0]$, where $0\leqslant r\leqslant n-1$. The base-$L$ representation of  $k$ is then given by
\[
[\,1,0,\overbracket[0.5pt]{\overbracket[0.5pt]{0,\dotsc,0}^{n-r-1},m_r,\dotsc,m_1,m_0,\dotsc\dotsc,\overbracket[0.5pt]{0,\dotsc,0}^{n-r-1},m_r,\dotsc,m_1,m_0}^{L\text{ times}}\,].
\]
For $0\leqslant m \leqslant L^{n-1}-1$, we have $n-r-1\geqslant1$ and, by Eq.~\eqref{nth letter of RS}, $u_k=\prod_{i=0}^{r-1}(V(m_{i+1},m_i))^L = 1$. For $m=L^{n-1}$, we have $n-r-1=0$, but also $m_i=0$, for all $0\leqslant i \leqslant r-1$. 
Consequently, the base-$L$ representation of $k$ consists of $L+1$ isolated digits equal to $1$ separated by sequences of digits equal to $0$, which, by Eq.~\eqref{nth letter of RS}, implies that $u_k = 1$.
\end{proof}

It is easy to check that the arithmetic progression of $1$'s found in the proof of Proposition~\ref{prop:general Vandermonde, lower bound} cannot be extended to the right, and that it can neither be extended to the left, except if $L=2$ (thus yielding as a corollary Proposition~\ref{prop:RS,2**n+1} for the Rudin--Shapiro sequence).
The following is immediate from Proposition~\ref{prop:general Vandermonde, lower bound} and it implies Theorem~\ref{thm:spin-growth}.

\begin{coro}\label{coro:general Vandermonde, lower bound}
For all $d=\frac{L^{kL}-1}{L^k-1}$ with $k\geqslant 1$, $A(d)\gtrsim \frac{d^{\alpha}}{L}$, where $\alpha=(L-1)^{-1}$.
\end{coro}

Note that if $L=2$, we recover Corollary~\ref{coro:RS,growth} for the Rudin--Shapiro sequence, for which $A(d)$ grows linearly in $d$.

\section{Outlook} \label{SecOutlook}

It is not obvious how to extend Proposition \ref{prop:genbijective} to the general constant-length case. 
Unlike in the bijective setting, the columns generate a semigroup rather than a group, and we no longer necessarily have an identity column. 
Even the task of finding a suitable subsequence with growing arithmetic progressions becomes nontrivial, as the following example illustrates. 

Let $\mathcal{A}$ be the $6$-letter alphabet and $\varrho$ be the length $2$ substitution
\begin{align}  \label{eq:ex-six-letter}
    a& \mapsto ad   &c& \mapsto ea   &e& \mapsto bf\\[-3pt]
    b& \mapsto bc   &d& \mapsto ab    &f& \mapsto ba. \nonumber
\end{align}
The graph of sets in Figure \ref{fig:column_graph} traces which letters occur in the columns. 
 This follows a modified version of the graph in \cite{Reemetal} and incorporates the subsets of $\mathcal{A}$ which appear as columns. 
 The sets which appear at the lowermost level are called \emph{minimal sets}; see also \cite{LemMuellner}. 
 These are the subsets of the alphabet that appear as columns in a large enough power of the substitution. 
 Thus, the size of the minimal sets corresponds to the column number $c(\varrho)$.  
 In this example, $c(\varrho)=2$. 

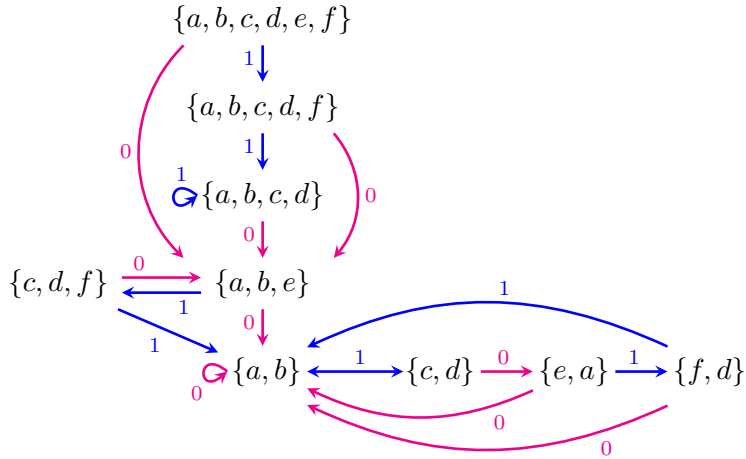
\begin{figure}[h!]
\begin{centering}

\begin{tikzcd}[column sep=5mm, row sep=5mm]
&\{a,b,c,d,e,f\}
\ar[ddd, line width=1, magenta, "0", font=\tiny, bend right=45, swap, pos=0.5, shift right=7mm]
\ar[d, line width=1, blue, "1", pos=0.35, swap]
&&&
\\
&\{a,b,c,d,f\}
\ar[dd, line width=1, magenta, "0", bend left=40, pos=0.15, shift left=6.5mm, pos=0.5]
\ar[d, line width=1, blue, "1", pos=0.35, swap]
&&&
\\
&\{a,b,c,d\}
\ar[d, line width=1, magenta, "0", pos=0.35, swap]
\ar[loop, line width=1, blue, out=175, in=195, distance=15, "1", bend left=50, shift right=18.5mm, swap, pos=0.22]
&&&
\\
\{c,d,f\}
\ar[r, line width=1, magenta, "0", bend left=0, shift left=0.8mm, pos=0.2]
\ar[rd, line width=1, blue, "1", bend right=0, swap, pos=0.5]
&
\{a,b,e\}
\ar[d, line width=1, magenta, "0", pos=0.35, swap]
\ar[l, line width=1, blue, "1", bend left=0, shift left=1.2mm, pos=0.2]
&&&
\\
&
\{a,b\}
\hspace{-1mm}
\ar[loop, line width=1, magenta, out=175, in=195, distance=15, "0", bend left=50, shift right=10.5mm, swap, pos=0.75]
\ar[r, <->, line width=1, blue, "1", pos=0.55]
&
\hspace{-1.5mm}
\{c,d\}
\hspace{1mm}
\ar[r, shorten <= -2mm, line width=1, magenta, "0", pos=0.2]
&
\hspace{-1mm}
\{e,a\}
\hspace{1mm}
\ar[ll, line width=1, magenta, "0", bend left=25, pos=0.2, shift right=0.5mm]
\ar[r, shorten <= -2mm,  line width=1, blue, "1", pos=0.1]
&
\hspace{-1mm}
\{f,d\}
\ar[lll, line width=1, magenta, bend left=25, "0", shift left=1.7mm,  pos=0.2]
\ar[lll, line width=1, blue, "1", bend right=25, swap, pos=0.45, shift right=0.5mm]
\end{tikzcd}
\caption{Graph of sets for $\varrho$}
\label{fig:column_graph}
\end{centering}
\end{figure}

This graph incorporates many interesting paths. 
Any path starting from a minimal set leads only to other minimal sets. 
The relation between this graph and arithmetic progressions found within the fixed points of the substitution can be seen through the following observations together with Fact \ref{fact:i-th-column-power}, 
which does not require the substitution to be bijective. 
The graph helps narrow down the scope of differences $d$ for which suitable long arithmetic progressions may be found. 
For example, there is no arithmetic progression that includes both positions $4$ and $37$.
Converting to binary and following the path from the topmost level, we obtain 
the disjoint minimal sets $\{a,b\}$ and $\{f,d\}$; see the columns highlighted in blue in Figure \ref{fig:6 letter fixed points}.

For the next example, let us restrict ourselves to the minimal set $\{a,b\}$ to illustrate how the paths in the graph represent the columns. 
Note that for a difference $d=4=[1,0,0]$, the path indexed by  $d,2d,3d,\ldots$ always returns to the minimal set $\{a,b\}$. 
However, this condition does not guarantee a large arithmetic progression at these positions, as the corresponding letter might be $b$, instead, as the following picture illustrates via the columns highlighted in red. 

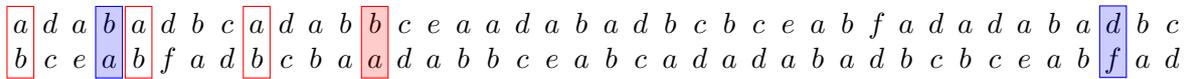
\begin{figure}[h!]
\centering
\begin{tikzpicture}
\matrix(A)
[matrix of math nodes,
nodes={minimum size=1em,inner sep=0pt},
column sep=0.1pt, row sep=1pt]
{
a&d&a&b&a&d&b&c&a&d&a&b&b&c&e&a&a&d&a&b&a&d&b&c&b&c&e&a&b&f&a&d&a&d&a&b&a&d&b&c\\
b&c&e&a&b&f&a&d&b&c&b&a&a&d&a&b&b&c&e&a&b&c&a&d&a&d&a&b&a&d&b&c&b&c&e&a&b&f&a&d\\
};
\node[draw=red, inner sep=0, fit={(A-1-1)(A-2-1)}, text height=27pt, text width=10pt]{};
\node[draw=blue, fill=blue, fill opacity=0.2, inner sep=0, fit={(A-1-4)(A-2-4)}, text height=27pt, text width=10pt]{};
\node[draw=red, inner sep=0, fit={(A-1-5)(A-2-5)}, text height=27pt, text width=10pt]{};
\node[draw=red, inner sep=0, fit={(A-1-9)(A-2-9)}, text height=27pt, text width=10pt]{};
\node[draw=red, fill=red, fill opacity=0.2, inner sep=0, fit={(A-1-13)(A-2-13)}, text height=27pt, text width=10pt]{};
\node[draw=blue, fill=blue, fill opacity=0.2, inner sep=0, fit={(A-1-38)(A-2-38)}, text height=27pt, text width=10pt]{};
\end{tikzpicture}
\caption{The fixed points of $\varrho$ starting from the seeds $a$ and $b$, with positions $0,4,8,12$ highlighted in red and places $4$ and $37$ highlighted in blue.}
\label{fig:6 letter fixed points}
\end{figure}

Since the substitution is primitive, the subgraph of minimal sets is strongly connected; in particular we can find a cycle starting from the set $\{a,b\}$ that visits every other minimal set. 
This path is indexed by the edges $1, \ 0, \ 1,\ 0$ and might be a natural candidate for difference of a long arithmetic progression.

It would be interesting to find out if one can use these graphs to bound $A(d)$ or gather more information on its behaviour. 
In particular, it is currently not certain whether there exist a sequence of differences for which $A(d)$ grows polynomially. 
Numerical data suggests that $A(d)$ grows polynomially along 
the subsequence $d_n=2^n+2$; see Figure~\ref{fig: A vs d plot, six-letter}. 

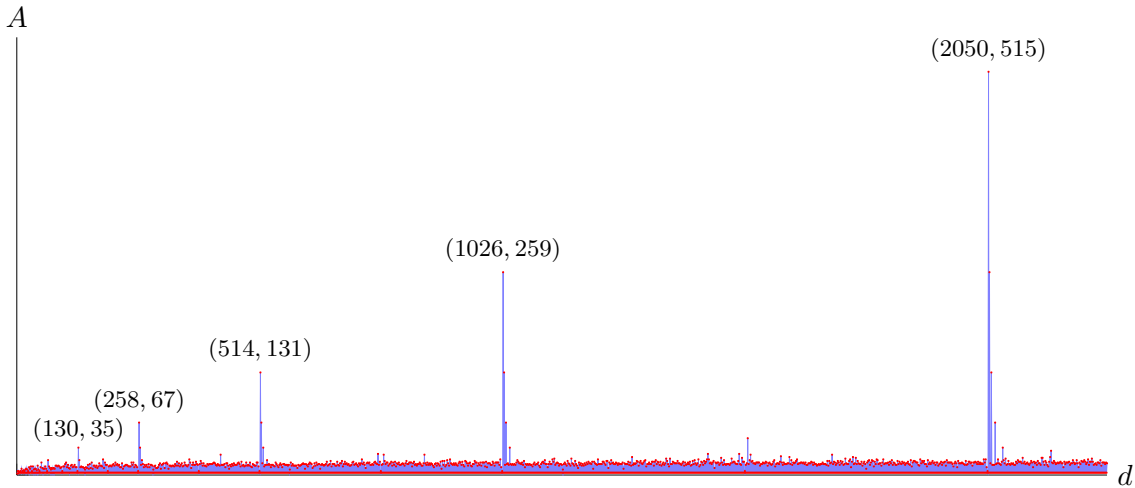
\begin{figure}[h!]
\centering
\begin{tikzpicture}
	\begin{axis}
	[
		width=\textwidth,
		height=0.3\textheight,
		line width=.3pt,
		axis lines*=left,
		xlabel style={at=(current axis.right of origin),anchor=west},
		xlabel=$d$,
		xmin=0,
		xmax=2300,
		xtick={0},
		xticklabels={},
		ylabel style={at=(current axis.above origin), anchor=south, rotate=-90},
		ylabel=$A$,
		ymin=0,
		ymax=559,
		ytick={0},
		yticklabels={}
	]
	\addplot
	[
		blue!50!white,
		line width=.2pt,
		mark options={scale=.1,fill=red,draw=red},
		mark=*
	]
	coordinates{(1,2)(2,5)(3,2)(4,5)(5,3)(6,4)(7,3)(8,5)(9,3)(10,8)(11,3)(12,5)(13,3)(14,6)(15,3)(16,5)(17,3)(18,7)(19,3)(20,8)(21,3)(22,7)(23,3)(24,5)(25,3)(26,10)(27,3)(28,8)(29,3)(30,11)(31,3)(32,5)(33,3)(34,11)(35,3)(36,7)(37,3)(38,8)(39,3)(40,8)(41,3)(42,6)(43,3)(44,12)(45,3)(46,7)(47,3)(48,5)(49,3)(50,10)(51,3)(52,16)(53,3)(54,9)(55,3)(56,8)(57,3)(58,8)(59,3)(60,11)(61,3)(62,8)(63,3)(64,5)(65,3)(66,19)(67,3)(68,11)(69,3)(70,9)(71,3)(72,7)(73,3)(74,9)(75,3)(76,8)(77,3)(78,9)(79,3)(80,8)(81,3)(82,12)(83,3)(84,12)(85,3)(86,11)(87,3)(88,12)(89,3)(90,11)(91,3)(92,9)(93,3)(94,10)(95,3)(96,5)(97,3)(98,11)(99,3)(100,11)(101,3)(102,14)(103,3)(104,16)(105,3)(106,10)(107,3)(108,13)(109,3)(110,9)(111,3)(112,8)(113,3)(114,12)(115,3)(116,13)(117,3)(118,10)(119,3)(120,11)(121,3)(122,13)(123,3)(124,10)(125,3)(126,10)(127,3)(128,5)(129,3)(130,35)(131,3)(132,19)(133,3)(134,10)(135,3)(136,11)(137,3)(138,10)(139,3)(140,10)(141,3)(142,9)(143,3)(144,7)(145,3)(146,10)(147,3)(148,9)(149,3)(150,13)(151,3)(152,8)(153,3)(154,14)(155,3)(156,9)(157,3)(158,16)(159,3)(160,8)(161,3)(162,12)(163,3)(164,12)(165,3)(166,10)(167,3)(168,12)(169,3)(170,9)(171,3)(172,11)(173,3)(174,15)(175,3)(176,12)(177,3)(178,12)(179,3)(180,11)(181,3)(182,20)(183,3)(184,9)(185,3)(186,16)(187,3)(188,13)(189,3)(190,13)(191,3)(192,5)(193,3)(194,11)(195,3)(196,12)(197,3)(198,12)(199,3)(200,11)(201,3)(202,13)(203,3)(204,14)(205,3)(206,12)(207,3)(208,16)(209,3)(210,13)(211,3)(212,14)(213,3)(214,12)(215,3)(216,13)(217,3)(218,10)(219,3)(220,10)(221,3)(222,15)(223,3)(224,8)(225,3)(226,11)(227,3)(228,12)(229,3)(230,13)(231,3)(232,13)(233,3)(234,12)(235,3)(236,11)(237,3)(238,11)(239,3)(240,11)(241,3)(242,16)(243,3)(244,13)(245,3)(246,16)(247,3)(248,10)(249,3)(250,13)(251,3)(252,12)(253,3)(254,10)(255,3)(256,5)(257,3)(258,67)(259,3)(260,35)(261,3)(262,13)(263,3)(264,19)(265,3)(266,14)(267,3)(268,10)(269,3)(270,13)(271,3)(272,11)(273,3)(274,13)(275,3)(276,10)(277,3)(278,14)(279,3)(280,10)(281,3)(282,13)(283,3)(284,13)(285,3)(286,12)(287,3)(288,7)(289,3)(290,11)(291,3)(292,12)(293,3)(294,9)(295,3)(296,9)(297,3)(298,10)(299,3)(300,13)(301,3)(302,15)(303,3)(304,8)(305,3)(306,13)(307,3)(308,14)(309,3)(310,15)(311,3)(312,9)(313,3)(314,15)(315,3)(316,16)(317,3)(318,14)(319,3)(320,8)(321,3)(322,12)(323,3)(324,12)(325,3)(326,11)(327,3)(328,12)(329,3)(330,12)(331,3)(332,11)(333,3)(334,14)(335,3)(336,12)(337,3)(338,17)(339,3)(340,12)(341,3)(342,15)(343,3)(344,11)(345,3)(346,13)(347,3)(348,15)(349,3)(350,11)(351,3)(352,12)(353,3)(354,16)(355,3)(356,13)(357,3)(358,12)(359,3)(360,11)(361,3)(362,13)(363,3)(364,20)(365,3)(366,16)(367,3)(368,9)(369,3)(370,13)(371,3)(372,16)(373,3)(374,13)(375,3)(376,13)(377,3)(378,17)(379,3)(380,14)(381,3)(382,13)(383,3)(384,5)(385,3)(386,14)(387,3)(388,18)(389,3)(390,15)(391,3)(392,12)(393,3)(394,15)(395,3)(396,12)(397,3)(398,12)(399,3)(400,11)(401,3)(402,14)(403,3)(404,13)(405,3)(406,14)(407,3)(408,14)(409,3)(410,14)(411,3)(412,14)(413,3)(414,15)(415,3)(416,16)(417,3)(418,15)(419,3)(420,13)(421,3)(422,12)(423,3)(424,14)(425,3)(426,13)(427,3)(428,12)(429,3)(430,26)(431,3)(432,13)(433,3)(434,13)(435,3)(436,13)(437,3)(438,14)(439,3)(440,10)(441,3)(442,15)(443,3)(444,15)(445,3)(446,12)(447,3)(448,8)(449,3)(450,13)(451,3)(452,11)(453,3)(454,14)(455,3)(456,12)(457,3)(458,15)(459,3)(460,13)(461,3)(462,13)(463,3)(464,13)(465,3)(466,15)(467,3)(468,16)(469,3)(470,14)(471,3)(472,11)(473,3)(474,17)(475,3)(476,12)(477,3)(478,13)(479,3)(480,11)(481,3)(482,12)(483,3)(484,16)(485,3)(486,12)(487,3)(488,13)(489,3)(490,12)(491,3)(492,16)(493,3)(494,11)(495,3)(496,10)(497,3)(498,14)(499,3)(500,13)(501,3)(502,12)(503,3)(504,12)(505,3)(506,16)(507,3)(508,16)(509,3)(510,11)(511,3)(512,5)(513,3)(514,131)(515,3)(516,67)(517,3)(518,15)(519,3)(520,35)(521,3)(522,11)(523,3)(524,13)(525,3)(526,13)(527,3)(528,19)(529,3)(530,13)(531,3)(532,14)(533,3)(534,13)(535,3)(536,10)(537,3)(538,13)(539,3)(540,13)(541,3)(542,12)(543,3)(544,11)(545,3)(546,16)(547,3)(548,14)(549,3)(550,12)(551,3)(552,10)(553,3)(554,13)(555,3)(556,15)(557,3)(558,13)(559,3)(560,10)(561,3)(562,14)(563,3)(564,14)(565,3)(566,13)(567,3)(568,13)(569,3)(570,13)(571,3)(572,12)(573,3)(574,13)(575,3)(576,7)(577,3)(578,16)(579,3)(580,12)(581,3)(582,14)(583,3)(584,12)(585,3)(586,12)(587,3)(588,12)(589,3)(590,12)(591,3)(592,9)(593,3)(594,13)(595,3)(596,12)(597,3)(598,12)(599,3)(600,13)(601,3)(602,14)(603,3)(604,15)(605,3)(606,14)(607,3)(608,8)(609,3)(610,14)(611,3)(612,13)(613,3)(614,11)(615,3)(616,14)(617,3)(618,16)(619,3)(620,15)(621,3)(622,14)(623,3)(624,9)(625,3)(626,11)(627,3)(628,15)(629,3)(630,12)(631,3)(632,16)(633,3)(634,13)(635,3)(636,14)(637,3)(638,14)(639,3)(640,8)(641,3)(642,13)(643,3)(644,14)(645,3)(646,13)(647,3)(648,12)(649,3)(650,13)(651,3)(652,14)(653,3)(654,17)(655,3)(656,12)(657,3)(658,13)(659,3)(660,14)(661,3)(662,18)(663,3)(664,11)(665,3)(666,14)(667,3)(668,14)(669,3)(670,14)(671,3)(672,12)(673,3)(674,12)(675,3)(676,17)(677,3)(678,13)(679,3)(680,12)(681,3)(682,15)(683,3)(684,15)(685,3)(686,18)(687,3)(688,11)(689,3)(690,14)(691,3)(692,17)(693,3)(694,14)(695,3)(696,15)(697,3)(698,13)(699,3)(700,14)(701,3)(702,12)(703,3)(704,12)(705,3)(706,13)(707,3)(708,16)(709,3)(710,17)(711,3)(712,13)(713,3)(714,14)(715,3)(716,15)(717,3)(718,13)(719,3)(720,11)(721,3)(722,15)(723,3)(724,13)(725,3)(726,15)(727,3)(728,20)(729,3)(730,14)(731,3)(732,16)(733,3)(734,13)(735,3)(736,9)(737,3)(738,15)(739,3)(740,16)(741,3)(742,15)(743,3)(744,16)(745,3)(746,13)(747,3)(748,13)(749,3)(750,14)(751,3)(752,13)(753,3)(754,16)(755,3)(756,17)(757,3)(758,15)(759,3)(760,14)(761,3)(762,27)(763,3)(764,13)(765,3)(766,17)(767,3)(768,5)(769,3)(770,13)(771,3)(772,14)(773,3)(774,26)(775,3)(776,18)(777,3)(778,14)(779,3)(780,16)(781,3)(782,16)(783,3)(784,12)(785,3)(786,14)(787,3)(788,15)(789,3)(790,17)(791,3)(792,12)(793,3)(794,16)(795,3)(796,13)(797,3)(798,17)(799,3)(800,11)(801,3)(802,15)(803,3)(804,14)(805,3)(806,16)(807,3)(808,13)(809,3)(810,13)(811,3)(812,14)(813,3)(814,15)(815,3)(816,14)(817,3)(818,12)(819,3)(820,15)(821,3)(822,18)(823,3)(824,14)(825,3)(826,14)(827,3)(828,16)(829,3)(830,16)(831,3)(832,16)(833,3)(834,11)(835,3)(836,15)(837,3)(838,17)(839,3)(840,13)(841,3)(842,14)(843,3)(844,12)(845,3)(846,14)(847,3)(848,14)(849,3)(850,16)(851,3)(852,14)(853,3)(854,12)(855,3)(856,12)(857,3)(858,15)(859,3)(860,26)(861,3)(862,15)(863,3)(864,13)(865,3)(866,16)(867,3)(868,13)(869,3)(870,15)(871,3)(872,13)(873,3)(874,15)(875,3)(876,14)(877,3)(878,12)(879,3)(880,10)(881,3)(882,14)(883,3)(884,15)(885,3)(886,12)(887,3)(888,15)(889,3)(890,14)(891,3)(892,12)(893,3)(894,17)(895,3)(896,8)(897,3)(898,18)(899,3)(900,13)(901,3)(902,18)(903,3)(904,11)(905,3)(906,13)(907,3)(908,19)(909,3)(910,15)(911,3)(912,12)(913,3)(914,20)(915,3)(916,15)(917,3)(918,14)(919,3)(920,13)(921,3)(922,15)(923,3)(924,13)(925,3)(926,15)(927,3)(928,13)(929,3)(930,12)(931,3)(932,15)(933,3)(934,17)(935,3)(936,16)(937,3)(938,13)(939,3)(940,17)(941,3)(942,13)(943,3)(944,11)(945,3)(946,18)(947,3)(948,17)(949,3)(950,15)(951,3)(952,12)(953,3)(954,15)(955,3)(956,16)(957,3)(958,12)(959,3)(960,11)(961,3)(962,14)(963,3)(964,14)(965,3)(966,18)(967,3)(968,16)(969,3)(970,16)(971,3)(972,12)(973,3)(974,15)(975,3)(976,13)(977,3)(978,16)(979,3)(980,12)(981,3)(982,16)(983,3)(984,16)(985,3)(986,16)(987,3)(988,11)(989,3)(990,17)(991,3)(992,10)(993,3)(994,15)(995,3)(996,15)(997,3)(998,15)(999,3)(1000,13)(1001,3)(1002,13)(1003,3)(1004,17)(1005,3)(1006,15)(1007,3)(1008,12)(1009,3)(1010,16)(1011,3)(1012,16)(1013,3)(1014,16)(1015,3)(1016,16)(1017,3)(1018,13)(1019,3)(1020,20)(1021,3)(1022,10)(1023,3)(1024,5)(1025,3)(1026,259)(1027,3)(1028,131)(1029,3)(1030,13)(1031,3)(1032,67)(1033,3)(1034,14)(1035,3)(1036,15)(1037,3)(1038,13)(1039,3)(1040,35)(1041,3)(1042,14)(1043,3)(1044,13)(1045,3)(1046,14)(1047,3)(1048,13)(1049,3)(1050,14)(1051,3)(1052,13)(1053,3)(1054,16)(1055,3)(1056,19)(1057,3)(1058,17)(1059,3)(1060,15)(1061,3)(1062,15)(1063,3)(1064,14)(1065,3)(1066,16)(1067,3)(1068,13)(1069,3)(1070,18)(1071,3)(1072,10)(1073,3)(1074,14)(1075,3)(1076,13)(1077,3)(1078,14)(1079,3)(1080,13)(1081,3)(1082,14)(1083,3)(1084,13)(1085,3)(1086,13)(1087,3)(1088,11)(1089,3)(1090,14)(1091,3)(1092,16)(1093,3)(1094,13)(1095,3)(1096,14)(1097,3)(1098,14)(1099,3)(1100,16)(1101,3)(1102,12)(1103,3)(1104,10)(1105,3)(1106,14)(1107,3)(1108,16)(1109,3)(1110,17)(1111,3)(1112,15)(1113,3)(1114,19)(1115,3)(1116,14)(1117,3)(1118,14)(1119,3)(1120,10)(1121,3)(1122,14)(1123,3)(1124,17)(1125,3)(1126,14)(1127,3)(1128,14)(1129,3)(1130,20)(1131,3)(1132,17)(1133,3)(1134,21)(1135,3)(1136,13)(1137,3)(1138,15)(1139,3)(1140,16)(1141,3)(1142,15)(1143,3)(1144,12)(1145,3)(1146,14)(1147,3)(1148,14)(1149,3)(1150,13)(1151,3)(1152,7)(1153,3)(1154,14)(1155,3)(1156,16)(1157,3)(1158,14)(1159,3)(1160,12)(1161,3)(1162,18)(1163,3)(1164,15)(1165,3)(1166,14)(1167,3)(1168,12)(1169,3)(1170,21)(1171,3)(1172,15)(1173,3)(1174,14)(1175,3)(1176,12)(1177,3)(1178,18)(1179,3)(1180,14)(1181,3)(1182,12)(1183,3)(1184,9)(1185,3)(1186,16)(1187,3)(1188,13)(1189,3)(1190,15)(1191,3)(1192,12)(1193,3)(1194,14)(1195,3)(1196,12)(1197,3)(1198,15)(1199,3)(1200,13)(1201,3)(1202,16)(1203,3)(1204,14)(1205,3)(1206,17)(1207,3)(1208,15)(1209,3)(1210,15)(1211,3)(1212,14)(1213,3)(1214,17)(1215,3)(1216,8)(1217,3)(1218,17)(1219,3)(1220,15)(1221,3)(1222,14)(1223,3)(1224,13)(1225,3)(1226,20)(1227,3)(1228,14)(1229,3)(1230,14)(1231,3)(1232,14)(1233,3)(1234,14)(1235,3)(1236,17)(1237,3)(1238,17)(1239,3)(1240,15)(1241,3)(1242,13)(1243,3)(1244,18)(1245,3)(1246,18)(1247,3)(1248,9)(1249,3)(1250,16)(1251,3)(1252,14)(1253,3)(1254,19)(1255,3)(1256,15)(1257,3)(1258,16)(1259,3)(1260,14)(1261,3)(1262,14)(1263,3)(1264,16)(1265,3)(1266,15)(1267,3)(1268,13)(1269,3)(1270,15)(1271,3)(1272,14)(1273,3)(1274,14)(1275,3)(1276,14)(1277,3)(1278,16)(1279,3)(1280,8)(1281,3)(1282,15)(1283,3)(1284,14)(1285,3)(1286,14)(1287,3)(1288,14)(1289,3)(1290,17)(1291,3)(1292,13)(1293,3)(1294,16)(1295,3)(1296,12)(1297,3)(1298,23)(1299,3)(1300,13)(1301,3)(1302,15)(1303,3)(1304,14)(1305,3)(1306,15)(1307,3)(1308,18)(1309,3)(1310,17)(1311,3)(1312,12)(1313,3)(1314,16)(1315,3)(1316,18)(1317,3)(1318,15)(1319,3)(1320,14)(1321,3)(1322,16)(1323,3)(1324,18)(1325,3)(1326,19)(1327,3)(1328,11)(1329,3)(1330,14)(1331,3)(1332,14)(1333,3)(1334,16)(1335,3)(1336,14)(1337,3)(1338,14)(1339,3)(1340,17)(1341,3)(1342,17)(1343,3)(1344,12)(1345,3)(1346,17)(1347,3)(1348,19)(1349,3)(1350,14)(1351,3)(1352,17)(1353,3)(1354,16)(1355,3)(1356,13)(1357,3)(1358,17)(1359,3)(1360,12)(1361,3)(1362,14)(1363,3)(1364,16)(1365,3)(1366,13)(1367,3)(1368,15)(1369,3)(1370,21)(1371,3)(1372,18)(1373,3)(1374,14)(1375,3)(1376,11)(1377,3)(1378,21)(1379,3)(1380,14)(1381,3)(1382,14)(1383,3)(1384,17)(1385,3)(1386,16)(1387,3)(1388,15)(1389,3)(1390,15)(1391,3)(1392,15)(1393,3)(1394,18)(1395,3)(1396,15)(1397,3)(1398,15)(1399,3)(1400,14)(1401,3)(1402,13)(1403,3)(1404,15)(1405,3)(1406,17)(1407,3)(1408,12)(1409,3)(1410,16)(1411,3)(1412,13)(1413,3)(1414,14)(1415,3)(1416,16)(1417,3)(1418,14)(1419,3)(1420,17)(1421,3)(1422,12)(1423,3)(1424,13)(1425,3)(1426,18)(1427,3)(1428,14)(1429,3)(1430,19)(1431,3)(1432,15)(1433,3)(1434,13)(1435,3)(1436,15)(1437,3)(1438,20)(1439,3)(1440,11)(1441,3)(1442,13)(1443,3)(1444,17)(1445,3)(1446,16)(1447,3)(1448,13)(1449,3)(1450,18)(1451,3)(1452,15)(1453,3)(1454,17)(1455,3)(1456,20)(1457,3)(1458,27)(1459,3)(1460,20)(1461,3)(1462,15)(1463,3)(1464,16)(1465,3)(1466,14)(1467,3)(1468,16)(1469,3)(1470,18)(1471,3)(1472,9)(1473,3)(1474,16)(1475,3)(1476,15)(1477,3)(1478,16)(1479,3)(1480,16)(1481,3)(1482,13)(1483,3)(1484,15)(1485,3)(1486,22)(1487,3)(1488,16)(1489,3)(1490,13)(1491,3)(1492,19)(1493,3)(1494,14)(1495,3)(1496,13)(1497,3)(1498,14)(1499,3)(1500,14)(1501,3)(1502,14)(1503,3)(1504,13)(1505,3)(1506,15)(1507,3)(1508,21)(1509,3)(1510,16)(1511,3)(1512,17)(1513,3)(1514,17)(1515,3)(1516,15)(1517,3)(1518,14)(1519,3)(1520,14)(1521,3)(1522,18)(1523,3)(1524,27)(1525,3)(1526,14)(1527,3)(1528,13)(1529,3)(1530,22)(1531,3)(1532,17)(1533,3)(1534,17)(1535,3)(1536,5)(1537,3)(1538,13)(1539,3)(1540,15)(1541,3)(1542,47)(1543,3)(1544,14)(1545,3)(1546,19)(1547,3)(1548,26)(1549,3)(1550,17)(1551,3)(1552,18)(1553,3)(1554,15)(1555,3)(1556,15)(1557,3)(1558,14)(1559,3)(1560,16)(1561,3)(1562,16)(1563,3)(1564,16)(1565,3)(1566,16)(1567,3)(1568,12)(1569,3)(1570,14)(1571,3)(1572,14)(1573,3)(1574,16)(1575,3)(1576,15)(1577,3)(1578,15)(1579,3)(1580,17)(1581,3)(1582,16)(1583,3)(1584,12)(1585,3)(1586,15)(1587,3)(1588,16)(1589,3)(1590,14)(1591,3)(1592,13)(1593,3)(1594,14)(1595,3)(1596,17)(1597,3)(1598,19)(1599,3)(1600,11)(1601,3)(1602,14)(1603,3)(1604,15)(1605,3)(1606,15)(1607,3)(1608,14)(1609,3)(1610,17)(1611,3)(1612,24)(1613,3)(1614,16)(1615,3)(1616,13)(1617,3)(1618,16)(1619,3)(1620,14)(1621,3)(1622,14)(1623,3)(1624,14)(1625,3)(1626,15)(1627,3)(1628,15)(1629,3)(1630,23)(1631,3)(1632,14)(1633,3)(1634,19)(1635,3)(1636,13)(1637,3)(1638,16)(1639,3)(1640,15)(1641,3)(1642,16)(1643,3)(1644,18)(1645,3)(1646,20)(1647,3)(1648,14)(1649,3)(1650,15)(1651,3)(1652,15)(1653,3)(1654,16)(1655,3)(1656,16)(1657,3)(1658,15)(1659,3)(1660,16)(1661,3)(1662,14)(1663,3)(1664,16)(1665,3)(1666,13)(1667,3)(1668,13)(1669,3)(1670,15)(1671,3)(1672,15)(1673,3)(1674,15)(1675,3)(1676,17)(1677,3)(1678,16)(1679,3)(1680,13)(1681,3)(1682,16)(1683,3)(1684,14)(1685,3)(1686,14)(1687,3)(1688,12)(1689,3)(1690,13)(1691,3)(1692,14)(1693,3)(1694,13)(1695,3)(1696,14)(1697,3)(1698,14)(1699,3)(1700,16)(1701,3)(1702,14)(1703,3)(1704,14)(1705,3)(1706,14)(1707,3)(1708,13)(1709,3)(1710,18)(1711,3)(1712,12)(1713,3)(1714,18)(1715,3)(1716,15)(1717,3)(1718,16)(1719,3)(1720,26)(1721,3)(1722,14)(1723,3)(1724,17)(1725,3)(1726,16)(1727,3)(1728,13)(1729,3)(1730,15)(1731,3)(1732,18)(1733,3)(1734,16)(1735,3)(1736,13)(1737,3)(1738,18)(1739,3)(1740,15)(1741,3)(1742,17)(1743,3)(1744,13)(1745,3)(1746,14)(1747,3)(1748,21)(1749,3)(1750,15)(1751,3)(1752,14)(1753,3)(1754,14)(1755,3)(1756,18)(1757,3)(1758,20)(1759,3)(1760,10)(1761,3)(1762,17)(1763,3)(1764,23)(1765,3)(1766,14)(1767,3)(1768,15)(1769,3)(1770,22)(1771,3)(1772,14)(1773,3)(1774,18)(1775,3)(1776,15)(1777,3)(1778,14)(1779,3)(1780,14)(1781,3)(1782,17)(1783,3)(1784,12)(1785,3)(1786,15)(1787,3)(1788,17)(1789,3)(1790,14)(1791,3)(1792,8)(1793,3)(1794,16)(1795,3)(1796,18)(1797,3)(1798,19)(1799,3)(1800,13)(1801,3)(1802,16)(1803,3)(1804,19)(1805,3)(1806,18)(1807,3)(1808,11)(1809,3)(1810,17)(1811,3)(1812,14)(1813,3)(1814,20)(1815,3)(1816,19)(1817,3)(1818,14)(1819,3)(1820,15)(1821,3)(1822,15)(1823,3)(1824,12)(1825,3)(1826,16)(1827,3)(1828,20)(1829,3)(1830,14)(1831,3)(1832,15)(1833,3)(1834,15)(1835,3)(1836,15)(1837,3)(1838,16)(1839,3)(1840,13)(1841,3)(1842,17)(1843,3)(1844,15)(1845,3)(1846,19)(1847,3)(1848,13)(1849,3)(1850,16)(1851,3)(1852,15)(1853,3)(1854,13)(1855,3)(1856,13)(1857,3)(1858,18)(1859,3)(1860,14)(1861,3)(1862,17)(1863,3)(1864,15)(1865,3)(1866,14)(1867,3)(1868,17)(1869,3)(1870,15)(1871,3)(1872,16)(1873,3)(1874,13)(1875,3)(1876,15)(1877,3)(1878,15)(1879,3)(1880,17)(1881,3)(1882,14)(1883,3)(1884,13)(1885,3)(1886,19)(1887,3)(1888,11)(1889,3)(1890,16)(1891,3)(1892,18)(1893,3)(1894,18)(1895,3)(1896,17)(1897,3)(1898,16)(1899,3)(1900,16)(1901,3)(1902,13)(1903,3)(1904,12)(1905,3)(1906,20)(1907,3)(1908,15)(1909,3)(1910,18)(1911,3)(1912,16)(1913,3)(1914,18)(1915,3)(1916,14)(1917,3)(1918,14)(1919,3)(1920,11)(1921,3)(1922,19)(1923,3)(1924,14)(1925,3)(1926,16)(1927,3)(1928,14)(1929,3)(1930,16)(1931,3)(1932,18)(1933,3)(1934,15)(1935,3)(1936,16)(1937,3)(1938,13)(1939,3)(1940,16)(1941,3)(1942,16)(1943,3)(1944,12)(1945,3)(1946,14)(1947,3)(1948,15)(1949,3)(1950,17)(1951,3)(1952,13)(1953,3)(1954,16)(1955,3)(1956,16)(1957,3)(1958,17)(1959,3)(1960,12)(1961,3)(1962,19)(1963,3)(1964,16)(1965,3)(1966,16)(1967,3)(1968,16)(1969,3)(1970,16)(1971,3)(1972,16)(1973,3)(1974,14)(1975,3)(1976,11)(1977,3)(1978,15)(1979,3)(1980,17)(1981,3)(1982,15)(1983,3)(1984,10)(1985,3)(1986,18)(1987,3)(1988,15)(1989,3)(1990,15)(1991,3)(1992,15)(1993,3)(1994,16)(1995,3)(1996,15)(1997,3)(1998,14)(1999,3)(2000,13)(2001,3)(2002,18)(2003,3)(2004,13)(2005,3)(2006,14)(2007,3)(2008,17)(2009,3)(2010,17)(2011,3)(2012,15)(2013,3)(2014,16)(2015,3)(2016,12)(2017,3)(2018,16)(2019,3)(2020,16)(2021,3)(2022,17)(2023,3)(2024,16)(2025,3)(2026,14)(2027,3)(2028,17)(2029,3)(2030,14)(2031,3)(2032,16)(2033,3)(2034,14)(2035,3)(2036,15)(2037,3)(2038,14)(2039,3)(2040,20)(2041,3)(2042,16)(2043,3)(2044,20)(2045,3)(2046,10)(2047,3)(2048,5)(2049,3)(2050,515)(2051,3)(2052,259)(2053,3)(2054,15)(2055,3)(2056,131)(2057,3)(2058,13)(2059,3)(2060,13)(2061,3)(2062,13)(2063,3)(2064,67)(2065,3)(2066,14)(2067,3)(2068,15)(2069,3)(2070,20)(2071,3)(2072,15)(2073,3)(2074,14)(2075,3)(2076,13)(2077,3)(2078,17)(2079,3)(2080,35)(2081,3)(2082,19)(2083,3)(2084,15)(2085,3)(2086,21)(2087,3)(2088,13)(2089,3)(2090,14)(2091,3)(2092,18)(2093,3)(2094,14)(2095,3)(2096,13)(2097,3)(2098,15)(2099,3)(2100,14)(2101,3)(2102,15)(2103,3)(2104,13)(2105,3)(2106,13)(2107,3)(2108,17)(2109,3)(2110,15)(2111,3)(2112,19)(2113,3)(2114,17)(2115,3)(2116,17)(2117,3)(2118,17)(2119,3)(2120,15)(2121,3)(2122,22)(2123,3)(2124,15)(2125,3)(2126,17)(2127,3)(2128,14)(2129,3)(2130,15)(2131,3)(2132,16)(2133,3)(2134,14)(2135,3)(2136,13)(2137,3)(2138,15)(2139,3)(2140,18)(2141,3)(2142,16)(2143,3)(2144,10)(2145,3)(2146,17)(2147,3)(2148,15)(2149,3)(2150,16)(2151,3)(2152,13)(2153,3)(2154,18)(2155,3)(2156,15)(2157,3)(2158,15)(2159,3)(2160,13)(2161,3)(2162,22)(2163,3)(2164,22)(2165,3)(2166,17)(2167,3)(2168,13)(2169,3)(2170,18)(2171,3)(2172,16)(2173,3)(2174,13)(2175,3)(2176,11)(2177,3)(2178,17)(2179,3)(2180,23)(2181,3)(2182,31)(2183,3)(2184,16)(2185,3)(2186,15)(2187,3)(2188,18)(2189,3)(2190,16)(2191,3)(2192,14)(2193,3)(2194,15)(2195,3)(2196,16)(2197,3)(2198,15)(2199,3)(2200,16)(2201,3)(2202,17)(2203,3)(2204,15)(2205,3)(2206,15)(2207,3)(2208,10)(2209,3)(2210,17)(2211,3)(2212,14)(2213,3)(2214,16)(2215,3)(2216,16)(2217,3)(2218,16)(2219,3)(2220,17)(2221,3)(2222,13)(2223,3)(2224,15)(2225,3)(2226,15)(2227,3)(2228,19)(2229,3)(2230,16)(2231,3)(2232,14)(2233,3)(2234,15)(2235,3)(2236,19)(2237,3)(2238,15)(2239,3)(2240,10)(2241,3)(2242,16)(2243,3)(2244,14)(2245,3)(2246,18)(2247,3)(2248,17)(2249,3)(2250,17)(2251,3)(2252,19)(2253,3)(2254,17)(2255,3)(2256,14)(2257,3)(2258,16)(2259,3)(2260,20)(2261,3)(2262,13)(2263,3)(2264,17)(2265,3)(2266,17)(2267,3)(2268,21)(2269,3)(2270,17)(2271,3)(2272,13)(2273,3)(2274,14)(2275,3)(2276,15)(2277,3)(2278,13)(2279,3)(2280,16)(2281,3)(2282,16)(2283,3)(2284,15)(2285,3)(2286,19)(2287,3)(2288,12)(2289,3)(2290,16)(2291,3)(2292,16)(2293,3)(2294,16)(2295,3)(2296,14)(2297,3)(2298,16)(2299,3)(2300,15)};
	\node[anchor=south, font=\footnotesize] at (axis cs: 130,30) {$(130,35)$};
	\node[anchor=south, font=\footnotesize] at (axis cs: 258,67) {$(258,67)$};
	\node[anchor=south, font=\footnotesize] at (axis cs: 514,131) {$(514,131)$};
	\node[anchor=south, font=\footnotesize] at (axis cs: 1026,259) {$(1026,259)$};
	\node[anchor=south, font=\footnotesize] at (axis cs: 2050,515) {$(2050,515)$};
	\node[anchor=south, font=\footnotesize] at (axis cs: 4098,1027) {$(4098,1027)$};
	\end{axis}
\end{tikzpicture}
\caption{$A(d)$ for $1 \leqslant d \leqslant 2300$, for the substitution in Eq~\eqref{eq:ex-six-letter}.}
\label{fig: A vs d plot, six-letter}
\end{figure}

\end{document}